\documentclass[11pt,reqno]{amsart}

\usepackage{fullpage}
\usepackage{mathrsfs}
\usepackage{amsfonts}
\usepackage{amsmath}
\usepackage{amsthm}
\usepackage{amssymb}
\usepackage{latexsym}
\usepackage[all]{xy}
\usepackage[xindy]{imakeidx}
\makeindex[program=xindy,title=Index of notation, columns=4]
\usepackage{palatino}
\usepackage[mathscr]{eucal}
\usepackage{xr-hyper}
\usepackage{hyperref}
\usepackage{multicol}
\usepackage[dvipsnames]{xcolor}

\theoremstyle{plain}
\newtheorem{thm}[equation]{Theorem}
\newtheorem{cor}[equation]{Corollary}
\newtheorem{prop}[equation]{Proposition}
\newtheorem{lem}[equation]{Lemma}
\newtheorem*{thm*}{Theorem}

\theoremstyle{definition}

\theoremstyle{remark}

\newtheorem{ex}[equation]{Example}
\newtheorem{rem}[equation]{Remark}

\makeatletter
\renewcommand{\subsection}{\@startsection{subsection}{2}{0pt}{-3ex
plus -1ex minus -0.2ex}{-2mm plus -0pt minus
-2pt}{\normalfont\bfseries}} \makeatother

\numberwithin{equation}{subsection}

\renewcommand{\o}{\otimes }
\newcommand{\hdot}{{\:\raisebox{2pt}{\text{\circle*{1.5}}}}}
\newcommand{\idot}{{\:\raisebox{2pt}{\text{\circle*{1.5}}}}}
\newcommand{\too}{\,\longrightarrow\,}
\newcommand{\mto}{\mapsto }
\newcommand{\iso}{{\;\stackrel{_\sim}{\to}\;}}

\newcommand{\onto}{\,\,\twoheadrightarrow\,\,}
\newcommand{\into}{\hookrightarrow }
\newcommand{\sset}{\subset}

\newcommand{\inv}{^{-1}}
\newcommand{\bplus}{\mbox{$\bigoplus$}}
\newcommand{\beq}{\begin{equation}\label}
\newcommand{\eeq}{\end{equation}}
\newcommand{\g}{{\mathfrak{g}}}
\renewcommand{\t}{{\mathfrak t}}
\renewcommand{\b}{{\mathfrak b}}
\renewcommand{\u}{{\mathfrak u}}
\renewcommand{\l}{{\mathfrak l}}
\newcommand{\fn}{{\mathfrak n}}
\renewcommand{\k}{\bbc}
\newcommand{\X}{{\mathscr X}}
\newcommand{\lla}{_\la }
\newcommand{\rhoshift}{}
\newcommand{\kalg}{\kappa^{\operatorname{alg}} }
\newcommand{\kgeom}{\kappa^{\operatorname{top}} }
\newcommand{\kbalg}{\overline{\kappa}^{\operatorname{alg}} }
\newcommand{\kbgeom}{\overline{\kappa}^{\operatorname{top}} }
\newcommand{\oo}{{\mathcal O}}
\newcommand{\Ga}{\Gamma }
\newcommand{\KK}{{\mathbf K}}
\newcommand{\OO}{{\mathbf O}}

\newcommand{\GGM}{{\mathbb G}_m}

\DeclareMathOperator{\sh}{{\mathrm{S}_\hbar}}
\DeclareMathOperator{\sym}{\mathrm{S}}
\DeclareMathOperator{\Ext}{\mathrm{Ext}}
\DeclareMathOperator{\gr}{\mathrm{gr}}
\DeclareMathOperator{\End}{\mathrm{End}}
\DeclareMathOperator{\ad}{\mathrm{ad}}
\DeclareMathOperator{\Lie}{{\mathrm{Lie}}}
\newcommand{\la}{\lambda}

\newcommand{\bla}{{\boldsymbol{\lambda}}}
\newcommand{\bmu}{{\boldsymbol{\mu}}}
\newcommand{\bnu}{{\boldsymbol{\nu}}}
\newcommand{\M}{{\mathbf{M}}}
\newcommand{\MM}{\overline{\mathbf{M}} }
\newcommand{\MMloc}{\overline{\mathbf{M}}_{\mathrm{loc}}}
\newcommand{\uh}{U_\hbar }
\newcommand{\uloc}{U_{\mathrm{loc}}}
\newcommand{\dd}{{\mathscr D}}
\renewcommand{\dh}{\dd_\hbar }
\newcommand{\zh}{{Z_\hb(\g)}}

\newcommand{\hb}{\hbar }
\newcommand{\kh}{{\bbc}[\hb]}

\newcommand{\bbc}{\C }

\newcommand{\C}{\mathbb{C}}

\newcommand{\fO}{\mathbf{O}}
\newcommand{\Z}{\mathbb{Z}}

\newcommand{\cO}{\mathcal{O}}

\newcommand{\Gm}{\mathbb{G}_m}
\newcommand{\bA}{\mathbb{A}}

\newcommand{\pt}{\mathrm{pt}}

\newcommand{\Gv}{{\check G}}
\newcommand{\Nv}{{\check N}}
\newcommand{\Tv}{{\check T}}
\newcommand{\Bv}{{\check B}}
\newcommand{\Lv}{{\check L}}
\newcommand{\Pv}{{\check P}}

\newcommand{\Uv}{{\check U}}
\newcommand{\GvO}{\Gv(\fO)}

\newcommand{\fg}{\mathfrak{g}}
\newcommand{\fb}{\mathfrak{b}}
\newcommand{\ft}{\mathfrak{t}}
\newcommand{\fl}{\mathfrak{l}}
\newcommand{\fp}{\mathfrak{p}}
\newcommand{\fk}{\mathfrak{k}}

\newcommand{\bX}{\mathbb{X}}

\newcommand{\Gr}{\mathsf{Gr}}

\newcommand{\wcN}{\widetilde{\mathcal{N}}}
\newcommand{\wfg}{\widetilde{\mathfrak{g}}}
\newcommand{\wfgr}{\widetilde{\mathfrak{g}}_{\mathrm{r}}}
\newcommand{\wfl}{\widetilde{\mathfrak{l}}}
\newcommand{\wflr}{\widetilde{\mathfrak{l}}_{\mathrm{r}}}
\newcommand{\wft}{\widetilde{\mathfrak{t}}}

\newcommand{\muv}{{\check \mu}}
\newcommand{\rhov}{{\check \rho}}
\newcommand{\alv}{{\check \alpha}}

\newcommand{\cDb}{\cD^{\mathrm{b}}}
\newcommand{\cDbc}{\cD^{\mathrm{b}}_{\mathrm{c}}}

\newcommand{\Perv}{\mathsf{Perv}}

\newcommand{\IC}{\mathrm{IC}}

\newcommand{\cF}{\mathcal{F}}
\newcommand{\cG}{\mathcal{G}}

\newcommand{\cM}{\mathcal{M}}

\newcommand{\cL}{\mathcal{L}}
\newcommand{\cR}{\mathcal{R}}

\newcommand{\cV}{\mathcal{V}}

\newcommand{\simto}{\xrightarrow{\sim}}

\newcommand{\lan}{\langle}
\newcommand{\ran}{\rangle}

\newcommand{\For}{\mathrm{For}}

\newcommand{\cD}{\mathcal{D}}

\newcommand{\Rep}{\mathsf{Rep}}

\newcommand{\fT}{\mathfrak{T}}

\DeclareMathOperator{\Hom}{Hom}

\DeclareMathOperator{\Ind}{\mathrm{Ind}}

\newcommand{\rs}{\mathrm{rs}}

\newcommand{\bv}{\mathbf{v}}
\newcommand{\bvv}{\overline{\mathbf{v}}}
\newcommand{\B}{\mathscr{B}}
\newcommand{\coH}{\mathsf{H}}
\newcommand{\w}{{}^w \hspace{-1pt}}
\newcommand{\y}{{}^y \hspace{-1pt}}
\newcommand{\tla}{{}^{(\la)} \hspace{-1pt}}

\newcommand{\FBK}{\mathsf{F}^{\operatorname{BK}} }
\newcommand{\Fgeom}{\mathsf{F}^{\operatorname{geom}} }
\newcommand{\qh}{\mathrm{Q}_\hbar}
\newcommand{\W}{\mathfrak{W}}
\newcommand{\V}{\mathsf{V}}
\newcommand{\uSp}{\underline{\mathsf{Sp}}}
\newcommand{\Sp}{\mathsf{Sp}}
\newcommand{\rest}{\mathbf{R}}
\newcommand{\restH}{\mathrm{R}}
\renewcommand{\sc}{\mathrm{sc}}
\newcommand{\sfF}{\mathsf{F}}
\newcommand{\ResGr}{\mathfrak{R}}
\renewcommand{\a}{\mathfrak{a}}
\newcommand{\dwalg}{\mathsf{DW}^{\operatorname{alg}} }
\newcommand{\dwgeom}{\mathsf{DW}^{\operatorname{geom}} }
\newcommand{\cS}{\mathbb{S}}

\newcommand*{\longhookrightarrow}{\ensuremath{\lhook\joinrel\relbar\joinrel\rightarrow}}
\newcommand{\llangle}{\langle \hspace{-1.5pt} \langle}
\newcommand{\rrangle}{\rangle \hspace{-1.5pt} \rangle}

\newcommand{\id}{\mathrm{id}}
\def\mon#1{\text{\rm $#1$-mon}}

\title{Differential operators on $G/U$ and the affine Grassmannian}

\author{Victor Ginzburg}
\address{
Department of Mathematics, University of Chicago,  Chicago, IL 
60637, USA.}
\email{ginzburg@math.uchicago.edu}

\author{Simon Riche}
\address{
Universit{\'e} Blaise Pascal - Clermont-Ferrand II, Laboratoire de Math{\'e}matiques, CNRS, UMR 6620, Campus universitaire des C{\'e}zeaux, F-63177 Aubi{\`e}re Cedex, France.
}
\email{simon.riche@math.univ-bpclermont.fr}

\thanks{The work of V.G.~was supported in part by the NSF grant DMS-1001677.
The work of S.R.~was supported by ANR Grants No.~ANR-09-JCJC-0102-01 and No.~ANR-2010-BLAN-110-02.}

\begin{document}
\begin{abstract}
We describe the equivariant cohomology of cofibers of spherical perverse sheaves on the affine Grassmannian of a reductive algebraic group in terms of the geometry of the Langlands dual group. In fact we give two equivalent descriptions: one in terms of $\dd$-modules of the basic affine space, and one in terms of intertwining operators for universal Verma modules. We also construct natural collections of isomorphisms parametrized by the Weyl group in these three contexts, and prove that they are compatible with our isomorphisms. As applications we reprove some results of the first author and of Braverman--Finkelberg.
\end{abstract}
\maketitle


\section{Introduction}
\label{sec:intro}

\subsection{}

The geometric Satake equivalence relates perverse sheaves (with complex coefficients in our case) on the affine Grassmannian $\Gr$ of a complex connected reductive algebraic group $\Gv$ and representations of the Langlands dual (complex) reductive group $G$. The underlying vector space of the representation $\mathbb{S}(\cF)$ attached to a perverse sheaf $\cF$ is given by its total cohomology $\coH^{\hdot}(\Gr,\cF)$. It turns out that various \emph{equivariant} cohomology groups attached to $\cF$ also carry information on the representation $\mathbb{S}(\cF)$, see e.g.~\cite{gi,yz,bf}. In this paper, if $\Tv$ is a maximal torus of $\Gv$, we describe, in terms of $G$, the equivariant cohomology of cofibers of $\cF$ at $\Tv$-fixed point, with respect to the action of $\Tv$ or of $\Tv \times \C^\times$, where $\C^\times$ acts on $\Gr$ by loop rotation.
In fact these groups can be described in two equivalent ways, either in terms of $\dd$-modules on the basic affine space or in terms of intertwining operators for universal Verma modules. We also describe the Weyl group action on this collection of spaces induced by the action of $N_{\Gv}(\Tv)$ on $\Gr$.

\subsection{}
\label{ss:intro-1}

To state our results more precisely, choose some Borel subgroup $\Bv \subset \Gv$ containing $\Tv$, and let $T,B$ be the maximal torus and the Borel subgroup of $G$ provided by the geometric Satake equivalence. Note that the Tannakian construction of $G$ also provides no zero vectors in each simple root subspace of $\g:=\Lie(G)$.
In this paper we study three families of graded modules over a polynomial algebra, attached to $G$ or $\Gv$, and endowed with symmetries parametrized by their common Weyl group $W$.

Let $\t:=\Lie(T)$ and $\sh:=\sym(\t)[\hb]$, considered as a graded
algebra where $\hb$ and the vectors in $\t$ are in degree $2$. (Here, $\sym(\t)$ is the symmetric algebra of the vector space $\t$.) Let also
$\bX:=X^*(T)$
be the character lattice. Let $\Rep(G)$ be the category of finite dimensional algebraic $G$-modules. 


Our first family of graded modules over $\sh$ is of ``geometric'' nature. Let $U$ be the unipotent radical of $B$, and let $\X:=G/U$ be the basic affine space. Consider the algebra $\dh(\X)$ of (global) asymptotic differential operators on $\X$, i.e.~the Rees algebra of the algebra $\Gamma(\X,\dd_\X)$ of differential operators on $\X$, endowed with the order filtration (see \S\ref{subsec2} for details). This algebra is naturally graded, and endowed with an action of $T$ induced by \emph{right} multiplication on $\X$. We denote by $\dh(\X)_\la$ the weight space associated with $\la \in \bX$. Then we set
\[
\cM^{\operatorname{geom}}_{V,\la} := \bigl( V \o \tla \dh(\X)_\la \bigr)^G.
\]
(Here the twist functor $\tla(\cdot)$ will be defined in \S\ref{subsec2}.)

Our second family of graded $\sh$-modules is of ``algebraic'' nature. Let $\uh(\g)$ be the asymptotic enveloping algebra of $\g$ (i.e.~the Rees algebra of the algebra $U(\g)$ endowed with the Poincar\'e-Birkhoff-Witt filtration, see~\S\ref{subsec1} for details). For $\la \in \bX$ we let $\M(\la)$ be the asymptotic universal Verma module associated with $\la$, a graded $(\sh,\uh(\g))$-bimodule whose precise definition is recalled in \S\ref{subsec1}. Then we set
\[
\cM^{\operatorname{alg}}_{V,\la} := \Hom_{(\sh,\uh(\g))} \bigl( \M(0), V \o \M(\la) \bigr)
\]
where we consider morphisms in the category of $(\sh,\uh(\g))$-bimodules. 

We will also construct a third family of graded modules, of
``topological'' nature,
which is associated with the ``Langlands dual data.''
Let ${\check \t}:=\Lie(\Tv)$. We have canonical identifications $\bX \cong X_*(\Tv)$ and $\sh \cong \sym({\check \t}^*)[\hb]$. Consider the category $\Perv_{\Gv(\OO)}(\Gr)$ of $\Gv(\OO)$-equivariant perverse sheaves on the affine Grassmannian $\Gr$ of $\Gv$. Then for any $\la \in \bX$ and $\cF$ in $\Perv_{\Gv(\OO)}(\Gr)$ we set
\[
\cM^{\operatorname{top}}_{\cF,\la} := \coH^{\hdot+\la(2\rhov)}_{\Tv \times \C^\times}(i_\la^! \cF).
\]
Here $i_\la$ is the inclusion of the point of $\Gr$ naturally associated with $\la$, $\C^\times$ acts on $\Gr$ by loop rotation, and $\rhov$ is the half sum of positive coroots of $G$. Then $\cM^{\operatorname{top}}_{\cF,\la}$ is in a natural way a graded $\sh$-module. 


\subsection{}
\label{ss:intro-conv}

Each of these families is endowed with a kind of ``symmetry'' governed by the Weyl group $W$ of $(G,T)$ or $(\Gv,\Tv)$. (Note that these Weyl groups 
can be canonically identified.)
Namely, we have isomorphisms of graded $\sh$-modules
\[
\mathsf{A}_{V,\la,w} : \cM_{V,\la} \simto \w \cM_{V,w\la} \qquad \text{or} \qquad \mathsf{A}_{\cF,\la,w} : \cM_{\cF,\la} \simto \w \cM_{\cF,w\la}
\] 
for all $w \in W$. (Here the twist functor $\w (\cdot )$ will be defined in \S\ref{subsec2}.)

In the ``geometric'' case, the isomorphisms $\mathsf{A}^{\operatorname{geom}}_{V,\la,w}$ are constructed using a $W$-action on $\dh(\X)$ given by partial Fourier transforms due to Gelfand--Graev and studied in particular by Bezrukav\-nikov--Braverman--Positselskii in \cite{bbp}. These operators depend on a choice of (non-zero) simple root vectors in $\g$, which we choose to be those provided by the geometric Satake equivalence.

In the ``algebraic'' setting, the isomorphisms $\mathsf{A}^{\operatorname{alg}}_{V,\la,w}$ are constructed using properties of intertwining operators between a Verma module and a tensor product of a $G$-module and a Verma module. Our 
constructions are ``renormalized'' variants of classical constructions appearing in the
definition of the dynamical Weyl group (see \cite{tv,ev}) but, as
  opposed to those considered in \emph{loc}.~\emph{cit.},
our  isomorphisms do not have poles.
Again, the operators  $\mathsf{A}^{\operatorname{alg}}_{V,\la,w}$ depend on a choice of simple root vectors in $\g$, which we choose as above.

In the ``topological'' setting, the isomorphisms $\mathsf{A}^{\operatorname{top}}_{\cF,\la,w}$ are induced by the action of $N_\Gv(\Tv)$ on $\Gr$ by left multiplication.

In each setting, the collection of operators 
is compatible with the product in $W$ in the sense that 
\[
\y \bigl( \mathsf{A}_{V,y\la,x} \bigr) \circ \mathsf{A}_{V,\la,y} = \mathsf{A}_{V,\la,xy} \qquad \text{or} \qquad \y \bigl( \mathsf{A}_{\cF,y\la,x} \bigr) \circ \mathsf{A}_{\cF,\la,y} = \mathsf{A}_{\cF,\la,xy}
\]
for any $\la,V,\cF$ as above and $x,y \in W$.

\subsection{}

In addition,
 these families of graded modules are endowed with morphisms 
\[
\mathsf{Conv}_{V,V',\la,\mu} : \cM_{V,\la} \o_{\sh} \tla \cM_{V',\mu} \to \cM_{V \o V',\la+\mu}, \quad \mathsf{Conv}_{\cF,\cF',\la,\mu} :
\cM_{\cF,\la} \o_{\sh} \tla \cM_{\cF',\mu} \to \cM_{\cF \star \cF',\la+\mu}
\]
related to the monoidal structure on the category $\Rep(G)$ (denoted $\o$) or $\Perv_{\Gv(\OO)}(\Gr)$ (denoted $\star$).

In the ``geometric'' setting, morphisms $\mathsf{Conv}^{\operatorname{geom}}_{V,V',\la,\mu}$ are induced by the product in the algebra $\dh(\X)$. In the ``algebraic'' setting, morphisms $\mathsf{Conv}^{\operatorname{alg}}_{V,V',\la,\mu}$ are induced by composition of morphisms of bimodules. In the ``topological'' setting, morphisms $\mathsf{Conv}^{\operatorname{top}}_{\cF,\cF',\la,\mu}$ are defined using a standard construction considered in particular in \cite{abg}.

\subsection{}
\label{ss:intro-main-thm}

Our main result might be stated as follows (see Corollary \ref{key_cor}, Theorem \ref{thm:W-symmetry}, and Proposition \ref{prop:convolution}).

\begin{thm*}
For $\cF$ in $\Perv_{\Gv(\OO)}(\Gr)$ and $\lambda \in \bX$ there exist canonical isomorphisms
\[
\cM^{\operatorname{top}}_{\cF,\la} \ \cong \ \cM^{\operatorname{geom}}_{\cS(\cF),\la} \ \cong \ \cM^{\operatorname{alg}}_{\cS(\cF),\la},
\]
where $\cS : \Perv_{\Gv(\OO)}(\Gr) \simto \Rep(G)$ is the geometric Satake equivalence. These families of isomorphisms are compatible with operators $\mathsf{A}$ and with morphisms $\mathsf{Conv}$.
\end{thm*}

The proof of this theorem is based on another crucial property of the modules $\cM^{\operatorname{geom}}_{V,\la}$, $\cM^{\operatorname{alg}}_{V,\la}$ and $\cM^{\operatorname{top}}_{\cF,\la}$: they are all compatible with restriction to a Levi subgroup in the appropriate sense. This property is used to reduce the proof of our claims to the case $G$ and $\Gv$ have semisimple rank one, in which case they can be checked by explicit computation. This strategy is rather classical in this context, see e.g.~\cite{bfm, bf, brf, ahr}.

\subsection{}

The present paper is closely related to,
and motivated by, results of \cite{abg} and \cite{bf}.
In fact, in a follow-up paper the results
of the present article will be used to
obtain a common generalization of
the equivalences of categories established in these papers. A similar generalization can also be
obtained using recent results of Dodd \cite{do}, but
our approach is different and, we believe,
more explicit. We will follow the strategy of \cite{abg}
and a key technical step in our approach is the
following algebra isomorphism, which is a "quantum"
analogue of \cite[Theorem 8.5.2]{abg} and which follows from the theorem stated in \S\ref{ss:intro-main-thm}:
\[
\bigoplus_{\lambda \in \bX^+}\  \Ext^{\hdot}_{{\Tv \times\C^\times }}(\cR_G,\mathcal{W}^{\lambda} \star \cR_G) \  \cong \
U_\hb(\g)\ltimes\left(\bigoplus_{\lambda \in \bX^+}\
\dh(\X)_\lambda\right).
\]
Here,  $\mathcal{W}^{\lambda}$ is the {\em Wakimoto sheaf} associated with $\la$, $\cR_G$ is an ind-perverse sheaf on $\Gr$ corresponding to the regular representation of $G$,
and we refer to \cite[\S8]{abg} for this and other unexplained notation.

\subsection{}

We will also consider ``classical analogues'' of the above
constructions, by which we mean specializing $\hb$ to $0$, hence replacing $\sh$ by $\sym(\t)$ or $\sym({\check \t}^*)$. The classical analogues of $\cM^{\operatorname{top}}$ are easy to define: we simply set
\[
\overline{\cM}^{\operatorname{top}}_{\cF,\la} := \coH^{\hdot+\la(2\rhov)}_{\Tv}(i_\la^! \cF).
\]
We also have morphisms $\overline{\mathsf{A}}^{\operatorname{top}}$ and $\overline{\mathsf{Conv}}^{\operatorname{top}}$ given by the same constructions as for $\mathsf{A}^{\operatorname{top}}$ and $\mathsf{Conv}^{\operatorname{top}}$.

There is no interesting classical analogue of $\cM^{\operatorname{alg}}$. The classical analogues of  $\cM^{\operatorname{geom}}$ are defined using the geometry of the Grothendieck--Springer resolution $\wfg$. More precisely we set
\[
\overline{\cM}^{\operatorname{geom}}_{V,\la} := \bigl( V \o \Gamma(\wfg,\oo_{\wfg}(\la)) \bigr)^G
\]
where $\oo_{\wfg}(\la)$ is the $G$-equivariant line bundle on $\wfg$ associated with $\la$. The operators $\overline{\mathsf{Conv}}^{\operatorname{geom}}$ are induced by the natural morphisms 
\[
\Gamma(\wfg,\oo_{\wfg}(\la)) \o \Gamma(\wfg,\oo_{\wfg}(\mu)) \to \Gamma(\wfg,\oo_{\wfg}(\la+\mu)).
\]
Finally, the operators $\overline{\mathsf{A}}^{\operatorname{geom}}$ are defined using the $W$-action on the regular part of $\wfg$. Again, these operators depend on a choice of simple root vectors in $\g$. This construction seems to be new, and has interesting consequences (see \S\ref{ss:action-T*X}).

Then we prove the following (see Corollary \ref{key_cor-classical}, Theorem \ref{thm:W-symmetry-classical} and Remark \ref{rk:convolution}(2)).

\begin{thm*}
For $\cF$ in $\Perv_{\Gv(\OO)}(\Gr)$ and $\lambda \in \bX$, there exist canonical isomorphisms
\[
\overline{\cM}^{\operatorname{top}}_{\cF,\la} \ \cong \ \overline{\cM}^{\operatorname{geom}}_{\cS(\cF),\la}.
\]
This family of isomorphisms is compatible with operators $\overline{\mathsf{A}}$ and with morphisms $\overline{\mathsf{Conv}}$.
\end{thm*}

The modules appearing in the theorem (and the corresponding morphisms) are related to those appearing in the theorem of \S\ref{ss:intro-main-thm} by the functor $\C \otimes_{\C[\hb]} (-)$ (where $\hb$ acts by zero on $\C$). For $\cM^{\operatorname{top}}$ and $\overline{\cM}^{\operatorname{top}}$, this easily follows from the parity vanishing of $\coH^{\hdot}(i_{\la}^! \cF)$, see Lemma \ref{lem:forget-hbar}. For $\cM^{\operatorname{geom}}$ and $\overline{\cM}^{\operatorname{geom}}$, this requires a more subtle argument, see \S\ref{ss:Fourier-classical}. In particular, our results establish a relation between the automorphisms of $\dh(\X)$ induced by partial Fourier transforms and the $W$-action on the regular part of $\wfg$, which seems to be new.

\subsection{}

As applications of our constructions we give new proofs of two results: a geometric description of the Brylinski--Kostant filtration due to the first author (see \cite{gi}), and a geometric construction of the dynamical Weyl group due to Braverman--Finkelberg (see \cite{brf}). We also observe that some of our technical preliminary results have interesting applications: they allow to give simpler proofs of results on the structure of the algebra $\dd(\X)$ of differential operators on $\X$ (see \S\ref{ss:complements-DX}) and to construct an action of $W$ on the regular part of $T^*\X$ which ``lifts'' the action on the regular part of $\wfg$, see \S\ref{ss:action-T*X}.

One important tool in the first proof of the geometric Satake equivalence in \cite{gi} was specialized equivariant cohomology of cofibers (see in particular [\emph{loc.}~\emph{cit.}, \S 3.5]), while in \cite{mv} the authors replaced this tool by cohomology of corestrictions to semi-infinite orbits $\fT_\la$. Our descriptions of $\coH^{\hdot}_{\Tv}(i_\la^! \cF)$, $\coH^{\hdot}_{\Tv}(t_\la^! \cF)$ (where $t_\la$ denotes the inclusion of $\fT_\la$) and the natural morphism between them (see Theorem \ref{thm:equiv-coh-classical}) shed some light on the precise relation between these points of view.


\subsection{Description of the paper}

In Section \ref{sec:statement} we define our main players, and state our main results. In Section \ref{sec:DX} we study the modules $\cM^{\operatorname{geom}}_{V,\la}$ and define their symmetries. In Section \ref{sec:morphisms} we study the modules $\cM^{\operatorname{alg}}_{V,\la}$, define their symmetries, and relate this algebraic family with the geometric one. In Section \ref{sec:wfg} we study the modules $\overline{\cM}^{\operatorname{geom}}_{V,\la}$, define their symmetries, and relate them with the modules $\cM^{\operatorname{geom}}_{V,\la}$. In Section \ref{sec:Satake} we recall the construction of the geometric Satake equivalence and its main properties. In Section \ref{sec:proofs} we prove our main results. In Section \ref{sec:applications} we give some complements and applications of these results. Finally, the paper finishes with two appendices: Appendix \ref{sec:rank1} collects computations in semi-simple rank one that are needed in our proofs, and Appendix \ref{sec:appendix-Fourier} is a reminder on partial Fourier transforms for (asymptotic) $\dd$-modules.

\subsection{Conventions}

Throughout, we will work
over the ground field  $\bbc$
of complex numbers and write $\o=\o_\bbc$.
If $M=\bigoplus_{n \in \Z} M_n$ is a graded vector space and $m \in \Z$, we define the graded vector space $M\langle m \rangle$\index{1@$\langle m \rangle$} by the following rule: $(M \langle m \rangle)_n=M_{n-m}$. Note that $\langle 1 \rangle$ is a ``homological'' shift, i.e.~it shifts graded vector spaces to the \emph{right}.
We will always consider $\C[\hb]$ as a graded algebra where $\hb$ has degree $2$. If $A$ and $B$ are $\C[\hb]$-algebras, by an $(A,B)$-bimodule we mean an $(A \o_{\C[\hb]} B^{\mathrm{op}})$-module.
If $A$ is an algebra, we write $\Hom_{-A}(-,-)$ for $\Hom_{A^{\mathrm{op}}}(-,-)$.

\subsection{Acknowledgements}

We thank Pramod Achar for useful discussions at early stages of this work, and Pierre Baumann. Part of this work was completed while the second author visited University of Chicago.

\section{Statement of the main results}
\label{sec:statement}

\subsection{Asymptotic Verma modules}\label{subsec1}

Given a filtered $\C$-algebra $A=\bigcup_{i \in \Z_{\geq 0}} F_i A$, we let $A_\hb$
be the Rees algebra (sometimes referred to as
``graded'' 
or ``asymptotic'' version)
of the filtered algebra $A$. It can be defined as the following subalgebra of $A[\hb]$:
\beq{eqn:def-rees}
A_\hb:= \bigoplus_{i \in \Z} A_\hb^i \quad \text{with} \quad A_\hb^i = \begin{cases}
0 & \text{if $i$ is odd;} \\
\hb^{i} \cdot F_{i/2} A & \text{if $i$ is even}.
\end{cases}
\eeq
Thus,  $A_\hb$
is a graded $\kh$-algebra,
where the indeterminate $\hb$ has grade degree 2. (The reason for our convention will become clear later.) Moreover, one has a natural isomorphism
\[
A_\hb / \hb \cdot A_\hb \cong \mathrm{gr}^F A
\]
(where degrees are doubled on the left-hand side).

If ${\mathfrak k}$ is a Lie algebra,
the  enveloping algebra $U(\fk)$ comes equipped with
a natural ascending filtration,
the Poincar\'e-Birkhoff-Witt filtration, such that $\gr U(\fk) = \sym(\fk)$.
The corresponding asymptotic enveloping
algebra $\uh(\fk):=U(\fk)_\hb$\index{Uh@$\uh(\fk)$} has an alternative
(equivalent) definition as the $\C[\hbar]$-algebra generated by $\fk$, with
relations $xy-yx = \hbar [x,y]$ for $x,y \in \fk$. Here elements of $\fk$ have degree $2$. We will use this description of $\uh(\fk)$, and still denote by $x$ the image of an element $x \in \fk$. (If we were using the description \eqref{eqn:def-rees}, this element should rather be denoted $\hb x$.)

Let $G$ be a connected reductive group over $\k$ with Lie algebra $\g$.
We fix a triangular decomposition
$\g=\u\oplus\t\oplus\u^-$, so $\b=\t\oplus\u$ is a Borel subalgebra.
Let $T$ be the maximal torus and $B=T\cdot U$ the Borel subgroup
corresponding to the Lie algebras $\t$ and $\b$, respectively. We will denote by $R$ the set of roots of $G$ (relative to $T$), by $R^+$ the positive roots (i.e.~the roots of $\u$), and by $W$ the Weyl group of $(G,T)$.
Let
$\rho\in\t^*$, resp. $\check\rho\in\t$, be the half sum of positive
roots,
resp. coroots. We also let $\bX$\index{X@$\bX$} be the lattice of characters of $T$, and $\bX^+$, resp.~$\bX^-$, be the sub-semigroup of dominant, resp.~antidominant, weights. We will frequently consider elements of $\ft^*$ (resp.~of $\bX$) as linear forms on $\fb$ (resp.~characters of $B$) which are trivial on $\u$ (resp.~on $U$). Also, as usual, when convenient we identify $\bX$ with a subset of $\t^*$ via the differential.

We consider asymptotic $\kh$-algebras
$\uh(\g)$ and $\uh(\t)$.
The latter algebra is a commutative graded algebra which is clearly isomorphic
to $\sh:=\sym(\t)[\hb]$\index{sh@$\sh$} where $\sym(\t)$, the symmetric algebra
of $\t$, is equipped with its natural grading.
Let $Z(\g)$ be the center of the algebra $U(\g)$.
The Poincar\'e-Birkhoff-Witt filtration on
$U(\g)$ induces, by restriction, a filtration
on $Z(\g)$. The corresponding asymptotic algebra
$\zh$\index{Zh@$\zh$} is the center of the algebra $\uh(\g)$.
%
One has a Harish-Chandra isomorphism
$\zh\cong\sym_\hb^W=\sym(\t)^W[\hb]$, a graded $\kh$-algebra isomorphism induced by the composition
\[
\zh \hookrightarrow \uh(\g) \twoheadrightarrow \uh(\g) / (\u \cdot \uh(\g) + \uh(\g) \cdot \u^-) \xleftarrow{\sim} \sh \simto \sh
\]
where the isomorphism on the right-hand side sends $t \in \t$ to $t+\hb
\rho(t)$.
Using this isomorphism, we may (and will)
identify the algebra $\zh$ with
a subalgebra of $\sh$. 

 For any $\la\in\t^*$, let
${\sh} \llangle \lambda \rrangle$\index{shlambda@${\sh} \llangle \lambda \rrangle$} 
be the $\uh(\fb)$-bimodule
defined as follows. As a $\C[\hbar]$-module, it is
isomorphic to $\sh$. The left $\uh(\fb)$-module structure is given by the natural isomorphism $\sh \cong \uh(\fb) / \uh(\fb) \cdot \fn$. Then in the right $\uh(\fb)$-module structure, the Lie ideal $\u \subset \fb$ acts by
$0$, and $t \in \ft$ acts by multiplication by $t + \hbar
\lambda(t)$.

We define a graded $(\uh(\b),\,\uh(\g))$-bimodule,
a certain  asymptotic version of the universal  Verma module,
as follows:\index{Mlambda@$\M(\lambda)$}%
\[
\M(\la)=\sh \llangle\la+\rho\rrangle\o_{\uh(\b)}\uh(\g).
\]
(We will mainly only consider $\M(\la)$ as an $(\sh,\uh(\g))$-bimodule.)
The action of $\uh(\g)$ on the vector $\bv_\la := 1 \otimes 1 \in \M(\la)$\index{vlambda@$\bv_\la$}
induces an isomorphism of right $\uh(\g)$-modules
\begin{equation}
\label{eqn:isom-Verma}
\M(\lambda) \ \cong \ \uh(\g) / (\u \cdot \uh(\g)).
\end{equation}
Under this isomorphism, the right ${\sh}$-module
structure is such that the action of $t \in \ft$ is induced by right
multiplication by $t - \hbar\cdot (\lambda+\rho)(t)$ on
$\uh(\g)$.

It is immediate from definitions that
 there is  well defined `adjoint' action $b:\ m\mto\ad b(m)$
of the Lie algebra $\b$ on $\M(\la)$, which is related to the
bimodule structure by the equation 
\[
\hb\cdot\ad b(m) = (b+\hb \rho(b)) \cdot m - m \cdot b \qquad \forall b\in\b.
\]
The adjoint action of the subalgebra $\t\sset\b$
is  semisimple.
Therefore, one has a weight decomposition
$\M(\la)=\bplus_{\mu\in\t^*}\ \M(\la)_\mu$.
In particular, we have $\M(\la)_{-\la} = \C[\hb] \cdot \bv_\lambda$, and $\M(\la)_\mu=0$ unless $\mu \in -\la - \Z_{\geq 0} R^+$.
For $\la\in\bX$, the adjoint action on $\M(\la)$ can be exponentiated to
an algebraic $B$-action.

Let $\Rep(G)$\index{RepG@$\Rep(G)$} be the tensor category of
finite dimensional rational $G$-modules.
For $V\in\Rep(G)$ and $\la\in\bX$,
 let $V_\la$ denote the $T$-weight space
of $V$ of weight ~$\la$.

The assignment $\hb \mapsto 1 \o \hb$, $x\mto -x \otimes \hb + 1 \otimes x$
has a unique extension to an algebra homomorphism
$\uh(\g)\to U(\g)^{\mathrm{op}}\o\uh(\g)$. 
Via this homomorphism,
 for any right $\uh(\g)$-module $M$ and $V\in\Rep(G)$,
the vector space $V\o M$ acquires
the structure of a right $\uh(\g)$-module.
This gives an  $(\uh(\fb),
U_{\hbar}(\fg))$-bimodule structure on
$V\o \M(\la)$, where the left action of the algebra $\uh(\fb)$ on
$V\o \M(\la)$ comes from its action on $\M(\la)$ on the left. 
If $\la\in\bX$, the differential of the diagonal
$B$-action on $V \otimes \M(\la)$ and the bimodule structure are related as follows:
if $b \in \fb \subset U(\fb)$, $\hb$ times the action of $b$ is given by the assignment $n \mapsto (b + \hb \rho(b)) \cdot n - n \cdot b$.

One has a natural morphism of left $\sh$-modules $p_{\lambda} :
\M(\lambda) \to {\sh}$,\index{plambda@$p_{\lambda}$} induced by the projection $U_{\hbar}(\fg) \to
U_{\hbar}(\fb)$ orthogonal to $U_{\hbar}(\fg) \cdot \u^- $. If $\la \in \bX$, then $p_\lambda$ is also a morphism of $T$-modules $\M(\la) \to \sh \otimes \C_{-\la}$. 
An important role below will be played by
the morphism of graded ${\sh}$-modules $\kalg_{V,\la}$\index{kalg@$\kalg_{V,\la}$}
 defined (for $V$ in $\Rep(G)$ and $\la \in \bX$) as the following composition:
\beq{eqn:equivariant-cohomology-3}
\bigl(V \otimes \M(\la)\rhoshift  \bigr)^B
\hookrightarrow \xymatrix{\bigl(V \otimes \M(\la)\rhoshift 
\bigr)^T\ 
\ar[rr]^<>(0.5){\id_V\o p_{\lambda}}&&\ (V \o 
{\sh} \o  \bbc_{-\lambda})^T} =
V_{\lambda} \o  {\sh}.
\eeq
Note that $\bigl(V \otimes \M(\la)\rhoshift  \bigr)^B$ has a natural structure of $\zh$-module, induced by the (right) action of $\zh \subset \uh(\fg)$ on $\M(\la)$. With this definition, $\kalg_{V,\la}$ is also a morphism of $\zh$-modules, where $\zh=\sym_\hb^W$ acts on $\sh$ via the restriction of the right action of $\sh$ on $\sh \llangle \la \rrangle$.

%
%
%
%

\subsection{The affine Grassmannian: equivariant cohomology of cofibers}
\label{ss:Gr-intro}

Write $\GGM$ for the multiplicative group.
Let $\Gv$ be the Langlands dual group of $G$. The
group $\Gv$ comes equipped
with  the maximal torus
 $\Tv\sset \Gv$, with opposite Borel subgroups $\Bv=\Tv \cdot \Uv$ and $\Bv^-=\Tv \cdot \Uv^-$, and with a canonical isomorphism
$\bX=\Hom(\Gm ,\Tv)$, the cocharacter lattice of $\Tv$.
(To be completely precise, one should first choose $\Gv,\Bv,\Tv$, and then use the affine Grassmannian of $\Gv$ to define $G,B,T$ by Tannakian formalism; see \S\ref{ss:Satake} for details.)

Let $\KK=\C( \hspace{-1.5pt} ( z) \hspace{-1.5pt} )$,\index{K@$\KK$} resp. $\OO=\C [ \hspace{-1.5pt} [ z] \hspace{-1.5pt} ]$.\index{O@$\OO$}
Let $\Gr_{\Gv}:=\Gv(\KK)/\Gv(\OO)$,
resp. $\Gr_{\Tv}:=\Tv(\KK)/\Tv(\OO)$, be the affine Grassmannian 
associated with the group $\Gv$, resp. ${\Tv}$. (We will consider the \emph{reduced} ind-scheme structure on these affine Grassmannians.)
Thus, one has $\bX = \Gr_{\Tv}$ and 
there is a natural embedding $\bX = \Gr_{\Tv} \hookrightarrow
\Gr_{\Gv}$. For $\lambda \in \bX$, we let $\bla$ be the image of
$\lambda$ and let  $i_{\lambda} : \{\bla\} \hookrightarrow
\Gr_{\Gv}$ denote the one point embedding.
The group $\Gv(\KK)\rtimes\Gm $ acts on $\Gr_{\Gv} $ on the left,
where the factor $\Gm $ acts by rotation of the loop.

For the rest of this section, we will use simplified notation
$\Gr:=\Gr_{\Gv}$.
The following subsets of the affine Grassmannian will
play an important role. For $\la\in\bX^+,$  we let $\Gr^\la:=\Gv(\OO)\cdot\bla$.\index{Grlambda@$\Gr^\la$}
This is a finite dimensional $(\Gv(\OO) \rtimes \Gm)$-stable locally closed subvariety
of $\Gr $. One has a stratification
$\Gr =\sqcup_{\la\in\bX^+}\ \Gr^\la$.
Further, for any $\lambda \in \bX$,  following Mirkovi{\'c}-Vilonen 
one puts  $\fT_{\lambda} := \Uv^-(\KK)\cdot \bla$.\index{Tlambda@$\fT_{\la}$} We let $t_\la : \fT_\la \hookrightarrow \Gr$\index{tlambda@$t_\la$} be the inclusion.

Let $A:=\Tv\times\Gm$,\index{A@$A$} a toral subgroup of $\Gv(\KK)\rtimes\Gm $.
The  Mirkovi\'c-Vilonen space
$\fT_{\lambda}$ is  $A$-stable.
Further, the set $\bX\sset \Gr $ is known to be equal
to the set of $A$-fixed points in  $\Gr $.
Therefore, for any object
${\mathcal F}$ of the equivariant derived category $\cDb_A(\Gr)$, there are well defined
 $A$-equivariant cohomology groups
$\coH^{\hdot}_{A}(\fT_{\lambda},
t_{\lambda}^! {\mathcal F})$, resp. 
$\coH^{\hdot}_{A}(i_{\lambda}^!{\mathcal F})$.
These are graded modules over the graded
algebra $\coH^\hdot_A(\pt)\cong \sym(\t)[\hb]=\sh$.

Let $\Perv_{\Gv(\OO)}(\Gr)$, resp.~$\Perv_{\Gv(\OO) \rtimes \Gm}(\Gr)$, be the category of $\Gv(\OO)$-equivariant, resp.~$\Gv(\OO) \rtimes \Gm$-equivariant, perverse sheaves on $\Gr$. Let also $\Perv_{\mon{\Gv(\OO)}}(\Gr)$ be the category of perverse sheaves on $\Gr$ which are constructible with respect to the stratification by $\Gv(\OO)$-orbits. Recall that all three of these categories are semisimple, with simple objects parametrized by $\bX^+$. In particular, the forgetful functors
\[
\Perv_{\Gv(\OO) \rtimes \Gm}(\Gr) \to \Perv_{\Gv(\OO)}(\Gr) \to \Perv_{\mon{\Gv(\OO)}}(\Gr)
\]
are equivalences of categories (see \cite[Appendix A]{mv} for a similar result in a much more general situation). Let\index{S@$\cS$}%
\[
\cS : \Perv_{\Gv(\OO)}(\Gr) \simto \Rep(G)
\]
be the  geometric Satake equivalence. By the remark above, any 
object of $\Perv_{\Gv(\OO)}(\Gr)$
can be considered naturally as an object of $\cDb_A(\Gr)$.

The following
lemma is a simple consequence of results of Kazhdan--Lusztig \cite{kl} and
 Mirko\-vi\'c--Vilonen \cite{mv},
cf.~also \cite{yz, brf}. It will be proved in \S\ref{ss:equiv-cohomology}.

\begin{lem}
\label{lem:equiv-cohomology-first-properties}
For any $\cF$ in $\Perv_{\Gv(\OO)}(\Gr_\Gv)$ and $\lambda \in \bX$, one has
\begin{enumerate}
\item The graded $\sh$-module 
$\coH^{\hdot}_{A}(i_{\lambda}^! \cF)$, resp.~the graded $\sym(\t)$-module 
$\coH^{\hdot}_{\Tv}(i_{\lambda}^! \cF)$, is free.
\item There is a canonical isomorphism of graded $\sh$-modules, resp.~of graded $\sym(\t)$-modules
\beq{zeta}
\coH^{\hdot}_{A}(\fT_{\lambda}, t_{\lambda}^! \cF) \ \cong \ 
\bigl( \cS(\cF) \bigr)_{\lambda}
 \o \sh  \langle \lambda(2\rhov) \rangle, \qquad \coH^{\hdot}_{\Tv}(\fT_{\lambda}, t_{\lambda}^! \cF) \ \cong \ 
\bigl( \cS(\cF) \bigr)_{\lambda}
 \o \sym(\t)  \langle \lambda(2\rhov) \rangle.
\eeq
\end{enumerate}
\end{lem}

One may 
 factor the embedding $i_\lambda : \{\bla\} \hookrightarrow
\Gr$\index{ilambda@$i_\lambda$} as a composition
$
\{\bla\} \xrightarrow{\imath_\lambda}  \fT_{\lambda} \xrightarrow{t_\la}
\Gr$.\index{ilambda2@$\imath_\lambda$}
Hence, there is a push-forward morphism
\[
(\imath_\la)_!:\
\coH^{\hdot}_{A}(i_{\lambda}^! \cF) \ =\
\coH^{\hdot}_{A}(\imath_{\lambda}^!t_\la^! \cF) \ 
 \too \ 
\coH^{\hdot}_{A}(\fT_{\lambda},
 t_{\lambda}^! \cF).
\]
Let $\kgeom_{\cF,\la}$\index{ktop@$\kgeom_{\cF,\la}$} be the following 
 composite morphism
\beq{eqn:equivariant-cohomology-morphism}
\kgeom_{\cF,\la}:\
\xymatrix{
\coH^{\hdot}_{A}(i_{\lambda}^! \cF) \ \ar[rr]^<>(0.5){(\imath_\la)_!}&& \ 
\coH^{\hdot}_{A}(\fT_{\lambda}, t_{\lambda}^! \cF)\ 
\ar[rr]^<>(0.5){\eqref{zeta}}_<>(0.5){\sim}&&\
\bigl( \cS(\cF) \bigr)_{\lambda} \o\sh \langle \lambda(2\rhov)\rangle.
}
\eeq

Thus, we get a diagram of  morphisms
of graded $\sh$-modules
\[
\xymatrix@C=2.7cm{
\bigl(\cS(\cF) \otimes \M(\la)\rhoshift  \bigr)^B \langle \lambda(2\rhov)\rangle
\ar[r]^<>(0.5){\kalg_{\cS(\cF),\la}\langle \lambda(2\rhov)\rangle}_<>(0.5){\eqref{eqn:equivariant-cohomology-3}}
&
\bigl( \cS(\cF) \bigr)_{\lambda} \o  {\sh} \langle \lambda(2\rhov)\rangle
&
\coH^{\hdot}_{A}(i_{\lambda}^! \cF).
\ar[l]_<>(0.5){\kgeom_{\cF,\la}}^<>(0.5){\eqref{eqn:equivariant-cohomology-morphism}}
}
\]

One of our key results (to be proved in \S\ref{ss:proof-main-results}) reads as follows.

\begin{thm}
\label{thm:loop-equivariant-version} 
For any $\cF$ in $\Perv_{\Gv(\OO)}(\Gr)$ and $\lambda \in \bX$, the morphisms
 $\kalg_{\cS(\cF),\la} \langle \la(2\rhov) \rangle$
 and $\kgeom_{\cF,\la}$
are injective and have the same image. Thus, there is a
natural isomorphism of graded $\sh$-modules $\zeta_{\cF,\la}$\index{zeta@$\zeta_{\cF,\la}$} that fits into  the following 
commutative diagram
\beq{fits}
\vcenter{
\xymatrix@C=2cm{
\bigl(\cS(\cF) \otimes \M(\la) \bigr)^B \langle \lambda(2\rhov)\rangle
\ar[r]^<>(0.5){\zeta_{\cF,\la}}_-{\sim}
\ar@{^{(}->}[d]^<>(0.5){\kalg_{\cS(\cF),\la}\langle \lambda(2\rhov)\rangle}&
\coH^{\hdot}_{A}(i_{\lambda}^! \cF)
\ar@{^{(}->}[d]^<>(0.5){\kgeom_{\cF,\la}}\\
\bigl( \cS(\cF) \bigr)_{\lambda} \o\sh \langle \lambda(2\rhov)
\rangle \ar@{=}[r] &
\bigl( \cS(\cF) \bigr)_{\lambda} \o\sh \langle \lambda(2\rhov)
\rangle.
}
}
\eeq
\end{thm}

\begin{rem}
We have defined in \S\ref{subsec1} an action of $\zh$ on $(\cS(\cF) \o \M(\la))^B$. On the other hand, it is explained in \cite[\S 2.4]{bf} that $\coH^{\hdot}_{A}(i_{\lambda}^! \cF)$ also has a natural action of $\zh=\sym_\hb^W$ coming from the natural map $\Gr=(\Gv(\KK) \rtimes \Gm) / (\Gv(\OO) \rtimes \Gm) \to \pt / (\Gv(\OO) \rtimes \Gm)$. We claim that our isomorphism $\zeta_{\cF,\la}$ is also $\zh$-equivariant.

First, the action of $\sh \o \zh$ on $\coH^{\hdot}_{A}(i_{\lambda}^! \cF)$ factors through an action of $\sh \o_{\kh} \zh = \C[\t^* \times (\t^*/W) \times \bA^1]$. Then, the action of $\sh \o_{\kh} \zh$ factors through the natural action of the algebra $\coH^{\hdot}_A(\bla)$. Finally, it is explained in \cite[\S 3.2]{bf} that the $\sh \o_{\kh} \zh$-algebra $\coH^{\hdot}_A(\bla)$ is isomorphic to the direct image under the natural quotient map of $\oo_{\Gamma_\lambda}$, where
\[
\Gamma_\lambda:=\{(\eta_1,\eta_2,z) \in \t^* \times \t^* \times \bA^1 \mid \eta_2=\eta_1 + z \lambda \}.
\]
The claim easily follows from these remarks and the $\sh$-equivariance of $\zeta_{\cF,\la}$.
\end{rem}

\begin{rem}
Consider the case $\cF=\IC_\nu$ is the $\IC$-sheaf associated with the $\GvO$-orbit $\Gr^\nu$ for some $\nu \in \bX$, and $\la=w_0 \nu$ (where $w_0 \in W$ is the longest element).
Then $V^\nu:=\cS(\IC_\nu)$ is a simple $G$-module with highest weight $\nu$, and $\la$ is the lowest weight of $V^\nu$. In view of the right-hand isomorphism in Lemma \ref{lem:fixed-points-morphisms} below, the image of the morphism $\kalg_{V^\nu,w_0 \nu}$ is computed by Kashiwara in \cite{ka}: namely, with our conventions, combining Theorem 1.7 and Proposition 1.8 in \emph{loc}.~\emph{cit}.~we obtain that the image of $\kalg_{V^\nu,w_0 \nu}$ in $V^\nu_{w_0 \nu} \o \sh \cong \sh$ is generated by the following element:
\beq{eqn:poly-kashiwara}
\prod_{\alpha \in R^+} \left( \prod_{j=0}^{-\nu(w_0 \alpha)-1} (\alv - j \hb) \right).
\eeq
The topological context is easy in this case. Namely we have 
\[
\fT_{w_0 \nu} \cap \overline{\Gr^\nu} = \fT_{w_0 \nu} \cap \Gr^\nu = \Uv^-(\OO) \cdot (w_0 \bnu) \cong \prod_{\alpha \in R^+} \left( \prod_{j=0}^{-\nu(w_0 \alpha)-1} \C_{-\alv + j \hb} \right)
\]
as $A$-varieties. One can easily deduce that the image of $\kgeom_{\IC_\nu,w_0 \nu}$ is also generated by \eqref{eqn:poly-kashiwara}, see \S\ref{ss:equiv-cohomology}. Hence, in this particular case, Theorem \ref{thm:loop-equivariant-version} can be directly deduced from these remarks.

In the case $\la=\nu$, one can also directly check that both $\kalg_{V^\nu,\nu}$ and $\kgeom_{\IC_\nu,\nu}$ are isomorphisms.
\end{rem}

\subsection{Classical analogue}
\label{ss:classical}
We will also prove an analogue of Theorem \ref{thm:loop-equivariant-version} where one replaces $A$ by $\Tv$. In this case the representation theory of the algebra $\uh(\g)$ has to be replaced by the geometry of the algebraic variety $\g^*$.

We will identify $\t^*$ with the subspace $(\g/\u \oplus \u^-)^* \subset \g^*$. This way we obtain a canonical morphism $q : \sym(\fg/\u) \to \sym(\t)$ induced by restriction of functions. For $V$ in $\Rep(G)$ and $\la \in \bX$, the ``classical analogue'' of the morphism $\kalg_{V,\la}$, which we will denote by $\kbalg_{V,\la}$\index{kalg2@$\kbalg_{V,\la}$}, is the composition
\[
\bigl(V \otimes \sym(\g/\u) \o \C_{-\la}  \bigr)^B
\hookrightarrow \bigl(V \otimes \sym(\g/\u) \o \C_{-\la}
\bigr)^T \xrightarrow{\id_V \o q \o 1} (V \o 
\sym(\t) \o  \bbc_{-\lambda})^T =
V_{\lambda} \o \sym(\t).
\]
This morphism is $\sym(\t)$-equivariant, where the $\sym(\t)$-action on the left-hand side is induced by the morphism $(\g/\u)^* \to \t^*$ given by restriction of linear maps.

Now we consider perverse sheaves on $\Gr$. 
%
%
For $\cF$ in $\Perv_{\Gv(\OO)}(\Gr)$ and $\la \in \bX$,
we will denote by $\kbgeom_{\cF,\la}$\index{ktop2@$\kbgeom_{\cF,\la}$} the following 
 composite morphism
\[
\kbgeom_{\cF,\la}:\
\xymatrix{
\coH^{\hdot}_{\Tv}(i_{\lambda}^! \cF) \ \ar[rr]^<>(0.5){(\imath_\la)_!}&& \ 
\coH^{\hdot}_{\Tv}(\fT_{\lambda}, t_{\lambda}^! \cF)\ 
\ar[rr]^<>(0.5){\eqref{zeta}}_<>(0.5){\sim}&&\
\bigl( \cS(\cF) \bigr)_{\lambda} \o\sym(\t) \langle \lambda(2\rhov)\rangle.
}
\]
Then the classical analogue of Theorem \ref{thm:loop-equivariant-version} (to be proved in \S\ref{ss:proof-main-results-classical}) reads as follows.

\begin{thm}\label{thm:equiv-coh-classical} 
For any $\cF$ in $\Perv_{\Gv(\OO)}(\Gr)$ and $\lambda \in \bX$, the morphisms
 $\kbalg_{\cS(\cF),\la} \langle \lambda(2\rhov)\rangle$
 and $\kbgeom_{\cF,\la}$
are injective and have the same image. Thus, there is a
natural isomorphism of graded $\sym(\t)$-modules $\overline{\zeta}_{\cF,\la}$\index{zetab@$\overline{\zeta}_{\cF,\la}$} that fits into  the following 
commutative diagram
\beq{fits-classical}
\vcenter{
\xymatrix@C=2cm{
\bigl(\cS(\cF) \otimes \sym(\g/\u) \o \C_{-\la} \bigr)^B \langle \lambda(2\rhov)\rangle
\ar[r]^<>(0.5){\overline{\zeta}_{\cF,\la}}_-{\sim}
\ar@{^{(}->}[d]^<>(0.5){\kbalg_{\cS(\cF),\la}\langle \lambda(2\rhov)\rangle}&
\coH^{\hdot}_{\Tv}(i_{\lambda}^! \cF)
\ar@{^{(}->}[d]^<>(0.5){\kbgeom_{\cF,\la}}\\
\bigl( \cS(\cF) \bigr)_{\lambda} \o\sym(\t) \langle  \lambda(2\rhov)
\rangle \ar@{=}[r] &
\bigl( \cS(\cF) \bigr)_{\lambda} \o\sym(\t) \langle \lambda(2\rhov)
\rangle.
}
}
\eeq
\end{thm}

\subsection{Alternative descriptions: differential operators on  $G/U$ and intertwining operators for Verma modules}\label{subsec2}
An important role in our arguments will be played by two alternative descriptions of the $\C[\hb]$-modules $\bigl(V \otimes \M(\la)\rhoshift  \bigr)^B$.

If $X$ is a smooth algebraic variety, we write $\dd_X$\index{DX@$\dd_X$}  for the sheaf of differential operators on $X$.
The sheaf $\dd_X$ comes
equipped with a natural filtration by the
order of differential operator. We let 
 $\dd_{\hb,X}$\index{DXh@$\dd_{\hb,X}$} be the corresponding sheaf of asymptotic differential
operators. As for enveloping algebras, this algebra has an alternative description as the sheaf of graded $\C[\hb]$-algebras generated (locally) by $\oo_X$ in degree $0$ and the left $\oo_X$-module $\mathscr{T}_X$ (the tangent sheaf) in degree $2$, with relations $\xi \cdot \xi' - \xi' \cdot \xi = \hb [\xi,\xi']$ for $\xi,\xi' \in \mathscr{T}_X$ and $\xi \cdot f - f \cdot \xi = \hb \xi(f)$ for $\xi \in \mathscr{T}_X$ and $f \in \oo_X$. As for enveloping algebras, we will use this description of $\dd_{\hb,X}$ and still denote by $\xi$ the image of an element $\xi \in \mathscr{T}_X$. (If we were using the description provided by \eqref{eqn:def-rees}, this element should rather be denoted $\hb \xi$.) Note that $\dd_{\hb,X}$ acts on $\oo_X[\hb]$ via $\xi \cdot f = \hb \xi(f)$ for $\xi \in \mathscr{T}_X$ and $f \in \oo_X$.

We put $\dd(X)=\Ga(X, \dd_X)$,\index{DX2@$\dd(X)$} resp.
$\dh(X)= \Ga(X, \dd_{\hb,X})$,\index{DXh2@$\dh(X)$} for the corresponding
algebra of global sections. The order filtration
makes $\dd(X)$ a filtered algebra;
the associated Rees algebra $\dd(X)_\hb$ is canonically isomorphic to $\dh(X)$.
There is also a canonical \emph{injective} morphism
$\dh(X)/\hb \cdot \dh(X) \to \Ga(X,\dd_{\hb,X}/\hb \cdot \dd_{\hb,X})$, which is not surjective in general. 
Note finally that there exists a canonical algebra isomorphism
$\dd_{\hb,X}/\hb \cdot \dd_{\hb,X} \cong (p_X)_{*} \oo_{T^* X}$,
where $T^*X$ is the cotangent bundle to $X$ and $p_{X} : T^* X \to X$ is the projection.

Consider the quasi-affine variety $\X:=G/U$\index{X@$\X$}.
There is a natural $G\times T$-action on $\X$ 
defined as follows:
$g\times t:\ hU \mto ghtU.$
The 
$T$-action on $\X$ also induces an action of $T$ on 
$\dh(\X)$ by algebra automorphisms.
In particular,  this  $T$-action gives a weight decomposition
$\dh(\X)=\bplus_{\la\in\bX}\ \dh(\X)_\la$.
Thus, $\dh(\X)_0=\dh(\X)^T$ is the algebra of
right $T$-invariant asymptotic differential operators.

Differentiating the $T$-action on $\X$ yields
a morphism 
$\gamma : \sh \to \cD_{\hbar}(\X)$\index{gamma@$\gamma$} of graded $\bbc[\hbar]$-algebras.
Using this we will consider $\dh(\X)$ as an $\sh$-module where $t \in \t$ acts by right multiplication by $t - \hb \rho(t)$.
Note also that differentiating the $G$-action on $\X$ we obtain a morphism $\zh \to \dh(\X)$. This defines a $\zh$-module structure on $\dh(\X)$ (induced by multiplication on the left).

If $M$ is an $\sh$-module and $\varphi$ an algebra automorphism of $\sh$, we denote by ${}^{\varphi} \hspace{-1pt} M$ \index{phiM@${}^{\varphi} \hspace{-1pt} M$}the $\sh$-module which coincides with $M$ as a $\C$-vector space, and where $s \in \sh$ acts as $\varphi(m)$ acts on $M$. If $\varphi,\psi$ are algebra automorphisms of $\sh$ we have ${}^\psi \hspace{-1pt} \bigl( {}^\varphi \hspace{-1pt} M\bigr) = {}^{\varphi \circ \psi} \hspace{-1pt} M$. This construction provides an autoequivalence of the category of $\sh$-modules, acting trivially on morphisms. We will use in particular this notation when $\varphi=w \in W$ (extended in the natural way to an automorphism of $\sh$), and for the following automorphisms: if $\mu \in \bX$, we denote by $(\mu) : \sh \simto \sh$ the automorphism which sends $t \in \t$ to $t - \hb \mu(t) \in \sh$. We will use similar notation for $\sym(\t)$-modules.

If $M,N$ are $(\sh,\uh(\g))$-bimodules,
we will 
write $\Hom_{(\sh,\uh(\g))}(M,N)$ for the space of morphisms of bimodules from $M$ to $N$. It is an $\sh$-module and a $\zh$-module in a natural way. If $M,N$ are graded and $M$ is finitely generated as a bimodule then this space is a graded $\sh$-module and a $\zh$-module.

The following simple result
will be proved in \S\ref{ss:diffX} (for the first isomorphism) and \S\ref{ss:specialization} (for the second isomorphism).

\begin{lem}
\label{lem:fixed-points-morphisms}
For any $V$ in $\Rep(G)$ and $\lambda \in \bX$, there are canonical
isomorphisms of graded ${\sh}$-modules and $\zh$-modules.
\[
\bigl(V \otimes \tla \dh(\X)\lla\bigr)^G \ \cong\
\bigl(V \otimes \M(\la) \bigr)^B \ \cong \
\Hom_{(\sh,\uh(\g))}(\M(0),V \otimes
\M(\la)).
\]
\end{lem}
%
%
%
%
%
%

%

From Theorem \ref{thm:loop-equivariant-version} and Lemma
\ref{lem:fixed-points-morphisms} we deduce:

\begin{cor}\label{key_cor}
For any $\cF \in \Perv_{\Gv(\OO)}(\Gr)$ and $\lambda \in \bX$, there are natural
isomorphisms of graded ${\sh}$- and $\zh$-modules
\[
\bigl(\cS(\cF) \otimes \tla \dh(\X)\lla\bigr)^G \langle \lambda(2\rhov)
\rangle \ \cong\
\coH^{\hdot}_{A}(i_{\lambda}^! \cF)\ \cong \
\Hom_{(\sh,\uh(\g))}(\M(0),\cS(\cF) \otimes
\M(\la)) \langle \lambda(2\rhov)
\rangle.
\]
\end{cor}


One can also give an alternative description of equivariant cohomology of cofibers in the ``classical case'' of \S\ref{ss:classical}, as follows. Let $\B:=G/B$.\index{B@$\B$} For $\la \in \bX$, we denote by $\cO_\B(\la)$ the line bundle on $\B$ associated with the character $-\lambda$ of $B$ (so that ample line bundles correspond to dominant weights).
For any variety $X$ over $\B$, and any $\la \in \bX$, we will denote by $\oo_X(\la)$ the pull-back to $X$ of $\oo_{\B}(\la)$. In particular, consider the $G$-varieties
$G/T \times \t^*$ and $\wfg := G \times_B (\g/\u)^*$\index{gtilde@$\wfg$},
which are both equipped with a natural morphism to $\B$. Both varieties are equipped with a natural action of $G \times \C^\times$, where the action of $G$ is induced by left multiplication on $G$, and any $z \in \C^\times$ acts by multiplication by $z^{-2}$ on $(\g/\u)^*$ or $\t^*$.

Consider the morphism
\[
a : G/T \times \ft^* \to \wfg, \quad (gT,\eta) \mapsto (g \times_B \eta)
\]
where on the right-hand side $\eta$ is considered as an element of $\g^*$ trivial on $\u \oplus \u^-$. For any $\la \in \bX$ we have a canonical isomorphism $a^* \oo_{\wfg}(\la) \cong \oo_{G/T \times \ft^*}(\la)$, so that we obtain a pull-back morphism $a^* : \Ga \bigl( \wfg, \oo_{\wfg}(\la) \bigr) \to \Ga \bigl( G/T \times \t^*, \oo_{G/T \times \t^*}(\la) \bigr)$.

Note that we have a canonical isomorphism $\Ga \bigl( G/T \times \t^*, \oo_{G/T \times \ft^*}(\la) \bigr) \cong \Ind_T^G(-\la) \otimes \sym(\t)$ where $\Ind$\index{Ind@$\Ind$} is the usual induction functor for representations of algebraic groups, as defined e.g.~in \cite[\S I.3.3]{ja}. Hence if $V$ is in $\Rep(G)$, using the tensor identity and Frobenius reciprocity we obtain an isomorphism
\beq{eqn:wfg'}
\bigl( V \o \Ga ( G/T \times \ft^*, \oo_{G/T \times \t^*}(\la) ) \bigr)^G \cong \bigl( V \o \sym(\t) \o \C_{-\la} \bigr)^T = V_\la \o \sym(\t). 
\eeq
By a similar argument, there exists a canonical isomorphism
\beq{eqn:wfg}
\bigl( V \o \Ga ( \wfg, \oo_{\wfg}(\la) ) \bigr)^G \cong \bigl( V \o \sym(\g/\u) \o \C_{-\la} \bigr)^B. 
\eeq
Then one can easily check that, under isomorphisms \eqref{eqn:wfg'} and \eqref{eqn:wfg}, the morphism $\kbalg_{V,\la}$ identifies with the morphism
\[
\bigl( V \o \Ga ( \wfg, \oo_{\wfg}(\la) ) \bigr)^G \to 
\bigl( V \o \Ga ( G/T \times \t^*, \oo_{G/T \times \t^*}(\la) ) \bigr)^G
\]
induced by $a^*$.

Using \eqref{eqn:wfg}, from Theorem \ref{thm:equiv-coh-classical} we deduce the following description.

\begin{cor}
\label{key_cor-classical}
For any $\cF\in \Perv_{\Gv(\OO)}(\Gr)$ and $\lambda \in \bX$, there exists a natural
isomorphism of graded $\sym(\t)$-modules
\[
\coH^{\hdot}_{\Tv}(i_{\lambda}^! \cF)\ \cong \
\bigl( \cS(\cF) \o \Ga ( \wfg, \oo_{\wfg}(\la) ) \bigr)^G \langle  \la(2\rhov) \rangle.
\]
\end{cor}

\subsection{Weyl group symmetries}
\label{ss:W-symmetries}

Each of the spaces in Corollaries \ref{key_cor} and \ref{key_cor-classical} exhibits a kind of symmetry governed by the Weyl group $W$. These symmetries play a technical role in our proofs of Theorem \ref{thm:loop-equivariant-version} and \ref{thm:equiv-coh-classical}. But we will also show that they are respected by the isomorphisms in Corollaries \ref{key_cor} and \ref{key_cor-classical}. Some of our constructions are based on isomorphisms which do not respect the gradings; hence for simplicity we just forget the gradings in this subsection.

The constructions on the side of the group $G$ depend on the choice of root vectors for all simple roots. For these constructions to match with the constructions in perverse sheaves on $\Gr$, one has to choose the root vectors provided by the Tannakian construction of $G$ from the tensor category $\Perv_{\Gv(\OO)}(\Gr)$; see \S\ref{ss:root-vectors} for details.



The symmetry in the case of equivariant cohomology of cofibers of perverse sheaves is easy to construct. Namely, the normalizer $N_{\Gv}(\Tv)$ of $\Tv$ in $\Gv$ acts naturally on $\Gr$; we denote by $m_g : \Gr \simto \Gr$ the action of $g \in N_{\Gv}(\Tv)$. If we denote by $g \mapsto \overline{g}$ the projection $N_{\Gv}(\Tv) \twoheadrightarrow N_{\Gv}(\Tv) / \Tv \cong W$, for any $\la \in \bX$  we have $m_g \circ i_\la = i_{\overline{g} \lambda}$ (where we identify the one-point varieties $\{\la\}$ and $\{\overline{g} \la \}$). If $\cF$ is in $\Perv_{\Gv(\OO)}(\Gr)$, we deduce an isomorphism of graded $\sh$-modules
\[
\coH^{\hdot}_A(i_\la^! m_g^! \cF) \simto {}^{\overline{g}} \hspace{-1pt} \coH^{\hdot}_A(i_{\overline{g} \la}^! \cF).
\]
On the other hand, since $\cF$ is $\Gv(\OO)$-equivariant (hence in particular $N_{\Gv}(\Tv)$-equivariant), there exists a canonical isomorphism $m_g^! \cF \cong \cF$. Hence we obtain an isomorphism $\coH^{\hdot}_A(i_\la^!\cF) \simto {}^{\overline{g}} \hspace{-1pt} \coH^{\hdot}_A(i_{\overline{g} \la}^! \cF)$. Using classical arguments one can check that this isomorphism only depends on $\overline{g}$ (and not on $g$); we will denote by\index{Xi@$\Xi^{\cF,\la}_w$}%
\[
\Xi^{\cF,\la}_w : \coH^{\hdot}_A(i_\la^!\cF) \simto \w \coH^{\hdot}_A(i_{w \la}^! \cF)
\]
the isomorphism associated with $w \in W$. These isomorphisms generate an action of $W$ in the sense that for any $\cF$ in $\Perv_{\Gv(\OO)}(\Gr)$, $\la \in \bX$ and $x,y \in W$ we have
\begin{equation}
\label{eqn:cocycle-Xi}
\y \bigl( \Xi^{\cF,y \la}_x \bigr) \circ \Xi^{\cF,\la}_y \ = \ \Xi^{\cF,\la}_{xy}.
\end{equation}

Now, let us consider $\dd$-modules on $\X$. In \S\ref{ss:partial-Fourier} we recall a construction of Gelfand--Graev (see \cite{kla, bbp}) based on ``partial Fourier transforms'' for $\dd$-modules which provides an action of $W$ on $\dh(\X)$ by algebra automorphisms such that, for any $w \in W$ and $\la \in \bX$ the action of $w$ restricts to an isomorphism of $\sh$-modules
\begin{equation}
\label{eqn:isom-Fourier-DX}
\tla \dh(\X)_{\la} \simto \w \bigl( {}^{(w\la)} \hspace{-1pt} \dh(\X)_{w \la} \bigr).
\end{equation}
For $V$ in $\Rep(G)$, we will denote by\index{Phi@$\Phi_w^{V,\la}$}%
\[
\Phi_w^{V,\la} : \bigl(V \o \tla \dh(\X)_\la \bigr)^G \simto \w \bigl(V \o {}^{(w\la)} \hspace{-1pt} \dh(\X)_{w \la} \bigr)^G
\]
the induced isomorphism. This collection of isomorphisms satisfies the relations
\begin{equation}
\label{eqn:cocycle-Phi}
\y \bigl( \Phi^{V,y \la}_x \bigr) \circ \Phi^{V,\la}_y \ = \ \Phi^{V,\la}_{xy}.
\end{equation}

Finally, in \S\ref{ss:def-Theta} we will define, for any $V$ in $\Rep(G)$, $\la \in \bX$ and $w \in W$, an isomorphism of $\sh$-modules\index{Theta@$\Theta_w^{V,\la}$}%
\[
\Theta_w^{V,\la} : \Hom_{(\sh,\uh(\g))}(\M(0),V \otimes \M(\la)) \simto \w \Hom_{(\sh,\uh(\g))}(\M(0),V \otimes
\M(w \la)).
\]
This collection of isomorphisms naturally appears in the construction of the \emph{dynamical Weyl group}, see \S\ref{ss:dwg} below. As above, it satisfies the relations
\begin{equation}
\label{eqn:cocycle-Theta}
\y \bigl( \Theta^{V,y \la}_x \bigr) \circ \Theta^{V,\la}_y \ = \ \Theta^{V,\la}_{xy}.
\end{equation}

Our second main result, which will be proved in \S\ref{ss:proof-main-results}, is the following.

\begin{thm}
\label{thm:W-symmetry}
The isomorphisms of Corollary {\rm \ref{key_cor}} are such that the following diagram commutes for any $\cF$ in $\Perv_{\Gv(\OO)}(\Gr)$, $\la \in \bX$ and $w \in W$:
\[
\xymatrix@C=1cm{
\bigl( \cS(\cF) \o \tla \dh(\X)_\la \bigr)^G 
\ar[d]_-{\wr}^-{\Phi^{\cS(\cF),\la}_w 
} \ar[r]^-{\sim} & \coH^{\hdot}_A(i_\la^!\cF) \ar[d]_-{\wr}^-{\Xi^{\cF,\la}_w} \ar[r]^-{\sim} & \Hom_{(\sh,\uh(\g))}(\M(0), \cS(\cF) \otimes \M(\la)) 
\ar[d]_-{\wr}^-{\Theta^{\cS(\cF),\la}_w 
} \\ 
\w \bigl(\cS(\cF) \o {}^{(w\la)} \hspace{-1pt} \dh(\X)_{w \la} \bigr)^G 
\ar[r]^-{\sim} & \w \coH^{\hdot}_A(i_{w\la}^!\cF)  \ar[r]^-{\sim} & \w \Hom_{(\sh,\uh(\g))}(\M(0), \cS(\cF) \otimes
\M(w \la)). 
}
\]
\end{thm}

One can also give a ``classical'' analogue of Theorem \ref{thm:W-symmetry}. First, the same construction as above provides, for any $\cF$ in $\Perv_{\Gv(\OO)}(\Gr)$, $\la \in \bX$ and $w \in W$, an isomorphism of graded $\sym(\t)$-modules\index{xi@$\xi^{\cF,\la}_w$}%
\[
\xi^{\cF,\la}_w : \coH^{\hdot}_{\Tv}(i_\la^!\cF) \simto \w \coH^{\hdot}_{\Tv}(i_{w \la}^! \cF).
\]
This collection of isomorphisms satisfies relations
\[
\y \bigl( \xi^{\cF,y \la}_x \bigr) \circ \xi^{\cF,\la}_y \ = \ \xi^{\cF,\la}_{xy}.
\]

On the other hand, let $\g^*_{\mathrm{r}} \subset \g^*$\index{g*r@$\g^*_{\mathrm{r}}$} be the open set consisting of regular elements, and let $\wfgr$\index{gtilder@$\wfgr$} be the inverse image of $\g^*_{\mathrm{r}}$ under the morphism $\pi : \wfg \to \g^*$\index{pi@$\pi$} defined by $\pi(g \times_B \eta) = g \cdot \eta$. Then there exists a canonical action of $W$ on $\wfgr$, which induces for any $\la \in \bX$ and $w \in W$ an isomorphism of $\sym(\t)$-modules and $G$-modules
\[
\Ga \bigl( \wfg, \oo_{\wfg}(\la) \bigr) \simto \w \Ga \bigl( \wfg,\oo_{\wfg}(w\la) \bigr),
\]
see \S\ref{ss:construction-sigma} for details.
Hence we obtain, for any $V$ in $\Rep(G)$, an isomorphism\index{sigma@$\sigma^{V,\la}_w$}%
\[
\sigma^{V,\la}_w : \bigl( V \o \Ga ( \wfg, \oo_{\wfg}(\la) ) \bigr)^G \simto \w \bigl( V \o \Ga ( \wfg, \oo_{\wfg}(w \la) ) \bigr)^G.
\]
This collection of isomorphisms again satisfies relations
\begin{equation}
\label{eqn:cocycle-sigma}
\y \bigl( \sigma^{V,y \la}_x \bigr) \circ \sigma^{V,\la}_y \ = \ \sigma^{V,\la}_{xy}.
\end{equation}

We have the following compatibility property, to be proved in \S\ref{ss:proof-main-results-classical}.

\begin{thm}
\label{thm:W-symmetry-classical}
The isomorphism of Corollary {\rm \ref{key_cor-classical}} is such that the following diagram commutes for any $\cF$ in $\Perv_{\Gv(\OO)}(\Gr)$, $\la \in \bX$ and $w \in W$:
\[
\xymatrix@C=1.5cm{
\bigl(\cS(\cF) \o\Ga ( \wfg, \oo_{\wfg}(\la)  \bigr)^G 
\ar[d]_-{\wr}^-{\sigma^{\cS(\cF),\la}_w 
} \ar[r]^-{\sim} & \coH^{\hdot}_{\Tv}(i_\la^!\cF) \ar[d]_-{\wr}^-{\xi^{\cF,\la}_w} \\ 
\w \bigl(\cS(\cF) \o \Ga ( \wfg, \oo_{\wfg}(w \la) \bigr)^G 
\ar[r]^-{\sim} & \w \coH^{\hdot}_{\Tv}(i_{w\la}^!\cF).
}
\]
\end{thm}

\subsection{Applications: dynamical Weyl groups and Brylinski--Kostant filtration}
\label{ss:dwg}

The first application of our results concerns a geometric realization of the \emph{dynamical Weyl group} due to Braverman--Finkelberg (\cite{brf}). Let $\qh$\index{Qh@$\qh$} be the field of fractions of $\sh$. If $M$ is a $\qh$-module, we define the $\qh$-module $\w M$ by a similar formula as above. In \S\ref{ss:definition-dwg} we will recall the definition of the dynamical Weyl group, a collection of isomorphisms of $\qh$-modules\index{DWalg@$\dwalg_{V,\la,w}$}%
\[
\dwalg_{V,\la,w} : \qh \o V_\la \simto \w \qh \o V_{w \la}
\]
for all $V$ in $\Rep(G)$, $\la \in \bX$ and $w \in W$.

Let $\OO^-:=\C[z^{-1}]$,\index{O-@$\OO^-$} and let $\Gv(\OO^-)_1$ be the kernel of the morphism $\Gv(\OO^-) \to \Gv$ given by evaluation at $z=\infty$. For $\la \in \bX$ we set $\W_\la := \Gv(\OO^-)_1 \cdot \bla$;\index{Wlambda@$\W_\la$} this is a locally closed (ind-)subvariety of $\Gr$ which is a transversal slice to the orbit $\Gr^\la$ at $\bla$. We denote by\index{slambda@$s_\la$}%
\[
s_\la : \W_\la \cap \fT_\la \hookrightarrow \Gr
\]
the inclusion. We also set\index{nlambda@$n_\la$}%
\[
n_\la:= \sum_{\alpha \in R^+} | \la(\alv) | 
\]
The following result is proved in \cite{brf}; we reproduce the proof in \S\ref{ss:transverse-slices} since we will need some of the details.

\begin{lem}
\label{lem:transverse-slices}
For any $\cF$ in $\Perv_{\Gv(\OO)}(\Gr)$ and $\la \in \bX$ there exists a canonical isomorphism of $\sh$-modules
\[
\coH^{\hdot}_A(\W_\la \cap \fT_\la, s_\la^! \cF) \cong \bigl( \cS(\cF) \bigr)_\la \o \sh \langle n_\la \rangle.
\]
\end{lem}

The inclusion $\{\bla\} \hookrightarrow \W_\la \cap \fT_\la$ induces a morphism of $\qh$-modules\index{Delta@$\Delta^{\cF,\la}$}%
\[
\Delta^{\cF,\la} : \qh \o_{\sh} \coH^{\hdot}_A(i_\la^! \cF) \to \qh \o_{\sh} \coH^{\hdot}_A(\W_\la \cap \fT_\la, s_\la^! \cF)
\]
which is an isomorphism due to the localization theorem in equivariant cohomology (see e.g.~\cite[Theorem B.2]{em}). For $\cF$ in $\Perv_{\Gv(\OO)}(\Gr)$, $\la \in \bX$ and $w \in W$ we define the morphism\index{DWgeom@$\dwgeom_{\cF,\la,w}$}%
\begin{multline*}
\dwgeom_{\cF,\la,w} := \w \bigl( \Delta^{\cF,w\la} \bigr) \circ \Xi^{\cF,\la}_w \circ (\Delta^{\cF,\la})^{-1} : \\ 
\qh \o_{\sh} \coH^{\hdot}_A(\W_\la \cap \fT_\la, s_\la^! \cF) \simto \w \bigl( \qh \o_{\sh} \coH^{\hdot}_A(\W_{w\la} \cap \fT_{w\la}, s_{w \la}^! \cF) \bigr).
\end{multline*}

The following result (which is a consequence of Theorem \ref{thm:W-symmetry}) is equivalent to the main result of \cite{brf}. Our proof, given in \S\ref{ss:dwg-proof}, cannot really be considered as a new proof since it is based on a similar strategy (namely reduction to rank $1$), but we believe our point of view should help understanding this question better.

\begin{prop}
\label{prop:dwg}
For any $\cF$ in $\Perv_{\Gv(\OO)}(\Gr)$ and $\la \in \bX$ \emph{dominant}, the following diagram is commutative, where the vertical isomorphisms are induced by those of Lemma {\rm \ref{lem:transverse-slices}}.
\[
\xymatrix@C=2.5cm{
\qh \o_{\sh} \coH^{\hdot}_A(\W_\la \cap \fT_\la, s_\la^! \cF) \ar[r]^-{\dwgeom_{\cF,\la,w}} \ar[d]_-{\wr} & \w \bigl( \qh \o_{\sh} \coH^{\hdot}_A(\W_{w\la} \cap \fT_{w\la}, s_{w \la}^! \cF) \bigr) \ar[d]^-{\wr} \\
\qh \o \bigl( \cS(\cF) \bigr)_\la \ar[r]^-{\dwalg_{\cS(\cF),\la,w}} & \w \bigl( \qh \o \bigl( \cS(\cF) \bigr)_{w \la} \bigr)
}
\]
\end{prop}


The second application of our results is a new proof of a result of the first author (\cite{gi}) giving a geometric construction of the \emph{Brylinski--Kostant} filtration. Namely, let $e \in \u$ be a regular nilpotent element which is a sum of (non-zero) simple root vectors.
If $V$ is in $\Rep(G)$ and $\la \in \bX$, the Brylinski--Kostant filtration on $V_\la$ associated with $e$ is defined by\index{FBK@$\FBK$}%
\[
\FBK_i(V_\la) = \{ v \in V_\la \mid e^{i+1} \cdot v = 0\} \qquad \text{for } i \geq 0
\]
and $\FBK_i(V_\la)=0$ for $i<0$. This filtration is independent of the choice of $e$ (since all the choices for $e$ are conjugate under the action of $T$).

On the other hand, for any $\varphi \in \t^*$ we
consider the specialized equivariant cohomology\index{Hphi@$\coH_{\varphi}$}%
\[
\coH_\varphi(i_\la^! \cF) \ := \ \C_\varphi \otimes_{\sym(\t)} \coH_{\Tv}^{\hdot}(i_\la^! \cF),
\]
a filtered vector space. Assume that $\varphi \in \t^* \smallsetminus \{0\}$ satisfies $(\mathrm{ad}^*e)^2 (\varphi)=0$, where $\mathrm{ad}^*$ is the coadjoint representation. Then $\varphi$ is regular, hence
by the localization theorem in equivariant cohomology, the morphism
\[
\coH_\varphi(i_\la^! \cF) \to \C_\varphi  \o_{\sym(\t)} \bigl( (\cS(\cF))_\la \o \sym(\t) \bigr) = \bigl( \cS(\cF) \bigr)_\la
\]
induced by $\kbgeom_{\cF,\la}$ is an isomorphism. The left-hand side is equipped with a natural filtration; we denote by $\Fgeom_{\hdot} \bigl( (\cS(\cF))_\la \bigr)$\index{Fgeom@$\Fgeom$} the resulting filtration on $\bigl( \cS(\cF) \bigr)_\la$.

The following result was first proved in \cite{gi} (see also \cite{ar} for a different proof). We observe in \S\ref{ss:BK-filtration} that it is an immediate consequence of Theorem \ref{thm:equiv-coh-classical}.

\begin{prop}
\label{prop:BK}
For any $\cF$ in $\Perv_{\Gv(\OO)}(\Gr)$, $\la \in \bX$ and $i \in \Z$ we have
\[
\FBK_i \bigl( (\cS(\cF))_\la \bigr) = \Fgeom_{2i+\la(2{\check \rho})} \bigl( (\cS(\cF))_\la \bigr) = \Fgeom_{2i+1+\la(2{\check \rho})} \bigl( (\cS(\cF))_\la \bigr).
\]
\end{prop}

\begin{rem}
Assume $G$ is quasi-simple. Then
it follows from \cite[Corollary 8.7]{Ko} that the $3$-dimensional representation of any $\mathfrak{sl}_2$-triple through $e$ occurs only once in $\g$, hence in $\g^*$. It follows that $\varphi \in \t^*$ is uniquely defined, up to scalar, by the condition $(\mathrm{ad}^*e)^2 (\varphi)=0$. Using this remark one can easily check (for a general reductive $G$, and independently of Proposition \ref{prop:BK}) that the filtration $\Fgeom_{\hdot}$ is independent of the choice of $\varphi$ once $e$ is fixed. On the other hand all the possible choices for $e$ are conjugate under the action of $T$, hence our filtration is independent of any choice (other than $T$ and $B$).
\end{rem}

\section{Differential operators on the basic affine space and partial Fourier transforms}
\label{sec:DX}

In Sections \ref{sec:DX}--\ref{sec:wfg} we fix a complex connected reductive group $G$, and we use the notation of \S\ref{subsec1}. We also choose for any simple root $\alpha$ a non-zero vector $e_\alpha \in \g_\alpha$. We denote by $f_\alpha \in \g_{-\alpha}$ the unique vector such that $[e_\alpha,f_\alpha]=\alv$.


\subsection{Structure of $\dh(\X)$}
\label{ss:diffX}

The results in this subsection are taken from \cite{bgg, Sh}. Below we will use the two natural actions of $G$ on $\C[G]$ induced by the actions of $G$ on itself. The action given by $(g \cdot f)(h) = f(g^{-1} h)$ for $f \in \C[G]$ and $g,h \in G$ will be called the \emph{left regular representation}; it is a \emph{left} action of $G$. The action given by $(f \cdot g)(h)=f(hg^{-1})$ for $f \in \C[G]$ and $g,h \in G$ will be called the \emph{right regular representation}; it is a \emph{right} action of $G$.

First we begin with the description of $\dh(G)$. Differentiating the right regular representation defines an anti-homomorphism of algebras $\uh(\g) \to \dh(G)$. Then it is well known that multiplication in $\dh(G)$ induces an isomorphism of $\C[G]$-modules
\begin{equation}
\label{eqn:diffG}
\C[G] \o \uh(\g) \simto \dh(G).
\end{equation}
The left regular representation induces an action on $\dh(G)$, which will be called simply the \emph{left action} below. Through isomorphism \eqref{eqn:diffG}, it is given by the left regular representation of $G$ on $\C[G]$ (and the trivial action on $\uh(\g)$). Similarly, the action induced by the right regular representation (which will be called simply the \emph{right action} below) corresponds, under isomorphism \eqref{eqn:diffG}, to the right action on $\C[G] \o \uh(\g)$ which is the tensor product of the right regular representation and the action on $\uh(\g)$ which is the composition of the adjoint action with the anti-automorphism $g \mapsto g^{-1}$. There is also a natural morphism $\uh(\g) \to \dh(G)$ obtained by differentiation of the left regular representation. Under isomorphism \eqref{eqn:diffG}, it is given by the map which sends $m \in \uh(\g)$ to the function $G \ni g \mapsto g^{-1} \cdot m \in \uh(\g)$, considered as an element in $\C[G] \o \uh(\g)$. In particular, this morphism restricts to the morphism $m \mapsto 1 \o m$ on $\zh \subset \uh(\g)$.

Let us recall the standard description of $\dh(\X)$ based on quantum hamiltonian reduction. As $\X$ is a quasi-affine variety, the natural morphism 
\[
\dh(\X) \to \End_{\C[\hb]}(\C[\X][\hb])
\] 
is injective. The algebra $\C[\X]$ identifies with the subalgebra of $\C[G]$ given by the elements fixed by the right action of $U \subset G$. Any element $D$ in $\dh(G)$ induces a morphism $\C[\X][\hb] \to \C[G][\hb]$. This morphism is trivial iff $D \in \C[G] \o \bigl(\u \cdot \uh(\g) \bigr)$, and its image is contained in $\C[\X][\hb]$ iff the image of $D$ in $\C[G] \o \bigl(\uh(\g) / \u \cdot \uh(\g) \bigr)$ is $U$-invariant for the right action.
In this way we obtain a canonical isomorphism
\begin{equation}
\label{eqn:diffX}
\dh(\X) \ \cong \ \bigl( \C[G] \o \bigl(\uh(\g) / \u \cdot \uh(\g) \bigr) \bigr)^{U_{\mathrm{right}}}.
\end{equation}
In this description, the action induced by the left $G$-action on $\X$ is induced by the left regular representation on $\C[G]$. The morphism $\zh \to \dh(\X)$ obtained from the differentiation of this left action of $G$ on $\X$ corresponds to the morphism $m \mapsto 1 \otimes (m \ \mathrm{mod} \ \u \cdot \uh(\g))$.

Recall that there is also a $T$-action on $\X$ defined by $t \cdot gU = gtU$. (Note that this action is \emph{not} induced by the right action of $G$ on itself considered above, but rather by its composition with $t \mapsto t^{-1}$.) This action provides a $T$-action on $\dh(\X)$ (where the action of $t \in T$ is induced by the right action of $t^{-1} \in G$ on $\dh(G)$ described above) and a weight decomposition $\dh(\X)=\bigoplus_{\la \in \bX} \dh(\X)_\la$. This $T$-action also defines a morphism $\gamma : \sh \to \dh(\X)$ which, under isomorphism \eqref{eqn:diffX}, is given by $\gamma(m) = 1\o (m \ \mathrm{mod} \ \u \cdot \uh(\g))$. Recall that we consider $\dh(\X)$ as an $\sh$-module where $t \in \t$ acts by right multiplication by $\gamma(t) - \hb \rho(t) \cdot 1$. Under isomorphism \eqref{eqn:diffX}, this action is given by the $\sh$-action on $\uh(\g) / \u \cdot \uh(\g)$ where $t \in \t$ acts by left multiplication by $t-\hb\rho(t)$.

From this description (and isomorphism \eqref{eqn:isom-Verma}) one easily obtains the following result. 

\begin{lem}
\label{lem:DX-Ind}
For any $\la \in \bX$ there exists a canonical isomorphism of graded $\sh$-modules and $\zh$-modules
\[
\tla \dh(\X)_\la \ \cong \ \Ind_B^G(\M(\la)).
\]
\end{lem}

\begin{proof}[Proof of the first isomorphism in Lemma {\rm \ref{lem:fixed-points-morphisms}}]
We have
\[
\bigl( V \o \tla \dh(\X)_\la \bigr)^G \ \cong \ \bigl( V \o \Ind_B^G(\M(\la)) \bigr)^G \ \cong \ \bigl( V \o \M(\la) \bigr)^B
\]
where the first isomorphism follows from Lemma \ref{lem:DX-Ind} and the second one from the tensor identity and Frobenius reciprocity.
\end{proof}

\subsection{Partial Fourier transforms for $\dh(\X)$}
\label{ss:partial-Fourier}

Let us recall a construction due to Gelfand--Graev, and studied by Kazhdan--Laumon \cite{kla} in the $\ell$-adic setting and by Bezrukavnikov--Braverman--Positselskii \cite{bbp} in our $\dd$-module setting. We choose a reductive group $G^\sc$ with simply-con\-nected derived subgroup and a surjective group morphism $G^\sc \twoheadrightarrow G$ with finite central kernel denoted $Z$. 
We denote by $T^\sc$, $B^\sc$ the inverse images of $T$, $B$ in $G^\sc$, and
let $U^\sc$ be the unipotent radical of $B^\sc$. We set $\X^\sc:=G^\sc/U^\sc$. Note that $Z$ acts naturally on $\X^\sc$, with quotient $\X$. Note also that for any simple root $\alpha$ there exists a unique injective morphism of algebraic groups
$\varphi_\alpha : \mathrm{SL}(2,\C) \to G^\sc$ such that
\[
\forall z \in \C^\times, \quad \varphi_\alpha \begin{pmatrix}
z & 0 \\ 0 & z^{-1}
\end{pmatrix} = \alv(z) 
\qquad \text{and} \qquad
d(\varphi_\alpha) \begin{pmatrix}
0 & 0 \\ 1 & 0
\end{pmatrix} = f_\alpha, \quad
d(\varphi_\alpha) \begin{pmatrix}
0 & 1 \\ 0 & 0
\end{pmatrix} = e_\alpha 
\]
(where we identify the Lie algebras of $G^\sc$ and $G$.)

Let $\alpha$ be a simple root, let $P_\alpha^\sc$ be the minimal parabolic subgroup of $G^\sc$ containing $B^\sc$ associated with $\alpha$, and let $Q^\sc_{\alpha}:=[P^\sc_\alpha,P^\sc_\alpha]$. Consider the projection
\[
\tau_{\alpha} :\ \X^\sc \to G^\sc/Q^\sc_{\alpha}.
\]
It is explained in \cite[\S 2.1]{kla} that $\tau_{\alpha}$ is the complement of the zero section of a $G^\sc$-equivariant vector bundle
\[
\tau'_{\alpha} :\ \cV_{\alpha} \to G^\sc/Q^\sc_{\alpha}
\]
of rank $2$. Moreover, there exists a canonical $G^\sc$-equivariant symplectic form on this vector bundle (which depends on $\varphi_\alpha$, i.e.~on the choice of $f_\alpha$). Hence the constructions recalled in \S\ref{ss:symplectic-Fourier} provide an automorphism of $\dh(\cV_{\alpha})$ as a $\C[\hb]$-algebra. As the complement of $\X^\sc$ in $\cV_{\alpha}$ has codimension $2$, restriction induces an isomorphism $\dh(\cV_{\alpha}) \simto \dh(\X^\sc)$. Hence we obtain an automorphism $\mathbf{F}_\alpha^\sc : \dh(\X^\sc) \simto \dh(\X^\sc)$. This automorphism is $Z$-equivariant (since $Z$ acts on $\cV_{\alpha}$ by symplectic automorphisms), and we have $\dh(\X^\sc)^Z=\dh(\X)$ in a natural way. Hence we obtain a $G$-equivariant $\C[\hb]$-algebra automorphism\index{Falpha@$\mathbf{F}_\alpha$}%
\[
\mathbf{F}_\alpha : \dh(\X) \simto \dh(\X). 
\]
on $Z$-fixed points. Using the fact that any two simply connected covers of a connected semi-simple group are isomorphic (as covers of the given group), one can check that the automorphism $\mathbf{F}_\alpha$ does not depend on the choice of $G^\sc$.

%

\begin{lem}
\label{lem:Fourier-transform}
\begin{enumerate}
\item The automorphisms $\mathbf{F}_{\alpha}$, $\alpha$ a simple root, generate an action of $W$ on $\dh(\X)$.
\item For any simple root $\alpha$ and any $\lambda \in \bX$, setting $s=s_\alpha$, $\mathbf{F}_{\alpha}$ restricts to an isomorphism of $G$-modules and of ${\sh}$- and $\zh$-modules
\[
\dh(\X)\lla \ \simto \ {}^{s} \dh(\X)_{s \lambda}.
\]
\end{enumerate}
\end{lem}

\begin{proof}
We observe that there are natural isomorphisms
\[
\dh(\X) [\hb^{-1}] \cong \Gamma \bigl( \X, \dd_{\hb,\X}[\hb^{-1}] \bigr) \cong \Gamma \bigl( \X, \dd_{\X} \o \C[\hb,\hb^{-1}] \bigr) \cong \dd(\X) \o \C[\hb,\hb^{-1}].
\]
Moreover, under these isomorphisms, the automorphism induced by $\mathbf{F}_\alpha$ coincides with the tensor product of the similar automorphism of $\dd(\X)$ considered in \cite{bbp} with $\id_{ \C[\hb,\hb^{-1}]}$. Hence the lemma follows from \cite[Proposition 3.1 and Lemma 3.3]{bbp}.
\end{proof}

By $(1)$, we can define a group morphism $w \mapsto \mathbf{F}_w$ from $W$ to the group of $\C[\hb]$-algebra automorphisms of $\dh(\X)$,
such that $\mathbf{F}_{s_\alpha}=\mathbf{F}_\alpha$ for any simple root $\alpha$. And by $(2)$ these isomorphisms restrict to isomorphisms of $G$-modules and of ${\sh}$- and $\zh$-modules\index{Flambda@$\mathbf{F}_w^{\la}$}%
\[
\mathbf{F}_w^{\la} : \dh(\X)\lla \ \simto \ {}^{w} \dh(\X)_{w(\lambda)}
\]
which satisfy the relations
\[
\y \bigl( \mathbf{F}_x^{y\la} \bigr) \circ \mathbf{F}_y^\la \ = \ \mathbf{F}_{xy}^\la.
\]
Using the relation $\tla (\w M) = \w ({}^{(w\la)} \hspace{-1pt} M)$ we obtain isomorphism \eqref{eqn:isom-Fourier-DX}, which allows to define the collection of isomorphisms $\Phi^{V,\la}_w$ of \S\ref{ss:W-symmetries}.

\begin{rem}
\begin{enumerate}
\item
As explained in \S\ref{ss:partial-Fourier-definition}, the construction of isomorphisms $\Phi^{V,\la}_w$ is not compatible with the natural gradings on $(V \o \tla \dh(\X)_\la)^G$ and $(V \o {}^{(w\la)} \hspace{-1pt} \dh(\X)_{w \la})^G$. However, it follows from Theorem \ref{thm:W-symmetry} that it induces an isomorphism of \emph{graded} modules
\[
\bigl(V \otimes \tla \dh(\X)\lla\bigr)^G \langle \lambda(2\rhov)
\rangle \simto \w \bigl(V \otimes {}^{(w\la)} \hspace{-1pt} \dh(\X)_{w\la} \bigr)^G \langle (w\lambda)(2\rhov)
\rangle.
\]
This property is also observed in \cite[Proposition 2.9]{LS}.
\item
Statement $(2)$ of
Lemma \ref{lem:Fourier-transform} (and the fact that $\mathbf{F}_{\alpha}$ is an involution) can also be proved directly as follows: by Corollary \ref{cor:Phi-rest}
below and the injectivity of the morphisms $\mathscr{R}^{V,\lambda}_{G,L}$ considered in this statement (see Lemmas \ref{lem:restH-injective} and \ref{lem:restriction-DX-Hom}), it is enough to prove the claim in the case $G$ has semisimple rank one, which can be treated by explicit computation (see e.g.~the proof of Lemma \ref{lem:Phi-rk1} below).
\end{enumerate}
\end{rem}

As a consequence of these constructions we also obtain the following result, which will be needed later. This result is also proved (using different methods) in \cite{Sh}.

\begin{prop}
\label{prop:freeness}

For any $\lambda \in \bX$,
the graded ${\sh}$-module 
$\dh(\X)_\la$
is free.

\end{prop}

\begin{proof}
Using isomorphisms $\mathbf{F}_w^\la$ defined above we can assume that $\lambda$ is dominant.
Then using Lemma \ref{lem:DX-Ind} it is equivalent to prove that $\Ind_B^G (\M(\la))$ is free over $\sh$.

We claim that, if $\la$ is dominant,
\begin{equation}
\label{eqn:vanishing-cohomology}
\Ind_B^G(\sym(\g/\u)[\hb] \o \C_{-\lambda}) \
 \text{ is free over $\sh$ and} \
R^{>0}\Ind_B^G(\sym(\g/\u)[\hb] \o \C_{-\lambda})=0.
\end{equation}
Indeed, it is sufficient to prove that the $\sym(\t)$-module $\Ind_B^G(\mathrm{S}(\fg/\u) \o
\bbc_{-\lambda})$ is free and that we have
$R^{>0}\Ind_B^G(\sym(\g/\u)\o  \bbc_{-\lambda})=0$. Consider
the vector bundles 
\[
q_{\wcN} : \wcN \to \B \quad \text{ and } \quad q_{\wfg}: \wfg \to \B.
\]
Here $\wcN:=G \times_B (\g/\b)^*$\index{Ntilde@$\wcN$} is the Springer resolution, and $\wfg$ is defined in \S\ref{subsec2}.
There is a natural
inclusion of vector bundles $\wcN \hookrightarrow \wfg$, and the quotient is the trivial
vector bundle $\t^* \times \B$. Hence there is a $\Z_{\geq 0}$-filtration on $(q_{\wfg})_*
\oo_{\wfg}$ (as a sheaf of $\sym(\t) \o \oo_{\B}$-modules) with
associated graded $(q_{\wcN})_* \oo_{\wcN} \o \sym(\t)$. By \cite[Theorem
2.4]{bro}, we have $\coH^{>0}(\wcN, \oo_{\wcN}(\lambda))=0$. It follows
that $\coH^{>0}(\wfg, \oo_{\wfg}(\lambda))=0$, and that $\coH^{0}(\wfg,
\oo_{\wfg}(\lambda))$ has a filtration with associated graded
$\coH^{0}(\wcN, \oo_{\wcN}(\lambda)) \o \mathrm{S}(\ft)$. Now it follows
from definitions that, for any $i \geq 0$, $\coH^{i}(\wfg,
\cO_{\wfg}(\lambda)) \cong R^i \Ind_B^G \bigl(\sym(\g/\u)\o
\C_{-\lambda} \bigr)$, which finishes the proof of \eqref{eqn:vanishing-cohomology}.

For $i \geq 0$, let $M_i := \sh \cdot \bv_{\lambda} \cdot U_{\hb}^{\leq i}(\g) \subset \M(\la)$, where $U_{\hbar}^{\leq \hdot}(\g)$ is the PBW filtration of $\uh(\g)$. Then $M_{\idot}$ is a $B$-stable and ${\sh}$-stable exhaustive filtration of $\M(\la)$, and the associated graded is isomorphic to $\mathrm{S}(\fg/\u)[\hbar] \o  \bbc_{-\lambda}$. From the second claim in \eqref{eqn:vanishing-cohomology} it follows that $\Ind_B^G (\M(\la)\rhoshift )$ has a filtration with associated graded $\Ind_B^G(\mathrm{S}(\fg/\u)[\hbar] \o  \bbc_{-\lambda})$, and then the corollary follows from the first claim in \eqref{eqn:vanishing-cohomology}.
\end{proof}

\begin{rem}
\label{rk:RInd}
The arguments in the proof of Proposition \ref{prop:freeness} also prove that, when $\la$ is dominant, we have
$R^{>0} \Ind_B^G (\M(\la)) = 0$.
\end{rem}

\subsection{Restriction to a Levi subgroup}
\label{ss:restriction-DX}

Fix a subset $I$ of the set of simple roots and
let $\fl$ be the Levi subalgebra containing $\ft$ that
has  the set $I$ as simple roots.
Note that our choice of a Borel subgroup, a maximal torus and simple root vectors for $\g$ determines a similar choice for $\l$, hence the constructions of the present section make sense both for $\g$ and for $\l$.

We put\index{nL+@$\fn_L^+$}\index{nL-@$\fn_L^-$}%
\[
\u_L:=\bigoplus_{\genfrac{}{}{0pt}{}{\alpha \in R^+}{\alpha \in \Z I}}
\g_\alpha, \qquad 
\b_L:=\t \oplus \u_L, \qquad
 \fn_L^+ := \bigoplus_{\genfrac{}{}{0pt}{}{\alpha \in R^+}{\alpha \notin
    \Z I}} \g_\alpha, \qquad 
\fn_L^- :=
\bigoplus_{\genfrac{}{}{0pt}{}{\alpha \in R^-}{\alpha \notin \Z I}}
\g_\alpha.
\]
Thus, one has a triangular decomposition
$\g=\fn^+_L \oplus \fl \oplus \fn_L^-$, and
$\u_L=\u \cap \l$ is the nilradical of the Borel subalgebra
$\b_L=\b \cap \l$ of $\fl$. Further, let
$\fp_\pm:=
\l\oplus \fn_L^\pm$ and $\b_\pm:=\b_L\oplus\fn_L^\pm$,
resp. $\u_\pm:=\u_L\oplus\fn_L^\pm$.
Thus, $\fp_\pm$ is a pair of opposite parabolic subalgebras of $\g$
such that $\fp_+\cap\fp_-=\fl$, and
$\b_\pm$  is a pair of Borel subalgebras of $\g$
such that $\b_+\cap\b_-=\b_L$, with respective nilpotent radicals $\u_{\pm}$.
Let $L,\,P_\pm,\,B_\pm,\,U_\pm, \, N^\pm_L,\,B_L,\,U_L$
be the subgroups of $G$ corresponding 
to the Lie algebras $\fl,\,\fp_\pm,\,\b_\pm,\, \u_{\pm}, \, \fn^\pm_L,\,\b_L,\,\u_L$, respectively.
By definition, we have $\X_G= G/U_+$ and $\X_L=L/U_L$. (Observe that $B_+, \, U_+, \, \b_+, \, \u_+$ coincide with the objects denoted by $B, \, U, \, \b, \, \u$ in the preceding sections.)



Now we construct a morphism of $L$-modules 
\[
r^G_L : \dh(\X_G) \to \dh(\X_L)
\]
as follows.
Note that
the right $\uh(\fl)$-action on $\uh(\g) / \u \cdot \uh(\g)$
descends to a well-defined action on
$\uh(\g) / (\u \cdot \uh(\g) + \uh(\g) 
\cdot \fn_L^-)$. Using this, from
the  diagram $\g\hookleftarrow \fp_- \twoheadrightarrow \fl$
of natural  Lie algebra morphisms,
one obtains
the following morphisms of right  $\uh(\fl)$-modules:
\beq{Uiso}
\uh(\g) / (\u{}^{\!} \cdot{}^{\!} \uh(\g)+ \uh(\g) 
{}^{\!}\cdot{}^{\!}\fn_L^-)\ \cong\
\uh(\fp_-)/(\u_L{}^{\!} \cdot{}^{\!} \uh(\fp_-) + \uh(\fp_-){}^{\!} \cdot{}^{\!} \fn_L^- )\
\cong\ \uh(\l)/\u_L{}^{\!} \cdot{}^{\!} \uh(\l).
\eeq
All the above  maps are bijections since
the linear maps
\[
\xymatrix{
\g/(\u\oplus \fn_L^-) & \fp_-/\u_- \ar[r]^-{\sim} \ar[l]_-{\sim} & \fl/\u_L
}
\]
are clearly vector space isomorphisms.
We deduce  the following chain of maps
\begin{align*}
\bigl(\C[G]\o (\uh(\g) / \u \cdot \uh(\g) ) \bigr)^U \ \longhookrightarrow
& \ \bigl(\C[G]\o (\uh(\g) / \u \cdot \uh(\g) )\bigr)^{U_L} \\
\longrightarrow & \ \Bigl(\C[G]\o \bigl( \uh(\g) / (\u \cdot \uh(\g)+\uh(\g) 
\cdot\fn_L^-) \bigr)\Bigr)^{U_L}\\
\longrightarrow & \ \bigl(\C[G]\o (\uh(\l)/\u_L \cdot
  \uh(\l) ) \bigr)^{U_L}\\
\longrightarrow & \ \bigl(\C[L]\o (\uh(\l)/\u_L \cdot
  \uh(\l) ) \bigr)^{U_L},
\end{align*}
where the third morphism is induced by \eqref{Uiso}, and
the last one is induced by restriction of functions $\C[G] \to \C[L]$.
Using isomorphism \eqref{eqn:diffX}, this allows to define the desired morphism
\[
r^G_L : \dh(\X_G)\ \cong\   
\bigl(\C[G]\o(\uh(\g) / \u \cdot \uh(\g)\bigr)^U
\to \bigl(\C[L]\o (\uh(\l)/\u_L \cdot
  \uh(\l))\bigr)^{U_L} \ \cong\ \dh(\X_L).
\]

We also define an automorphism of $L$-modules
\[
t^G_L : \dh(\X_L) \simto \dh(\X_L)
\]
as follows. The linear map $\rho_G-\rho_L$ on $\t$ extends naturally to a Lie algebra morphism $\l \to \C$, which we denote similarly. Then the assignment $\l \ni x \mapsto x + \hb(\rho_G-\rho_L)(x)$ defines a graded $\C[\hb]$-algebra automorphism $\imath^G_L$\index{iLG@$\imath^G_L$} of $\uh(\l)$, which descends to a $U_L$-equivariant automorphism of the quotient $\uh(\l) / (\u_L \cdot \uh(\l))$. Using the isomorphism $\dh(\X_L) \cong \bigl(\C[L]\o (\uh(\l)/\u_L \cdot
  \uh(\l))\bigr)^{U_L}$ as above, we obtain the wished for automorphism $t^G_L$. This morphism can also be described in more geometric terms as follows: the linear form $\rho_G-\rho_L$ on $\t$ can be considered as a character of $T$, which extends in a natural way to a character of $L$, and then descends to an invertible function $f^G_L$ on $L/U_L$. Then one can easily check that $t^G_L$ is the automorphism of $\dh(\X_L)$ sending $D$ to $(f^G_L)^{-1} \cdot D \cdot f^G_L$.
  
Finally we define the morphism of $L$-modules
\[
\mathrm{res}^G_L := t^G_L \circ r^G_L : \dh(\X_G) \to \dh(\X_L).
\]
One can easily check that $\mathrm{res}^G_L$ is $\sh$-equivariant, and $T$-equivariant for the $T=\{1\} \times T$-actions.
The following result (in which we use superscripts on the left to indicate which reductive group we consider) will be proved in \S\ref{ss:proof-restriction-DX} below. 

\begin{prop}
\label{prop:restriction-DX}
Let $\alpha \in I$. Then the following diagram commutes:
\[
\xymatrix@C=2.5cm{
\dh(\X_G) \ar[r]^-{{}^G \hspace{-1pt} \mathbf{F}_\alpha} \ar[d]_-{\mathrm{res}^G_L} & \dh(\X_G) \ar[d]^-{\mathrm{res}^G_L} \\
\dh(\X_L) \ar[r]^-{{}^L \hspace{-1pt} \mathbf{F}_\alpha} & \dh(\X_L).
}
\]
\end{prop}

As a corollary we obtain the following result. For $\lambda \in \bX$ and $V$ in $\Rep(G)$, we denote by\index{Rest3@$\mathscr{R}^{V,\la}_{G,L}$}%
\[
\mathscr{R}^{V,\la}_{G,L} : \bigl( V \otimes \tla \dh(\X_G)\lla \bigr)^G \to \bigl( V_{|L} \otimes \tla \dh(\X_L)\lla \bigr)^L
\]
the morphism induced by $\mathrm{res}^G_L$.

\begin{cor}
\label{cor:Phi-rest}
For any $\lambda \in \bX$, $V$ in $\Rep(G)$ and $w \in W_L \subset W_G$, the following diagram commutes:
\[
\xymatrix@C=2.5cm{
\bigl( V \otimes \tla \dh(\X_G)\lla \bigr)^G \ar[d]_-{\mathscr{R}^{V,\la}_{G,L}} \ar[r]^-{{}^G \hspace{-1pt} \Phi^{V,\lambda}_w} & \w \bigl( V \otimes {}^{(w\la)} \hspace{-1pt} \dh(\X_G)_{w\lambda} \bigr)^G \ar[d]^-{\w \mathscr{R}^{V,w\la}_{G,L}} \\
\bigl( V_{|L} \otimes \tla \dh(\X_L)\lla \bigr)^L \ar[r]^-{{}^L \hspace{-1pt} \Phi^{V_{|L},\lambda}_w} & \w \bigl( V_{|L} \otimes {}^{(w\la)} \hspace{-1pt} \dh(\X_L)_{w\lambda} \bigr)^L.
}
\]
\end{cor}

\subsection{Proof of Proposition \ref{prop:restriction-DX}}
\label{ss:proof-restriction-DX}

If $G$ has simply-connected derived subgroup, then so does $L$, as well. Hence we can assume that $G=G^\sc$.

In the next lemma, if $X$ is any variety we consider $\C[X][\hb]$ as the algebra of functions on $X$ with values in $\C[\hb]$. The subset $P_- \cdot P_+ \subset G$ is an open subvariety, so that if $x$ is in $\uh(\g)$ and $f$ is in $\C[P_- \cdot P_+][\hb]$, it makes sense to consider $x \cdot f \in \C[P_- \cdot P_+][\hb]$, and also the restriction $(x \cdot f)_{|L} \in \C[L][\hb]$. (Here we consider the \emph{right} action of $\uh(\g)$ on $\C[G][\hb]$ obtained by differentiating the \emph{right regular representation} of \S\ref{ss:diffX}.)

\begin{lem}
\label{lem:action-quotient}
Let $f \in \C[P_- \cdot P_+][\hb]$ be a left $N_L^-$-invariant function. Then for any $x \in \uh(\g) \cdot \fn_L^-$ we have 
\[
(x \cdot f)_{|L}=0.
\]
\end{lem}

\begin{proof}
For any $y \in \uh(\g)$ the function $y \cdot f$ is again left $N_L^-$-invariant, so that we can assume that $x \in \fn_L^-$. Then the result follows from the observation that for $g \in L$ we have $g \cdot N_L^- = N_L^- \cdot g$.
\end{proof}


In the next lemma we use the embedding $\X_L=L/U_L = P_+ / U_+ \hookrightarrow P_- \cdot P_+ / U_+ \subset G/U_+=\X_G$.

\begin{lem}
\label{lem:actionN_L^-invariants}
For any left $N_L^-$-invariant function $f \in \C[P_- \cdot P_+ / U_+][\hb]$ and any $D \in \dh(\X_G)$ we have
\[
D(f)_{|\X_L} = (r^G_L D) (f_{|\X_L}).
\]
\end{lem}

\begin{proof}
The element $D \in \dh(\X_G)$ induces a morphism $\C[P_- \cdot P_+ / U_+][\hb] \to \C[P_- \cdot P_+][\hb]$. Similarly, the restriction morphism $\C[P_- \cdot P_+/U_+][\hb] \to \C[\X_L][\hb]$ is the restriction to right $U_+$-invariants of the restriction morphism $\C[P_- \cdot P_+][\hb] \to \C[P_+][\hb]$. Hence it is enough to show that if $f \in \C[P_- \cdot P_+][\hb]$ is left $N_L^-$-invariant then the morphism $\dh(G) \to \C[L][\hb]$ sending $D'$ to $D'(f)_{|L}$ factors (via isomorphism \eqref{eqn:diffG}) through the quotient
\[
\C[G] \o \uh(\g) \to \C[G] \o \bigl( \uh(\g) / (\uh(\g) \cdot \fn_L^-) \bigr).
\]
This fact follows from Lemma \ref{lem:action-quotient}.
\end{proof}

As a corollary of Lemma \ref{lem:actionN_L^-invariants} we obtain the following description of the morphism $r^G_L$.
We denote by $\varepsilon_L : \dh(N_L^-) \to \C[\hb]$ the morphism sending a differential operator $D$ to the value at $1 \in N_L^-$ of the function $D(1_{N_L^-})$, where $1_{N_L^-}$ is the constant function with value $1$.


\begin{cor}
\label{cor:description-ur}
The morphism $r^G_L$ coincides with the composition
\[
\dh(\X_G) \hookrightarrow \dh(P_- \cdot P_+ / U_+) \xrightarrow{\sim} \dh(N_L^-) \o_{\C[\hb]} \dh(\X_L) \xrightarrow{D^- \o D \mapsto \varepsilon_L(D^-) \cdot D} \dh(\X_L)
\]
where the first morphism is restriction to the open subset $P_- \cdot P_+ / U_+ \subset \X_G$, and the second morphism uses the isomorphism $P_- \cdot P_+ / U_+ \cong N_L^- \times \X_L$ induced by the action of $N_L^-$ on $\X_G$.\end{cor}

\begin{proof}
As the action of $\dh(\X_L)$ on $\C[\X_L][\hb]$ is faithful, it is enough to check the claim after acting on any $f \in \C[\X_L]$. However one can write $f=\tilde{f}_{|\X_L}$ where $\tilde{f} = 1_{N_L^-} \o f \in \C[P_- \cdot P_+/U_+]$, and then the claim follows from Lemma \ref{lem:actionN_L^-invariants}.
\end{proof}

Finally we can give the proof of Proposition \ref{prop:restriction-DX}.

\begin{proof}[Proof of Proposition {\rm \ref{prop:restriction-DX}}]
It is enough to prove the commutativity of the following two diagrams:
\[
\xymatrix@C=1.5cm{
\dh(\X_G) \ar[r]^-{{}^G \hspace{-1pt} \mathbf{F}_\alpha} \ar[d]_-{r^G_L} & \dh(\X_G) \ar[d]^-{r^G_L} \\
\dh(\X_L) \ar[r]^-{{}^L \hspace{-1pt} \mathbf{F}_\alpha} & \dh(\X_L),
}
\qquad \qquad
\xymatrix@C=1.5cm{
\dh(\X_L) \ar[r]^-{{}^L \hspace{-1pt} \mathbf{F}_\alpha} \ar[d]_-{t^G_L} & \dh(\X_L) \ar[d]^-{t^G_L} \\
\dh(\X_L) \ar[r]^-{{}^L \hspace{-1pt} \mathbf{F}_\alpha} & \dh(\X_L).
}
\]
The commutativity of the right-hand diagram follows from Lemma \ref{lem:twist-Fourier} below, hence we only have to consider the left-hand diagram. 
Now we observe that (since the construction of partial Fourier transform is local on the base of the vector bundle) the automorphism ${}^G \hspace{-1pt} \mathbf{F}_\alpha$ extends to an automorphism of $\dh(P_- \cdot P_+ / U_+)$ denoted similarly, which makes the following diagram commutative, where vertical morphisms are induced by restriction:
\[
\xymatrix@C=2cm{
\dh(\X_G) \ar[r]^-{{}^G \hspace{-1pt} \mathbf{F}_\alpha} \ar@{^{(}->}[d] & \dh(\X_G) \ar@{^{(}->}[d] \\
\dh(P_- \cdot P_+ / U_+) \ar[r]^-{{}^G \hspace{-1pt} \mathbf{F}_\alpha} & \dh(P_- \cdot P_+ / U_+).
}
\]
Next, by construction the following diagram commutes:
\[
\xymatrix@C=2cm{
\dh(P_- \cdot P_+ / U_+) \ar[r]^-{{}^G \hspace{-1pt} \mathbf{F}_\alpha} \ar[d]_-{\wr} & \dh(P_- \cdot P_+ / U_+) \ar[d]^-{\wr} \\
\dh(N_L^-) \o_{\C[\hb]} \dh(\X_L) \ar[r]^-{\id \o {}^L \hspace{-1pt} \mathbf{F}_\alpha} & \dh(N_L^-) \o_{\C[\hb]} \dh(\X_L).
}
\]
Then the commutativity follows from Corollary \ref{cor:description-ur}.
\end{proof}

\subsection{Classical analogue}
\label{ss:Fourier-classical}

Consider the sheaf of algebras\index{AX@$\mathscr{A}_{\X}$}%
\[
\mathscr{A}_{\X} \ := \ \dd_{\hb,\X} / \hb \cdot \dd_{\hb,\X}
\]
on $\X$.
This sheaf is canonically isomorphic to $(p_{\X})_* \oo_{T^* \X}$, where $p_{\X} = T^* \X \to \X$ is the natural projection. We also set $\mathscr{A}(\X):=\Gamma(\X,\mathscr{A}_{\X})$.\index{AX2@$\mathscr{A}(\X)$} If $\alpha$ is a simple root, we define similarly the sheaf of rings $\mathscr{A}_{\mathcal{V}_\alpha}$ on $\mathcal{V}_{\alpha}$, and the ring $\mathscr{A}(\mathcal{V}_\alpha)$. In \S\ref{ss:partial-Fourier} we have defined an automorphism of $\dd_{\hb,\mathcal{V}_\alpha}$ as a sheaf of $\C[\hb]$-algebras. Specializing to $\hb=0$ we deduce an automorphism of $\mathscr{A}_{\mathcal{V}_\alpha}$. By the same arguments as in \S\ref{ss:partial-Fourier}, restriction induces an isomorphism $\mathscr{A}(\mathcal{V}_\alpha) \simto \mathscr{A}(\X)$, hence we obtain an algebra automorphism\index{Falpha2@$\mathscr{F}_\alpha$}%
\[
\mathscr{F}_\alpha : \mathscr{A}(\X) \simto \mathscr{A}(\X)
\]
such that the following diagram commutes, where vertical arrows are induced by the natural quotient morphism $\dd_{\hb,\X} \to \mathscr{A}_{\X}$:
\beq{eqn:diagram-F}
\vcenter{
\xymatrix@C=2cm{
\dh(\X) \ar[r]^-{\mathbf{F}_\alpha}_-{\sim} \ar[d] & \dh(\X) \ar[d] \\
\mathscr{A}(\X) \ar[r]^-{\mathscr{F}_\alpha}_-{\sim} & \mathscr{A}(\X).
}
}
\eeq

The following lemma will be used to deduce Theorem \ref{thm:equiv-coh-classical} from Theorem \ref{thm:loop-equivariant-version}.

\begin{lem}
\label{lem:hbar=0}

The natural morphism $\dh(\X) \to \mathscr{A}(\X)$ induces an isomorphism
\[
\dh(\X) / \hb \cdot \dh(\X) \simto \mathscr{A}(\X).
\]
In other words,
for any $V$ in $\Rep(G)$ and $\lambda \in \bX$, the morphism
\[
\bigl( V \otimes \M(\la) \bigr)^B \ \to \ \bigl( V \otimes 
\mathrm{S}(\g/\u) \otimes \bbc_{-\lambda} \bigr)^B
\]
induced by the quotient morphism $\M(\la) \to \M(\la)/(\hb \cdot \M(\la)) \cong \mathrm{S}(\g/\u) \otimes \bbc_{-\lambda}$ induces an isomorphism
\[
\bigl( V \otimes \M(\la) \bigr)^B / \hbar \ \xrightarrow{\sim}
\ \bigl( V \otimes 
\mathrm{S}(\g/\u) \otimes \bbc_{-\lambda} \bigr)^B.
\]

\end{lem}

\begin{proof}
By Lemma \ref{lem:DX-Ind}, both statements are equivalent to the fact that for any $\la \in \bX$ the morphism
\[
\Ind_B^G(\M(\la)) / \hb \to \Ind_B^G \bigl( \mathrm{S}(\g/\u) \otimes \bbc_{-\lambda} \bigr)
\]
is an isomorphism. In the case $\la$ is dominant, this property follows from the exact sequence of $B$-modules
\[
\xymatrix{
\M(\la) \ar@{^{(}->}[r]^-{\hb} & \M(\la) \ar@{->>}[r] & \mathrm{S}(\g/\u) \otimes \bbc_{-\lambda}
}
\]
and the cohomology vanishing $R^1 \Ind_B^G(\M(\la))=0$,
see Remark \ref{rk:RInd}. Since $\mathbf{F}_\alpha$ induces an isomorphism $\dh(\X)_\la \simto \dh(\X)_{s_\alpha \la}$ (see Lemma \ref{lem:Fourier-transform}), using diagram \eqref{eqn:diagram-F} we deduce the general case from the case $\la$ is dominant.
\end{proof}

It follows in particular from \eqref{eqn:diagram-F} and Lemma \ref{lem:hbar=0} (using Lemma \ref{lem:Fourier-transform}) that the assignment $s_\alpha \mapsto \mathscr{F}_\alpha$ defines an action of $W$ on $\mathscr{A}(\X)$ by algebra automorphisms, which we denote by $w \mapsto \mathscr{F}_w$.\index{Fw@$\mathscr{F}_w$} Moreover, for any $w \in W$ and $\la \in \bX$, $\mathscr{F}_w$ defines an isomorphism of $\mathrm{S}(\t)$-modules $\mathscr{A}(\X)_\la \simto \w \mathscr{A}(\X)_{w\la}$.


\subsection{Complementary results on the structure of $\dd(\X)$}
\label{ss:complements-DX}

In this subsection we observe that Lemma \ref{lem:hbar=0} has some interesting consequences on the structure of $\dd(\X)$. These results will not be used in the rest of the paper.

We begin with the following direct consequence of Lemma \ref{lem:hbar=0} (using the natural isomorphism $\mathscr{A}(\X) \cong \C[T^* \X]$), which appears to be new.

\begin{cor}
\label{cor:grD}
The canonical graded algebra morphism
$\gr\dd(\X)\to \k[T^*\X]$ is an isomorphism.
\end{cor}

%

This corollary allows to give new proofs of some results of \cite{LS} and \cite{bbp}. These proofs use the following simple lemma.

\begin{lem}
\label{lem:T*Xfg}
The $\C$-algebra $\C[T^* \X]$ is finitely generated, hence noetherian.
\end{lem}

\begin{proof}
First we observe that there exists a canonical isomorphism $T^* \X \cong G \times_U (\g/\u)^*$. Hence $T^* \X$ is a $T$-torsor over $\wfg$, which implies that we have a natural algebra isomorphism
\[
\C[T^*\X] \ \cong \ \bigoplus_{\la \in \bX} \, \Gamma \bigl( \wfg, \oo_{\wfg}(\la) \bigr).
\]

By the same observation, there exists a natural morphism $T^* \X \to \g^* \times_{\t^* /W} \t^*$, $g \times_U \eta \mapsto (g \cdot \eta, \eta_{|\t})$, which induces an algebra morphism $\mathrm{S}(\g) \o_{\mathrm{S}(\t)^W} \mathrm{S}(\t) \to \C[T^* \X]$.
Note also that
if $\la \in \bX^+$, then $\Gamma(\B,\oo_{\B}(\lambda))$ identifies naturally with a subspace of $\C[\X]$, hence also defines a subspace $X_\la$ of $\C[T^* \X]$ using the projection $T^* \X \to \X$.

Let $\lambda_1, \cdots, \lambda_r$ be a finite collection of dominant weights such that $\bX^+= \sum_{i=1}^r \Z_{\geq 0} \lambda_i$. We claim that $\C[T^* \X]$ is generated (as an algebra) by $\mathrm{S}(\g) \o_{\mathrm{S}(\t)^W} \mathrm{S}(\t)$ 
and the $G$-modules $X_{\la_i}$ for $i \in \{1, \cdots, r\}$, together with the images of these subspaces under the automorphisms $\mathscr{F}_w$ for all $w \in W$. This claim clearly implies the statement of the lemma.

To prove the claim we first observe that if $\la \in \bX^+$ the morphism
\[
\bigl( \mathrm{S}(\g) \o_{\mathrm{S}(\t)^W} \mathrm{S}(\t) \bigr) \o X_\la \to \Gamma \bigl( \wfg, \oo_{\wfg}(\la) \bigr)
\]
induced by the product in $\C[T^*\X]$
is surjective.
Indeed by the graded Nakayama lemma it is enough to prove surjectivity after tensoring with the trivial $\mathrm{S}(\g) \o_{\mathrm{S}(\t)^W} \mathrm{S}(\t)$-module, hence it is also enough to prove surjectivity after tensoring with the trivial $\mathrm{S}(\t)$-module. However the arguments in the proof of Proposition \ref{prop:freeness} imply that the natural morphism
\[
\C \o_{\mathrm{S}(\t)} \Gamma \bigl( \wfg, \oo_{\wfg}(\la) \bigr) \to \Gamma \bigl( \wcN, \oo_{\wcN}(\la) \bigr)
\]
induced by restriction is an isomorphism, hence the latter surjectivity statement follows from \cite[Proposition 2.6]{bro}.
From this observation, together with the fact that if $\la,\mu \in \bX^+$ the natural morphism
\[
\Gamma(\B,\oo_{\B}(\lambda)) \o \Gamma(\B,\oo_{\B}(\mu)) \to \Gamma(\B,\oo_{\B}(\lambda+\mu))
\]
is surjective (see e.g.~\cite[Theorem 3.1.2]{bk}), it follows that the subalgebra of $\C[T^*\X]$ generated by $\mathrm{S}(\g) \o_{\mathrm{S}(\t)^W} \mathrm{S}(\t)$ and the modules $X_{\la_i}$ contains
\[
\bigoplus_{\la \in \bX^+} \, \Gamma \bigl( \wfg, \oo_{\wfg}(\la) \bigr).
\]
The claim follows, using the $W$-action and the fact that every weight is $W$-conjugate to a dominant weight.
\end{proof}

In the following corollary, statement (1)
is due to Levasseur--Stafford (see \cite[Theorem 3.3]{LS}). The present proof is suggested in \cite[Remark 3.4]{LS}, but the authors didn't have Corollary \ref{cor:grD} to complete the argument. Statement (2) is due to Bezrukavnikov--Braverman--Positselskii (see \cite[Theorem 1.1]{bbp}).

\begin{cor}
The algebra $\dd(\X)$ is
\begin{enumerate}
\item
Auslander--Gorenstein and Cohen--Macaulay;
\item
left and right noetherian.
\end{enumerate}
\end{cor}

\begin{proof}
(1) Choose some elements $\muv_i \in X_*(T)$, $i=1, \cdots s$, such that if $\la(\muv_i)=0$ for all $i$ then $\lambda=0$, and set $\|\lambda \|:=\sum_i |\la(\muv_i)|$.
Consider the filtration on $\dd(\X)$ defined for $n \geq 0$ by
\[
F_n \dd(\X) := \dd(\X)^{\leq n} \cap \left( \bigoplus_{\| \la \| \leq n} \dd(\X)_\la \right),
\]
where $\dd(\X)^{\leq \hdot}$ is the filtration by order of differential operators. Then this filtration is connected (in the sense that $F_0 \dd(\X)=\C$) and the associated graded is canonically isomorphic to $\C[T^* \X]$ by Corollary \ref{cor:grD}. Hence the claim follows from \cite[Corollary 0.3]{YeZ}, using the simplicity of $\dd(\X)$ proved in \cite[Proposition 3.1]{LS} and Lemma \ref{lem:T*Xfg}.

(2)
It is enough to prove that $\dh(\X)$ is (left and right) noetherian. But this follows from \cite[Lemma 8.2]{atv} (for $g=\hb$) and Lemma \ref{lem:T*Xfg}. 
\end{proof}

\section{Morphisms between asymptotic universal Verma modules}
\label{sec:morphisms}

In this section we will use the following convention. If $\mathcal{P}(\mu)$ is a property depending on $\mu \in \t^*$, we will say that $\mathcal{P}(\mu)$ holds ``for $\mu \in \t^*$ sufficiently large'' is there exists $n \in \Z_{\geq 0}$ such that $\mathcal{P}(\mu)$ holds for any $\mu \in \t^*$ satisfying $|\mu(\alv)| \geq n$ for all simple roots $\alpha$. We will use similar conventions for subsets of $\t^*$ (e.g.~$\bX$).

\subsection{Reminer on Verma modules}
\label{ss:reminder-Verma}

The results in this subsection are well known, see e.g.~\cite{tv, ev}. We include (short) proofs for the reader's convenience. We will use the ``dot-action'' defined by $w \bullet \mu = w(\mu+\rho)-\rho$.

For any $\mu \in \t^*$ we consider the Verma module\index{Verma@$\V(\mu)$}%
\[
\V(\mu):=U(\g) \o_{U(\b)} \C_\mu
\]
(a left $U(\g)$-module). We also set $1_\mu := 1 \o 1 \in \V(\mu)$.\index{1mu@$1_\mu$} We set $\V(\mu)_- := U(\u^-) \u^- \cdot 1_\mu$, so that we have $\V(\mu)=\C \cdot 1_\mu \oplus \V(\mu)_-$.

If $V$ is in $\Rep(G)$, $\la \in \bX$, $\mu \in \t^*$, and if $\phi \in \Hom_{U(\g)}(\V(\mu),V \o \V(\mu-\la))$, then one can write
\[
\phi(1_\mu) = u \o 1_{\mu-\la} + x
\]
for unique $u \in V_\la$ and $x \in V \otimes \V(\mu-\la)_-$. We set $u:=\mathrm{E}^{V,\la}_\mu(\phi)$;\index{Exp@$\mathrm{E}^{V,\la}_\mu$} it is called the \emph{expectation value} of $\phi$.

Let $\mu \in \t^*$ and let $\alpha$ be a simple root. If $n:= \mu(\alv)+1 \in \Z_{\geq 0}$, then as in \cite[\S 1.4]{hu} there exists a unique embedding of $U(\g)$-modules
\[
\V(s_{\alpha} \bullet \mu) \hookrightarrow \V(\mu)
\]
which sends $1_{s_\alpha \bullet \mu}$ to $\frac{(f_\alpha)^n}{n!} \cdot 1_\mu$. Iterating, we perform the following construction. Let $\mu \in \bX^+ -\rho$, and let $w \in W$. Choose a reduced decomposition $w=s_k \cdots s_1$, where $\alpha_1, \cdots, \alpha_k$ is a sequence of simple roots (possibly with repetitions) and $s_i$ is the reflection associated with $\alpha_i$. Then for any $i=1, \cdots, k$, $n_i:=(s_{i-1} \cdots s_1 \bullet \mu)(\alv_{i})+1 \in \Z_{\geq 0}$. Moreover the collection $(n_1, \cdots, n_k)$ and the product $(f_{\alpha_k})^{n_k} \cdots (f_{\alpha_1})^{n_1}$ do not depend on the reduced decomposition of $w$ (see \cite[Lemma 4]{tv}). Hence there exists a unique embedding of graded right $U(\g)$-modules\index{iota@$\iota_{\mu}^w$}%
\[
\iota_{\mu}^w : \V(w \bullet \mu) \hookrightarrow \V(\mu)
\quad\text{such that}\quad \iota_{\mu}^w(1_{w \bullet \mu}) =
\frac{(f_{\alpha_{k}})^{n_k}}{n_k !} \cdots \frac{(f_{\alpha_{1}})^{n_1}}{n_1 !} \cdot 1_{\mu}.
\]

\begin{lem}
\label{lem:intertwiners}
Let $V$ in $\Rep(G)$, and $\lambda \in \bX$.
\begin{enumerate}
\item For $\mu \in \t^*$ sufficiently large, the morphism 
\[
\mathrm{E}^{V,\lambda}_\mu : \Hom_{U(\g)} \bigl( \V(\mu),V \o \V(\mu-\la) \bigr) \to V_\la
\] 
is an isomorphism.
\item For $w \in W$ and $\mu \in \bX^+$ sufficiently large, the morphism
\[
\Hom_{U(\g)} \bigl( \V(w \bullet \mu),V \o \V(w \bullet (\mu-\la)) \bigr) \to \Hom_{U(\g)} \bigl( \V(w \bullet \mu),V \o \V(\mu-\la) \bigr)
\]
defined by $\phi \mapsto (\id_V \o \iota^w_{\mu-\lambda}) \circ \phi$ is an isomorphism
\end{enumerate}
\end{lem}

\begin{proof}
(1) It is well known (see e.g.~\cite[Theorem 3.6]{hu}) that there exists an enumeration $\nu_1, \cdots, \nu_k$ of the $T$-weights of $V$ and a filtration (as a $U(\g)$-module)
\[
\{0\} = M_0 \subset M_1 \subset \cdots \subset M_k=V \o \V(\mu-\lambda)
\]
where for all $i = 1, \cdots, k$, $M_i/M_{i-1} \cong V_{\nu_i} \o \V(\mu - \lambda + \nu_i)$.
Consider, for any
$i=1, \cdots, n$, the associated exact sequence
 \begin{multline*}
  0 \to \Hom_{U(\g)} \bigl( \V(\mu), M_{i-1} \bigr) \to \Hom_{U(\g)} \bigl( \V(\mu), M_{i} \bigr)\\
  \to \Hom_{U(\g)} \bigl( \V(\mu), \V(\mu - \lambda + \nu_i) \o
  V_{\nu_i} \bigr) \to \Ext^1_{U(\fg)} \bigl( \V(\mu), M_{i-1} \bigr).
 \end{multline*}
It is also well known that
\[
 \Hom_{U(\g)} \bigl( \V(\eta_1), \V(\eta_2) \bigr) = \Ext^1_{U(\g)} \bigl( \V(\eta_1), \V(\eta_2) \bigr)=0
 \]
unless $\eta_1 \in W \bullet \eta_2$, and $\Hom_{U(\g)} \bigl( \V(\eta), \V(\eta) \bigr)=\bbc$. Now if $\mu \in \t^*$ is sufficiently large, the property $\mu \in W \bullet (\mu-\lambda + \nu_i)$ implies that $w=1$ and $\lambda=\nu_i$, and the result follows.

(2) First we remark that the morphism under consideration is indeed well defined if $\mu \in \bX^+$ is sufficiently large. As $\iota_{\mu-\la}^w$ is injective, our morphism is injective. Hence it is enough to prove that both sides have the same dimension for $\mu \in \bX^+$ sufficiently large.

By (1), if $\mu$ is sufficiently large the left-hand side is isomorphic to $V_{w \lambda}$. Now similar arguments, using the property that that if $\eta \in \bX^+-\rho$ then 
\[
\Hom_{U(\g)}(\V(w \bullet \eta), \V(\eta))=\C,
\] 
show that, again if $\mu$ is sufficiently large, the right-hand side is isomorphic to $V_{\lambda}$. As $\dim_{\bbc}(V_{\lambda}) = \dim_{\bbc}(V_{w \lambda})$, this finishes the proof.
\end{proof}

\subsection{Asymptotic Verma modules}
\label{ss:asymptotic-Verma}

If $\mu \in \t^*$, we denote by $\kh\llangle \mu \rrangle$\index{Chbarmu@$\kh\llangle \mu \rrangle$} the graded right $\sh$-module where any $t \in \t$ acts by multiplication by $\hb \mu(t)$. 

For $\mu \in \t^*$ we define the graded right $\uh(\g)$-module\index{Mlambda2@$\MM(\mu)$}%
\[
\MM(\mu):=\kh\llangle 0 \rrangle \o_{\sh} \M(\mu).
\]
We also set $\bvv_{\mu} := 1 \o \bv_{\mu} \in \MM(\mu)$.\index{vlambda2@$\bvv_{\mu}$}
Any $t \in \t$ acts on $\bvv_{\mu}$ by multiplication by $\hb (\mu+\rho)(t)$, and $\MM(\mu)$ admits a basis (as a $\C[\hb]$-module) such that, for any vector $v$ in this basis, there exists $\gamma \in \Z_{\geq 0} R^+$ such that $v \cdot t=(\mu+\rho+\gamma)(t)$ for any $t \in \t$.
We set $\MM(\mu)_-:=\bvv_\mu \cdot \u^- \uh(\u^-)$. Then we have
\[
\MM(\mu) = \kh \cdot \bvv_{\mu} \oplus \MM(\mu)_-.
\]

Note also that for any $\lambda,\mu \in \t^*$ there exists a natural isomorphism
\begin{equation}
\label{eqn:specialization-M}
\kh\llangle \mu \rrangle \o_{\sh} \M(\lambda) \ \cong \ \MM(\lambda+\mu)
\end{equation}
sending $1 \o \bv_{\lambda}$ to $\bvv_{\lambda+\mu}$.

It will be convenient to invert $\hb$. To simplify notation we set\index{Mlambda3@$\MMloc(\mu)$}%
\[
\MMloc(\mu):= \C[\hb,\hb^{-1}] \o_{\C[\hb]} \MM(\mu), \quad \uloc(\g):= \C[\hb,\hb^{-1}] \o_{\kh} \uh(\fg).
\]
We define $\MMloc(\mu)_-$ in the obvious way. There exists an isomorphism of graded $\C[\hb,\hb^{-1}]$-algebras
\[
U(\g) \o_{\C} \C[\hb,\hb^{-1}] \ \simto \ \uloc(\g)^{\mathrm{op}}, \qquad \g \ni x \mapsto - \frac{1}{\hb} x.
\]
(In the left-hand side, $U(\g)$ is in degree $0$.) Using this isomorphism one can regard $\MMloc(\mu)$ as a (left) module over the algebra $U(\g) \o_{\C} \C[\hb,\hb^{-1}]$. With this structure it is isomorphic to
$\V(-\mu-\rho) \o_{\C} \C[\hb,\hb^{-1}]$.

Let $\mu \in \t^*$ and let $\alpha$ be a simple root. If $n:= -\mu(\alv) \in \Z_{\geq 0}$, then as in \S\ref{ss:reminder-Verma} there exists a unique embedding of graded $\uloc(\g)$-modules
\[
\MMloc(s_{\alpha} \mu) \hookrightarrow \MMloc(\mu)
\]
which sends $\bvv_{s_\alpha \mu}$ to $\bvv_{\mu} \cdot \frac{(-f_\alpha)^n}{\hb^n n!}$. Iterating, we perform the following construction. Let $\mu \in \bX^-$, and let $w \in W$. Choose a reduced decomposition $w=s_k \cdots s_1$, where $s_i$ is the reflection associated with the simple root $\alpha_i$. Then for any $i=1, \cdots, k$, $n_i:=- ( s_{i-1} \cdots s_1 \mu, \alv_{i} ) \in \Z_{\geq 0}$, hence we can define
\[
\bvv_\mu^w := \bvv_{\mu} \cdot \frac{(-f_{\alpha_{i_1}})^{n_1}}{\hbar^{n_1} \cdot n_1 !} \cdots \frac{(-f_{\alpha_{i_k}})^{n_k}}{\hbar^{n_k} \cdot n_k !} \ \in \MMloc(\lambda),
\]
and there exists a unique embedding of graded right $\uloc(\g)$-modules\index{imuw@$\mathbf{i}_{\mu}^w$}%
\[
\mathbf{i}_{\mu}^w : \MMloc(w \mu) \hookrightarrow \MMloc(\mu)
\quad\text{such that}\quad \mathbf{i}_{\mu}^w(\bvv_{w \mu}) =
\bvv_{\mu}^w.
\]

\subsection{Specialization}
\label{ss:specialization}

Let $\la \in \bX$.
Recall that in \S\ref{subsec1} we have defined a structure of $B$-module and of $(\uh(\b),\uh(\g))$-bimodule on $\M(\la)$. In particular, if $V$ is in $\Rep(G)$, we have a (diagonal) $B$-module structure on $V \o \M(\la)$.

\begin{lem}
\label{lem:morphisms-Verma-v}
For $V$ in $\Rep(G)$ and $\la \in \bX$, 
the assignment $\phi \mapsto \phi(\bv_0)$, 
induces an isomorphism
\[
\Hom_{(\sh,\uh(\g))}(\M(0),V \otimes
\M(\la)) \simto \bigl(V \otimes \M(\la) \bigr)^B.
\]
\end{lem}

\begin{proof}
By definition, $\M(0)$ is an induced right $U_{\hbar}(\fg)$-module. Hence we have
\[
\Hom_{(\sh,\uh(\g))}(\M(0),V \otimes \M(\la))\  \cong \  \Hom_{({\sh}, \uh(\b))}({\sh}\llangle 0 \rrangle,V \otimes \M(\la)).
\]
It follows that the left-hand side is isomorphic to the space of vectors $x$ in $V \otimes \M(\la)$ which satisfy
\[
x \cdot b=(b+\hbar \rho(b)) \cdot x
\]
for all $b \in \b \subset \uh(\b)$. This is exactly the condition for being $U(\b)$-invariant (i.e.~$B$-invariant), see \S\ref{subsec1}.
\end{proof}

If $\mu \in \t^*$,
we denote by $P \mapsto P(\mu)$ the unique algebra morphism $\sh \to \C[\hb]$ which sends $x \in \t$ to $\hb \mu(x)$. Note that if $P \in \sh$ and if for some $w \in W$ we have $P(\mu)=0$ for all $\mu \in w (\bX^-)$ sufficiently large, then $P=0$.

If $\la,\mu \in \t^*$, we denote by\index{Spmu@$\uSp_\mu$}%
\[
\uSp_\mu : \M(\la) \to \MMloc(\la + \mu)
\]
the natural morphism induced by isomorphism \eqref{eqn:specialization-M}. The following lemma follows from the fact that $\M(\la)$ is a free $\sh$-module and the remarks above on the map $P \mapsto P(\mu)$.

\begin{lem}
\label{lem:specialization}
Let $\la \in \bX$ and $w \in W$.
Let $m \in \M(\la)$, and assume that $\uSp_\mu(m)=0$ for all $\mu \in w(\bX^-)$ sufficiently large. Then $m=0$.
\end{lem}

We will derive several useful corollaries.

\begin{cor}
\label{cor:fixB}
Let $V$ in $\Rep(G)$, $\la \in \bX$ and $w \in W$.
Let $m \in V \o \M(\la)$, and assume that for all $\mu \in w(\bX^-)$ sufficiently large we have
\[
(\id_V \o \uSp_\mu)(m) \cdot b = \hb(\mu+\rho)(b) \cdot (\id_V \o \uSp_\mu)(m)
\]
for all $b \in \b$.
Then there exists a unique $\phi \in  \Hom_{(\sh,\uh(\g))}(\M(0),V \otimes \M(\la))$ such that $m=\phi(\bv_0)$.
\end{cor}

\begin{proof}
By Lemma \ref{lem:morphisms-Verma-v} we only have to prove that $m$ is fixed by $B$, i.e.~that for any $b \in \b$ we have 
\[
m \cdot b - (b + \hb \rho(b)) \cdot m =0
\]
(see the proof of Lemma \ref{lem:morphisms-Verma-v}).
However, our assumption implies that this vector is annihilated by $(\id_V \o \uSp_\mu)$ for all $\mu \in w(\bX^-)$ sufficiently large. Hence we conclude using Lemma \ref{lem:specialization}.
\end{proof}

If $\phi \in \Hom_{(\sh,\uh(\g))}(\M(0),V \otimes \M(\la))$ we denote by $\Sp_\mu(\phi) : \MMloc(\mu) \to V \o \MMloc(\la + \mu)$ the morphism obtained by tensoring with $\C[\hb]\llangle \mu \rrangle$, using isomorphism \eqref{eqn:specialization-M}, and inverting $\hb$. This construction induces a morphism of $\C[\hb,\hb^{-1}]$-modules\index{Spmu2@$\Sp_\mu$}%
\[
\Sp_\mu : \Hom_{(\sh,\uh(\g))}(\M(0),V \otimes \M(\la)) \to \Hom_{-\uloc(\g)}(\MMloc(\mu),V \otimes \MMloc(\la+\mu)).
\]

\begin{cor}
\label{cor:sp-morphism}
Let $V$ in $\Rep(G)$, $\la \in \bX$, $w \in W$, and $\phi \in \Hom_{(\sh,\uh(\g))}(\M(0),V \otimes \M(\la))$. If $\Sp_\mu(\phi)=0$ for all $\mu \in w(\bX^-)$ sufficiently large, then $\phi=0$.
\end{cor}

\begin{proof}
This follows from the commutativity of the following diagram
\[
\xymatrix@C=2cm{
\Hom_{(\sh,\uh(\g))}(\M(0),V \otimes \M(\la)) \ar@{^{(}->}[r]^-{\phi \mapsto \phi(\bv_0)} \ar[d]_-{\Sp_\mu} & V \o \M(\la) \ar[d]^-{\id_V \o \uSp_\mu} \\
\Hom_{-\uloc(\g)}(\MMloc(\mu),V \otimes \MMloc(\la+\mu)) \ar[r]^-{\phi \mapsto \phi(\bvv_\mu)} & V \otimes \MMloc(\la+\mu)
}
\]
and Lemma \ref{lem:specialization}.
\end{proof}

\subsection{Intertwining operators}

Let $V$ in $\Rep(G)$, and let $\phi \in \Hom_{-\uloc(\g)}(\MMloc(\mu), V \o \MMloc(\lambda+\mu))$ be an intertwining operator. Then we can write
\[
\phi(\bvv_\mu) = u \o f(\hb) \cdot \bvv_{\lambda+\mu} + x
\]
for unique $u \in V_\lambda$, $f(\hb) \in \C[\hb,\hb^{-1}]$ and $x \in V \o \MM(\lambda+\mu)_-$. The vector $u \o f(\hb) \in V_\lambda \o \C[\hb, \hb^{-1}]$ is called the \emph{expectation value} of $\phi$. This way we have defined a morphism of graded $\C[\hb,\hb^{-1}]$-modules\index{Exp2@$\mathsf{E}^{V,\la}_{\mu}$}%
\[
\mathsf{E}^{V,\la}_{\mu} : \Hom_{-\uloc(\g)} \bigl( \MMloc(\mu), V \o \MMloc(\lambda+\mu) \bigr) \to V_\lambda \o \C[\hb, \hb^{-1}].
\]

\begin{lem}
\label{lem:intertwiners-quantum}
Let $V$ in $\Rep(G)$, and $\lambda \in \bX$.
\begin{enumerate}
\item For $\mu \in \t^*$ sufficiently large, $\mathsf{E}^{V,\lambda}_\mu$ is an isomorphism.
\item For $w \in W$ and $\mu \in \bX^-$ sufficiently large, the morphism
\[
\Hom_{-\uloc(\g)} \bigl( \MMloc(w \mu),V \o \MMloc(w (\la+\mu)) \bigr) \to \Hom_{-\uloc(\g)} \bigl( \MMloc(w \mu),V \o \MMloc(\la+\mu) \bigr)
\]
defined by $\phi \mapsto (\id_V \o \mathbf{i}^w_{\la+\mu}) \circ \phi$ is an isomorphism
\end{enumerate}
\end{lem}

\begin{proof}
By the remarks above we have an isomorphism $\uloc(\g)^{\mathrm{op}} \cong U(\g) \o \C[\hb,\hb^{-1}]$, which induces an isomorphism
\[
\Hom_{-\uloc(\g)}(\MMloc(\mu),V' \o \MMloc(\nu)) \cong \Hom_{U(\g)}(\V(-\mu-\rho), V' \o \V(-\nu-\rho)) \o \C[\hb,\hb^{-1}]
\]
for any $V'$ in $\Rep(G)$ and $\mu,\nu \in \t^*$. Moreover, under these isomorphisms the morphisms considered in the lemma are induced by those of Lemma \ref{lem:intertwiners}. Hence the claims follow from Lemma \ref{lem:intertwiners}.
\end{proof}

Fix $V$ in $\Rep(G)$, $\lambda \in \bX$ and $w \in W$. For any $\mu \in \bX^-$ sufficiently large, we define the morphism of graded $\bbc[\hbar,\hbar^{-1}]$-modules\index{Psi@$\Psi^{V,\lambda}_{w,\mu}$}%
\[
\Psi^{V,\lambda}_{w,\mu} : \Hom_{-\uloc(\g)} \bigl( \MMloc(\mu), V \o \MMloc(\mu+\lambda) \bigr) \to  
\Hom_{-\uloc(\g)} \bigl( \MMloc(w\mu), V \o \MMloc(w(\mu+\lambda)) \bigr),
\]
in such a way that for any $\phi \in \Hom_{-\uloc(\g)}(\MMloc(\mu), V \o \MMloc(\mu+\lambda) \bigr)$ we have
\[
\phi \circ \mathbf{i}_{\mu}^w \ = \ (\mathbf{i}_{\lambda+\mu}^w \otimes \id_V) \circ \Psi^{V,\lambda}_{w,\mu}(\phi).
\]
This morphism is well defined by Lemma \ref{lem:intertwiners-quantum}.

Let now $\alpha$ be a simple root, and consider the case $w=s_{\alpha}$. Then, even if $\mu$ is not in $\bX^-$, one can define a morphism 
\[
\Psi^{V,\lambda}_{s_{\alpha},\mu} : \Hom_{-\uloc(\g)} \bigl( \MMloc(\mu), V \o \MMloc(\mu+\lambda) \bigr) \to  
\Hom_{-\uloc(\g)} \bigl( \MMloc(s_\alpha \mu), V \o \MMloc(s_\alpha(\mu+\lambda)) \bigr)
\]
with the same properties as above as soon as $\mu \in \bX$ is sufficiently large and $\mu(\alv) < 0$. With this extension of the definition, consider again some $w \in W$, and let $w=s_k \cdots s_1$ be a reduced decomposition. Note that if $\mu \in \bX^-$, for any $i=1, \cdots, k$ we have $(s_{i-1} \cdots s_1 \mu) (\alv_{i}) \leq 0$. Then, by definition, for $\mu \in \bX^-$ sufficiently large we have
\begin{equation}
\label{eqn:composition-Psi}
\Psi^{V,\lambda}_{w,\mu} = \Psi^{V,s_{k-1} \cdots s_1(\lambda)}_{s_k,s_{k-1} \cdots s_1(\mu)} \circ \cdots \circ \Psi^{V,s_1(\lambda)}_{s_{2},s_1(\mu)} \circ \Psi^{V,\lambda}_{s_1,\mu}.
\end{equation}

\subsection{Simple reflections}

In this subsection we fix a simple root $\alpha$, and set $s=s_{\alpha}$.

\begin{prop}
\label{prop:definition-Theta}

Let $V$ in $\Rep(G)$ and $\lambda \in \bX$. There exists a unique morphism of graded ${\sh}$-modules
\[
\Theta_s^{V,\lambda} : \Hom_{(\sh,\uh(\g))} \bigl( \M(0),V \otimes \M(\la) \bigr) \langle \la(2\rhov) \rangle \to  
{}^s \hspace{-1pt} \Hom_{(\sh,\uh(\g))} \bigl( \M(0),V \otimes \M(s \lambda) \bigr) \langle (s \la)(2\rhov) \rangle
\]
such that for any 
$\mu \in \bX$ sufficiently large such that $\mu(\alv) < 0$ we have
\begin{multline}
\label{eqn:specialization-Theta}
\Sp_{s\mu} \circ \Theta_s^{V,\lambda} = \\
\begin{cases}
(-\hbar)^{\lambda(\alv)}(-\mu(\alv))(-\mu(\alv)-1) \cdots (-\mu(\alv)-\lambda(\alv)+1) \cdot  \Psi^{V,\lambda}_{s,\mu} \circ \Sp_\mu & \text{if }\  \lambda(\alv) \geq 0, \\
\frac{1}{(-\hbar)^{-\lambda(\alv)}(-\mu(\alv)-\lambda(\alv)) \cdots (-\mu(\alv)+1)} \cdot \Psi^{V,\lambda}_{s,\mu} \circ \Sp_\mu & \text{if }\ \lambda(\alv) \leq 0.
\end{cases}
\end{multline}

\end{prop}

\begin{proof}
Unicity follows from Corollary \ref{cor:sp-morphism}. Let us prove existence. Choose an enumeration of positive roots $\alpha_1, \cdots, \alpha_n$ such that $\alpha_n=\alpha$ and, for any $i=1, \cdots, n-1$, a non-zero vector $f_i \in \g_{-\alpha_i}$. (These choices can be arbitrary.) For any multi-index $\underline{k}=(k_1, \cdots, k_{n-1})$ we set $f^{\underline{k}}:=f_1^{k_1} \cdots f_{n-1}^{k_{n-1}} \in \uloc(\g)$. Then any vector of $\M(\la)$ can be written in a unique way in the form
\[
\sum_{\underline{k},i} P_{\underline{k},i} \otimes f^{\underline{k}} f_{\alpha}^i
\]
where $P_{\underline{k},i} \in {\sh}$ (and where only finitely many terms are non-zero). Choose also a basis $\{u_j,\, j\in J\}$ of $V$. To simplify notation, for $\mu \in \ft^*$ we set $n(\mu):=-\mu(\alv)$.

For $P$ in ${\sh}$ we set
\[
\genfrac{(}{)}{0pt}{}{P}{m}_{\hbar} \ := \ \frac{P(P-\hbar) \cdots (P-(m-1)\hbar)}{m(m-1) \cdots 1}.
\]
With this notation, for $\mu \in \bX$ such that $n(\mu) \geq m$ we have
\beq{eqn:binomial-hb}
\genfrac{(}{)}{0pt}{}{\alv}{m}_{\hbar} (s\mu) = \hbar^{m} \genfrac{(}{)}{0pt}{}{n(\mu)}{m}.
\eeq

Let $\phi : \M(0) \to V \otimes \M(\la)$ be a morphism of bimodules. Write
\[
\phi(\bv_0) = \sum_{j,\underline{k},i} u_j \otimes P_{j,\underline{k},i} \otimes f^{\underline{k}} f_{\alpha}^i
\]
where $P_{j,\underline{k},i} \in {\sh}$. Then for $\mu \in \bX$ such that $n(\mu)\geq 0$ we have
\[
\Sp_\mu(\phi)(\bvv_{\mu}) = \sum_{j,\underline{k},i} u_j \otimes P_{j,\underline{k},i}(\mu) \otimes f^{\underline{k}} f_{\alpha}^i \quad \in V \o \MMloc(\la+\mu),
\]
hence
\begin{multline*}
\Sp_\mu(\phi)(\bvv_{\mu}^s) = \Bigl( \sum_{j,\underline{k},i} u_j \otimes
P_{j,\underline{k},i}(\mu) \otimes f^{\underline{k}} f_{\alpha}^i \Bigr)
\cdot \frac{(-f_\alpha)^{n(\mu)}}{\hbar^{n(\mu)} n(\mu)!}\\  
= \frac{(-1)^{n(\mu)}}{\hbar^{n(\mu)} n(\mu)!} \left(
\sum_{m=0}^{n(\mu)} (-\hbar)^{m} \genfrac{(}{)}{0pt}{}{n(\mu)}{m}
\sum_{j,\underline{k},i} f_\alpha^{m} \cdot u_j \otimes
P_{j,\underline{k},i}(\mu) \otimes f^{\underline{k}}
f_{\alpha}^{i+n(\mu)-m} \right).
\end{multline*}
By Lemma \ref{lem:intertwiners-quantum}(2) (or more precisely an obvious generalization, when $w=s$, to the case $\mu(\alv) \leq 0$ instead of $\mu \in \bX^-$), if $\mu$ if sufficiently large then the terms for which $i+n(\mu)-m < n(\lambda+\mu)$ vanish. Hence we obtain that in this case $\Sp_\mu(\phi)(\bvv_{\mu}^s)$ equals
\[
\frac{(-1)^{n(\mu)}}{\hbar^{n(\mu)} n(\mu)!} \left(\sum_{\genfrac{}{}{0pt}{}{j,\underline{k},i,m}{0 \leq m \leq i-n(\lambda)}} (-\hbar)^{m} \genfrac{(}{)}{0pt}{}{n(\mu)}{m}  f_\alpha^{m} \cdot u_j \otimes P_{j,\underline{k},i}(\mu) \otimes f^{\underline{k}} f_{\alpha}^{i+n(\mu)-m} \right),
\]
i.e.~equals
\[
\frac{(-\hbar)^{n(\lambda)}n(\lambda+\mu)!}{n(\mu)!} (\id_V \o \mathbf{i}_{\lambda+\mu}^s) \Bigl( \sum_{\genfrac{}{}{0pt}{}{j,\underline{k},i,m}{0 \leq m \leq i-n(\lambda)}} (-\hbar)^{m} \genfrac{(}{)}{0pt}{}{n(\mu)}{m}  f_\alpha^{m} \cdot u_j  
\otimes P_{j,\underline{k},i}(\mu) \otimes f^{\underline{k}} f_{\alpha}^{i-n(\lambda)-m} \Bigr).
\]
Note that in this sum the indices do not depend on $\mu$. Hence we can consider the element
\begin{equation}
\label{eqn:formula-Theta-s}
x \ := \ \sum_{\genfrac{}{}{0pt}{}{j,\underline{k},i,m}{0 \leq m \leq i-n(\lambda)}} (-1)^m f_\alpha^{m} \cdot u_j
\otimes \genfrac{(}{)}{0pt}{}{\alv}{m}_{\hbar} s(P_{j,\underline{k},i}) \otimes f^{\underline{k}} f_{\alpha}^{i-n(\lambda)-m} \quad \in V \o \M(s\la).
\end{equation}

By construction (and using \eqref{eqn:binomial-hb}), for any $\mu \in \bX$ sufficiently large with $n(\mu)\geq 0$, $(\id_V \o \uSp_{s\mu})(x)$ is a multiple of $\Sp_\mu(\phi)(\bvv_{\mu}^s)$; hence it satisfies
\[
(\id_V \o \uSp_{s \mu})(x) \cdot b = \hb(s \mu+\rho)(b) \cdot (\id_V \o \uSp_{s \mu})(x)
\]
for all $b \in \b$. Hence by Corollary \ref{cor:fixB} there exists a unique $\psi \in \Hom_{(\sh,\uh(\g))}(\M(0),V \otimes \M(s \lambda))$ such that $x=\psi(\bv_0)$. We set $\Theta^{V,\la}_s(\phi):=\psi$.
With this definition, for any $\phi$ and any sufficiently large $\mu$ with $n(\mu) \geq 0$ we have
\[
\Sp_\mu(\phi) \circ \mathbf{i}_\mu^s = \frac{(-\hb)^{n(\lambda)} n(\lambda+\mu)!}{n(\mu)!} (\id_V \o \mathbf{i}^s_{\lambda+\mu}) \circ \Sp_{s\mu}(\Theta_s^{V,\lambda}(\phi)),
\]
which implies \eqref{eqn:specialization-Theta}.
\end{proof}

\subsection{Restriction to a Levi subgroup}
\label{ss:Levi-DX}

Fix a subset $I$ of the set of simple roots, and recall the notation of \S\ref{ss:restriction-DX}.
As in \S\ref{ss:restriction-DX} we can consider our constructions both for $G$ and for $L$; we add super- or subscripts to indicate which reductive group we consider.

The projection $\g \twoheadrightarrow \fl$,
along $\fn^+_L \oplus \fn_L^-$, is $L$-equivariant and it induces a
morphism of graded $\C[\hb]$-modules $\pi^G_L : \uh(\g) \twoheadrightarrow \uh(\g)
/ (\fn_L^+ \cdot \uh(\g) + \uh(\g) \cdot \fn_L^-) \cong  \uh(\l)$.
%
We consider the
morphism of graded $\sh$-modules
\begin{equation}
\label{eqn:Verma-restriction}
\M^G(\la) \to \M^L(\la)
\end{equation}
sending $p \o u \in \M^G(\la)$ to $p \o \bigl( \imath^G_L \circ \pi^G_L(u) \bigr)$.
One can easily check that \eqref{eqn:Verma-restriction} is also a
morphism of $B_L$-modules, so that it induces a morphism of graded
$\C[\hb]$-modules
\begin{equation}
\label{eqn:vic2}
\rest^{V,\la}_{G,L} : \bigl( V \o \M^G(\la) \bigr)^B \to \bigl(V_{|L} \o \M^L(\la) \bigr)^{B_L}
\end{equation}\index{Rest1@$\rest^{V,\la}_{G,L}$}%
for any $V$ in $\Rep(G)$ and $\la \in \bX$.

Note that if $K \subset L$ is a smaller Levi subgroup containing $T$, then for any $V$ in $\Rep(G)$ and $\la \in \bX$ we have
\begin{equation}
\label{eqn:rest}
\rest^{V_{|L},\la}_{L,K} \circ \rest^{V,\la}_{G,L} = \rest^{V,\la}_{G,K}.
\end{equation}
If $L=T$, then we have $\M^T(\la) = \sh \llangle \la \rrangle$, where $B_T=T$ acts via $-\la$. Hence we have an isomorphism of graded $\sh$-modules
\[
\bigl(V_{|T} \o \M^T(\la) \bigr)^{B_T} = V_\la \o \sh.
\]
Under this isomorphism, one can easily check that $\rest^{V,\la}_{G,T} = \kalg_{V,\la}$.

Using Lemma \ref{lem:morphisms-Verma-v}, $\rest^{V,\la}_{G,L}$ induces a morphism of graded $\sh$-modules\index{Rest2@$\restH^{V,\la}_{G,L}$}%
\[
\restH^{V,\la}_{G,L} : \Hom_{(\sh,\uh(\g))} \bigl( \M^G(0), V \o \M^G(\la) \bigr) \to \Hom_{(\sh,\uh(\l))} \bigl( \M^L(0), V_{|L} \o \M^L(\la) \bigr).
\]

\begin{lem}
\label{lem:restH-injective}
For all $V$ in $\Rep(G)$ and $\la \in \bX$, the morphism $\restH^{V,\la}_{G,L}$ is injective.
\end{lem}

\begin{proof}
By \eqref{eqn:rest}, it is enough to prove the claim when $L=T$. Then by construction, if we identify $\Hom_{-\uloc(\t)} \bigl( \MMloc^T(0), V_{|T} \o \MMloc^T(\la) \bigr)$ with $V_\la \o \C[\hb,\hb^{-1}]$ in the natural way, for any $\mu \in \t^*$ we have $\Sp_\mu^T \circ \restH^{V,\la}_{G,T} = \mathsf{E}^{V,\la}_\mu \circ \Sp_\mu^G$.

Let $\phi \in \Hom_{(\sh,\uh(\g))} \bigl( \M^G(0), V \o \M^G(\la) \bigr)$ be such that $\restH^{V,\la}_{G,T}(\phi)=0$. Then for any $\mu \in \t^*$ we have $\mathsf{E}^{V,\la}_\mu ( \Sp_\mu^G(\phi))=0$. Using Lemma \ref{lem:intertwiners-quantum}(1) we deduce that for $\mu$ sufficient large we have $\Sp_\mu^G(\phi)=0$. By Corollary \ref{cor:sp-morphism} we deduce that $\phi=0$, which finishes the proof.
\end{proof}

The following result follows from construction (see the proof of Proposition \ref{prop:definition-Theta}), using the fact that if $\alpha$ is a simple root of $L$ then $[f_\alpha,\fn_L^-] \subset \fn_L^-$. For simplicity, in this statement we neglect the gradings.

\begin{lem}
\label{lem:Theta-rest-Levi}
Let $V$ in $\Rep(G)$ and $\la \in \bX$. If $\alpha \in I$ the following diagram commutes:
\[
\xymatrix@C=2cm{
\Hom_{(\sh,\uh(\g))} \bigl( \M^G(0), V \o \M^G(\la) \bigr) \ar[r]^-{{}^G \Theta_{s_\alpha}^{V,\la}} \ar[d]_-{\restH^{V,\la}_{G,L}} & {}^{s_\alpha} \hspace{-1pt} \Hom_{(\sh,\uh(\g))} \bigl( \M^G(0), V \o \M^G(s_\alpha \la) \bigr) \ar[d]^-{{}^{s_\alpha} \hspace{-1pt} \restH^{V,s_\alpha \la}_{G,L}} \\
\Hom_{(\sh,\uh(\l))} \bigl( \M^L(0), V_{|L} \o \M^L(\la) \bigr) \ar[r]^-{{}^L \Theta_{s_\alpha}^{V_{|L},\la}} & {}^{s_\alpha} \hspace{-1pt} \Hom_{(\sh,\uh(\l))} \bigl( \M^L(0), V_{|L} \o \M^L(s_\alpha \la) \bigr).
}
\]
\end{lem}

\subsection{Definition of operators $\Theta_w$}
\label{ss:def-Theta}

\begin{lem}
\label{lem:Thetas-invertible}
Let $V$ in $\Rep(G)$, and $\la \in \bX$.
For any simple root $\alpha$, we have
\[
{}^{s_\alpha} \hspace{-1pt} \bigl( \Theta_{s_\alpha}^{V,s_\alpha \la} \bigr) \circ \Theta_{s_\alpha}^{V,\la} = \id
\]
as endomorphisms of $\Hom_{(\sh,\uh(\g))} \bigl( \M(0), V \o \M(\la) \bigr) \langle \la(2\rhov) \rangle$. In particular, $\Theta_s^{V,\la}$ is an isomorphism.
\end{lem}

\begin{proof}
Let $L_\alpha$ be the Levi subgroup containing $T$ with roots $\{\alpha,-\alpha\}$.
Using Lemma \ref{lem:Theta-rest-Levi} and Lemma
\ref{lem:restH-injective} it is enough to prove the equality when
$G=L_\alpha$. In this case it is checked in Corollary \ref{cor:Theta-rk1} below.
\end{proof}

Let again $V$ in $\Rep(G)$ and $\la \in \bX$. Let $w \in W$, and choose a reduced expression $w=s_k \cdots s_1$. We define the isomorphism of graded $\sh$-modules
\begin{multline*}
\Theta_w^{V,\la} : \Hom_{(\sh,\uh(\g))} \bigl( \M(0), V \o \M(\la) \bigr) \langle \la(2\rhov) \rangle \\
\simto \w \Hom_{(\sh,\uh(\g))} \bigl( \M(0), V \o \M(w\la) \bigr) \langle (w\la)(2\rhov) \rangle
\end{multline*}
by the formula
\[
\Theta_w^{V,\la} := {}^{(s_{k-1} \cdots s_1)} \hspace{-1pt} \Theta_{s_k}^{V,s_{k-1} \cdots s_1(\la)} \circ \cdots \circ {}^{s_1} \hspace{-1pt} \Theta_{s_2}^{V,s_1 \la} \circ \Theta_{s_1}^{V,\la}.
\]

\begin{lem}
The operator $\Theta_w^{V,\la}$ does not depend on the choice of the reduced decomposition.
\end{lem}

\begin{proof}
Using \eqref{eqn:specialization-Theta} and \eqref{eqn:composition-Psi} we obtain that for $\mu \in \bX^-$ sufficiently large we have
\[
\Sp_{w \mu} \circ \Theta_w^{V,\la} = n(w,\la,\mu) \cdot \Psi^{V,\la}_{w,\mu} \circ \Sp_\mu
\]
where
\[
n(w,\la,\mu) = \frac{\prod_{\alv>0, w(\alv)<0,\la(\alv) >0} (-\hb)^{\la(\alv)} (-\mu(\alv)) \cdots (-\mu(\alv)-\la(\alv)+1)}{\prod_{\alv>0, w(\alv)<0,\la(\alv) <0} (-\hb)^{-\la(\alv)} (-\mu(\alv)-\la(\alv)) \cdots (-\mu(\alv)+1)}.
\]
The right-hand side is independent of the reduced decomposition, hence we conclude by Corollary \ref{cor:sp-morphism}.
\end{proof}

By construction and Lemma \ref{lem:Thetas-invertible}, our collection of isomorphisms satisfies
condition \eqref{eqn:cocycle-Theta}
for all $V$ in $\Rep(G)$, $\la \in \bX$ and $x,y \in W$.

\subsection{Relation with the operators $\Phi$}
\label{ss:relation-Phi-Theta}

As explained in \S\ref{ss:diffX}, for any $V$ in $\Rep(G)$ and $\la \in \bX$ there exists a canonical isomorphism of graded $\sh$-modules
\[
\bigl( V \o \tla \dh(\X)_\la \bigr)^G \ \cong \ \bigl( V \o \M(\la) \bigr)^B.
\]
Using Lemma \ref{lem:morphisms-Verma-v} we deduce a canonical isomorphism
\begin{equation}
\label{eqn:isom-D-Hom}
\bigl( V \o \tla \dh(\X)_\la \bigr)^G \ \cong \ \Hom_{(\sh,\uh(\g))} \bigl( \M(0), V \o \M(\la) \bigr).
\end{equation}

The following result is clear from definitions.

\begin{lem}
\label{lem:restriction-DX-Hom}
Let $V$ in $\Rep(G)$, $\la \in \bX$, and $L \subset G$ be a Levi subgroup containing $T$. The following diagram commutes:
\[
\xymatrix@C=2cm{
\Hom_{(\sh,\uh(\g))} \bigl( \M^G(0), V \o \M^G(\la) \bigr) \ar[d]_-{\restH^{V,\la}_{G,L}} \ar[r]_-{\sim}^-{\eqref{eqn:isom-D-Hom}} & \bigl( V \o \tla \dh(\X_G)_\la \bigr)^G \ar[d]^-{\mathscr{R}^{V,\la}_{G,L}}  \\
\Hom_{(\sh,\uh(\l))} \bigl( \M^L(0), V_{|L} \o \M^L(\la) \bigr) \ar[r]_-{\sim}^-{\eqref{eqn:isom-D-Hom}} & \bigl( V_{|L} \o \tla \dh(\X_L)_\la \bigr)^L.
}
\]
\end{lem}

In the following proposition we assume that the vectors $f_\alpha$ of Section \ref{sec:DX} are the same as the vectors $f_\alpha$ of the present Section \ref{sec:morphisms}.

\begin{prop}
\label{prop:Phi-Theta}
For any $V$ in $\Rep(G)$, $\la \in \bX$ and $w \in W$
the following diagram commutes:
\[
\xymatrix@C=2cm{
\bigl( V \o \tla \dh(\X)_\la \bigr)^G \ar[r]^-{\eqref{eqn:isom-D-Hom}}_-{\sim} \ar[d]_-{\Phi^{V,\la}_w} & \Hom_{(\sh,\uh(\g))} \bigl( \M(0), V \o \M(\la) \bigr) \ar[d]^-{\Theta^{V,\la}_w} \\
\w \bigl( V \o  {}^{(w\la)} \hspace{-1pt} \dh(\X)_{w \la} \bigr)^G \ar[r]^-{\eqref{eqn:isom-D-Hom}}_-{\sim} & \w \Hom_{(\sh,\uh(\g))} \bigl( \M(0), V \o \M(w \la) \bigr).
}
\]
\end{prop}

\begin{proof}
By \eqref{eqn:cocycle-Phi} and \eqref{eqn:cocycle-Theta} it is enough to prove commutativity when $w$ is a simple reflection.

By Corollary \ref{cor:Phi-rest}, Lemma \ref{lem:Theta-rest-Levi} and Lemma \ref{lem:restriction-DX-Hom}, our constructions are compatible with restriction to a Levi subgroup. Using Lemma \ref{lem:restH-injective}, we deduce that it is enough to prove the claim when $G$ has semisimple rank one. In this case it is proved in Lemma \ref{lem:comparison-Phi-Theta} below.
\end{proof}

It follows from this proposition that for any $V$ in $\Rep(G)$, $\la \in \bX$ and $w \in W$ there exists a unique isomorphism\index{Omega@$\Omega^{V,\la}_w$}%
\[
\Omega^{V,\la}_w : \bigl( V \o \M(\la) \bigr)^B \simto \w \bigl( V \o \M(w \la) \bigr)^B
\]
which corresponds to $\Phi^{V,\la}_w$ under the left-hand isomorphism of Lemma \ref{lem:fixed-points-morphisms}, and to $\Theta^{V,\la}_w$ under the right-hand isomorphism of Lemma \ref{lem:fixed-points-morphisms}.

\section{Geometry of the Grothendieck--Springer resolution}
\label{sec:wfg}

\subsection{The $W$-action on the regular part of $\wfg$}
\label{ss:action-W-wfgr}

In this section we are interested in the geometry of the Grothendieck--Springer resolution $\wfg$ (see \S\ref{subsec2}). We have a standard commutative diagram
\[
\xymatrix{
\wfg \ar[r]^-{\pi} \ar[d]_-{\delta} & \fg^* \ar[d] \\
\ft^* \ar[r] & \ft^* \hspace{-1pt} / \hspace{-2pt} /W \cong \fg^* \hspace{-1pt} / \hspace{-2pt} / G
}
\]
where the right vertical and the lower horizontal maps are the natural quotient maps, $\pi$ is defined in \S\ref{ss:W-symmetries}, and $\delta$ is defined by $\delta(g \times_B \eta) = \eta_{|\b} \in (\b/\u)^* \cong \t^*$.

It is well known that there exists an action of $W$ on $\wfgr$, which commutes with the natural $G$-action, and such that the restrictions of $\pi$ and $\delta$ are $W$-equivariant, where $W$ acts naturally on $\ft^*$ and trivially on $\fg^*$ (see e.g.~\cite[Proposition 1.9.2]{br}). We denote by $\theta_w : \wfgr \simto \wfgr$\index{thetaw@$\theta_w$} the action of $w$. In the whole section we will use \cite{br} as a convenient reference for the properties of this action, though much of this material was known before.

Recall that for any $\lambda \in \bX$ we have a line bundle $\cO_{\wfg}(\lambda)$ on $\wfg$ and its restriction $\cO_{\wfgr}(\lambda)$ to $\wfgr$.

\begin{lem}
\label{lem:action-W-wfg-line-bundles}

For any $w \in W$ and $\lambda \in \bX$, there exists an isomorphism of $G \times \Gm $-equivariant line bundles
\[
(\theta_{w^{-1}})^* \oo_{\wfgr}(\lambda) \langle \la(2\rhov) \rangle \ \cong \ \oo_{\wfgr}(w \lambda) \langle (w\la)(2\rhov) \rangle.
\]

\end{lem}

\begin{proof}
It is enough to prove the isomorphism when $w=s_{\alpha}$ is a simple reflection. Next, by compatibility of $\theta_{w^{-1}}^*$ with tensor products (and passing to a simply connected cover if necessary), it is enough to prove the isomorphism when $\lambda(\alv) \in \{-1,0,1\}$. Finally using duality one can assume that $\lambda(\alv) \in \{-1,0\}$. In these cases the isomorphism follows from \cite[Lemma 1.11.3(2) and Remark 1.11.7(2)]{br}.
\end{proof}

%
%

\subsection{Construction of the operators $\sigma$}
\label{ss:construction-sigma}

The isomorphism of Lemma \ref{lem:action-W-wfg-line-bundles} is unique
 up to a scalar since we have 
 \[
 \Gamma(\wfgr,\cO_{\wfgr})^{G \times
 \Gm } \cong \Gamma(\wfg,\cO_{\wfg})^{G \times
 \Gm } = \bbc
 \]
 (see e.g.~equation \eqref{eqn:restriction-wfg} below).
 We will need to fix a normalization of this
 isomorphism, using our choice of vectors $f_\alpha$. Let us denote by $\eta_0 \in \g^*$
 the element which is zero on $\ft$ and on any $\g_{\beta}$ where $\beta$ is not opposite to a simple root, and such that
 $\eta_0(f_{\alpha})=1$ for any simple root $\alpha$. Then
 $(1 \times_B \eta_0) \in \wfgr$, and this point is $W$-invariant. For any
 $\la \in \bX$, $\cO_{\wfg}(\la)$ is the sheaf of sections of the line
 bundle\index{L@$\mathcal{L}(\la)$}%
 \[
 \mathcal{L}(\la) := G \times_B \bigl( (\g/\u)^* \times \C_{-\lambda} \bigr)
 \]
 over
 $\wfg$. The fiber of this line bundle over $(B/B,\eta_0)$ can be
 canonically identified with $\bbc$ through the morphism $x \mapsto
 (1 \times_B (\eta_0,x))$. Then there exists a unique isomorphism of $G \times \Gm$-equivariant line bundles
\begin{equation}
\label{eqn:isom-line-bundles}
(\theta_{w^{-1}})^* \cO_{\wfgr}(\lambda) \langle \la(2\rhov) \rangle \ \xrightarrow{\sim} \ \cO_{\wfgr}(w \lambda) \langle (w\la)(2\rhov) \rangle
\end{equation}
whose restriction to $(B/B,\eta_0) \in \wfgr$ is $\id_{\bbc}$ via the identifications above.

Let $j_{\mathrm{r}}: \wfgr \hookrightarrow \wfg$
be the inclusion. 
As the codimension of $\wfg \smallsetminus \wfgr$ in $\wfg$ is at least $2$ (see e.g.~\cite[Proposition 1.9.3]{br}), for any $\lambda \in \bX$ the morphism $\oo_{\wfg}(\la) \to (j_{\mathrm{r}})_* \oo_{\wfgr}(\la)$ induced by adjunction is an isomorphism (see \cite[Theorem 11.5.(ii)]{m}). We deduce that the
restriction induces an isomorphism
\begin{equation}
\label{eqn:restriction-wfg}
\Gamma(\wfg,\cO_{\wfg}(\lambda)) \ \simto \ \Gamma(\wfgr,\cO_{\wfgr}(\lambda)).
\end{equation}
As $\theta_{w^{-1}}$ is an isomorphism, the adjunction morphism $\cO_{\wfgr}(\lambda) \to (\theta_{w^{-1}})_* (\theta_{w^{-1}})^* \cO_{\wfgr}(\lambda)$ is also an isomorphism; in particular there is a canonical isomorphism
\[
\Gamma(\wfgr,\cO_{\wfgr}(\lambda)) \ \simto \ \Gamma(\wfgr,\theta_{w^{-1}}^* \cO_{\wfgr}(\lambda)).
\]
Putting these remarks together with isomorphism \eqref{eqn:isom-line-bundles}, we obtain for any $w \in W$ and $\lambda \in \bX$ a canonical isomorphism of graded $\mathrm{S}(\ft)$-modules and $G$-modules
\beq{eqn:isom-global-sections}
\Gamma(\wfg,\cO_{\wfg}(\lambda)) \langle \la(2\rhov) \rangle \ \simto \ \w \Gamma(\wfg, \cO_{\wfg}(w \lambda)) \langle (w\la)(2\rhov) \rangle.
\eeq
Hence for any $V$ in $\Rep(G)$ we obtain an isomorphism of graded $\mathrm{S}(\ft)$-modules
\[
\sigma^{V,\lambda}_w : \ \bigl(V \otimes \Gamma(\wfg,\cO_{\wfg}(\lambda)) \bigr)^G \langle \la(2\rhov) \rangle \ \xrightarrow{\sim} \ \w \bigl( V \otimes \Gamma(\wfg, \cO_{\wfg}(w \lambda)) \bigr)^G \langle (w\la)(2\rhov) \rangle.
\]
By construction, this collection of isomorphisms satifies relations \eqref{eqn:cocycle-sigma}.

As explained above,
for any $\mu \in \bX$ we can describe $\Gamma(\wfg,\cO_{\wfg}(\mu))$ as the space of sections of the line bundle $\mathcal{L}(\mu)$ over $\wfg$. In particular, for any $\eta \in (\g/\u)^*$ there exists a unique morphism\index{ev@$\mathrm{ev}_{\eta}^\mu$}%
\[
\mathrm{ev}_{\eta}^\mu :\ \Gamma(\wfg,\cO_{\wfg}(\mu)) \to \bbc
\]
(``evaluation at $(1,\eta)$'')
such that if $f \in \Gamma(\wfg,\cO_{\wfg}(\mu))$ is considered as a section $\wfg \to \mathcal{L}(\mu)$ we have 
\[
f(1 \times_B \eta) = 1 \times_B (\eta, \mathrm{ev}_{\eta}^\mu(f)).
\]

With this definition, \eqref{eqn:isom-line-bundles} can be characterized as the unique isomorphism of $G \times \C^\times$-equiva\-riant line bundles $(\theta_{w^{-1}})^* \cO_{\wfgr}(\la)  \langle \la(2\rhov) \rangle \simto \cO_{\wfgr}(w \la) \langle (w\la)(2\rhov) \rangle$ such that the following diagram commutes:
\beq{eqn:diagram-sigma}
\vcenter{
\xymatrix@C=1.5cm{
\Gamma(\wfg,\cO_{\wfg}(\lambda)) \ar[rr]_-{\sim}^-{\eqref{eqn:isom-global-sections}} \ar[rd]_-{\mathrm{ev}_{\eta_0}^\la} & & \Gamma(\wfg,\cO_{\wfg}(w \lambda)) \ar[ld]^-{\mathrm{ev}_{\eta_0}^{w\la}} \\
& \bbc. &
}
}
\eeq
Indeed the diagram commutes by construction. To prove unicity if suffices to prove that the morphism $\mathrm{ev}_{\eta_0}^\la$ is non-zero. However, if $\la$ is dominant this property follows from the fact that $\oo_{\wfg}(\la)$ is globally generated (which itself follows from the similar claim for $\B$), and the general case follows from commutativity of \eqref{eqn:diagram-sigma} (and the fact that every weight it $W$-conjugate to a dominant weight).

Below we will need a refinement of this characterization in the case $w=s_\alpha$ for a simple root $\alpha$. We denote by $\eta_\alpha \in \g^*$ the element which is zero on $\t$ and on any $\g_\beta$ with $\beta \neq -\alpha$, and such that $\eta_\alpha(f_\alpha)=1$.

\begin{lem}
\label{lem:evaluation-sigma}
When $w=s_\alpha$, \eqref{eqn:isom-line-bundles} is the unique isomorphism of $G \times \C^\times$-equivariant line bundles $(\theta_{s_\alpha})^* \cO_{\wfgr}(\la)  \langle \la(2\rhov) \rangle \to \cO_{\wfgr}(s_\alpha \la) \langle (s_\alpha \la)(2\rhov) \rangle$ such that the following diagram commutes:
\[
\xymatrix@C=1.5cm{
\Gamma(\wfg,\cO_{\wfg}(\lambda)) \ar[rr]_-{\sim}^-{\eqref{eqn:isom-global-sections}} \ar[rd]_-{\mathrm{ev}_{\eta_\alpha}^\la} & & \Gamma(\wfg,\cO_{\wfg}(s_\alpha \lambda)) \ar[ld]^-{\mathrm{ev}_{\eta_\alpha}^{s_\alpha \la}} \\
& \bbc &
}
\]

\end{lem}

\begin{proof}
As for the similar claim concerning diagram \eqref{eqn:diagram-sigma}, we only have to check that the diagram commutes. In this proof we denote the isomorphism \eqref{eqn:isom-global-sections} by $\vartheta_\alpha : \Gamma(\wfg,\cO_{\wfg}(\lambda)) \simto \Gamma(\wfg,\cO_{\wfg}(s_\alpha \lambda))$. Choose a coweight $\muv \in X_*(T)$ such that $\alpha(\muv)=0$ and $\beta(\muv)>0$ for all simple roots $\beta\neq\alpha$. Then $\lim_{z \to 0} \muv(z) \cdot \eta_0 = \eta_\alpha$, so that it is enough to prove that the following diagram commutes for all $z \in \C^\times$:
\[
\xymatrix@C=1.5cm{
\Gamma(\wfg,\cO_{\wfg}(\lambda)) \ar[rr]_-{\sim}^-{\vartheta_\alpha} \ar[rd]_-{\mathrm{ev}_{\muv(z) \cdot \eta_0}^\la} & & \Gamma(\wfg,\cO_{\wfg}(s_\alpha \lambda)) \ar[ld]^-{\mathrm{ev}_{\muv(z) \cdot \eta_0}^{s_\alpha \la}} \\
& \bbc &
}
\]
However for $\nu \in \bX$ and $f \in \Gamma(\wfg,\cO_{\wfg}(\nu))$ we have
\begin{multline*}
f \bigl( 1 \times_B (\muv(z) \cdot \eta_0) \bigr)=f(\muv(z) \times_B \eta_0) = \muv(z) \cdot \Bigl( (\muv(z^{-1}) \cdot f)(1 \times_B \eta_0) \Bigr) \\
= \muv(z) \cdot \Bigl( 1 \times_B \bigl( \eta_0, \mathrm{ev}^\la_{\eta_0}(\muv(z^{-1}) \cdot f) \bigr) \Bigr) = \Bigl( \muv(z) \times_B \bigl( \eta_0, \mathrm{ev}^\la_{\eta_0}(\muv(z^{-1}) \cdot f) \bigr) \Bigr) \\
= \Bigl( 1 \times_B \bigl( \muv(z) \cdot \eta_0, z^{\la(\muv)} \cdot \mathrm{ev}^\la_{\eta_0}(\muv(z^{-1}) \cdot f) \bigr) \Bigr).
\end{multline*}
On the other hand we have
\[
f \bigl( 1 \times_B (\muv(z) \cdot \eta_0) \bigr)= \Bigl( 1 \times_B \bigl(\muv(z) \cdot \eta_0, \mathrm{ev}^\la_{\muv(z) \cdot \eta_0}(f) \bigr) \Bigr),
\]
which implies that
\[
\mathrm{ev}^{\la}_{\muv(z) \cdot \eta_0}(f) \ = \ z^{\la(\muv)} \cdot \mathrm{ev}^\la_{\eta_0}(\muv(z^{-1}) \cdot f).
\]
Hence we obtain
\[
\mathrm{ev}^{s_\alpha \la}_{\muv(z) \cdot \eta_0}(\vartheta_\alpha f) = z^{(s_\alpha \la)(\muv)} \cdot \mathrm{ev}^{s_\alpha\la}_{\eta_0}(\muv(z^{-1}) \cdot (\vartheta_\alpha f)) = z^{\la(\muv)} \cdot \mathrm{ev}^{\la}_{\eta_0}(\muv(z^{-1}) \cdot f) = \mathrm{ev}^{\la}_{\muv(z) \cdot \eta_0}(f)
\]
since $(s_\alpha \la)(\muv) = \la(\muv)$ and $\vartheta_\alpha$ is $G$-equivariant. This finishes the proof.
\end{proof}

\subsection{Restriction to a Levi subgroup}

Fix a subset $I$ of the set of the set of simple roots, and recall the notation of \S\ref{ss:restriction-DX}.
We can consider the Grothendieck--Springer resolution $\wfl$ associated with $L$, and there exists a natural morphism $\varpi^G_L : G \times_L \wfl = G \times_{B_L} (\l/\u_L)^* \to \wfg$\index{piGL@$\varpi^G_L$} induced by the identification $\l^* \cong (\g/(\fn^+_L \oplus \fn_L^-))^*$. In particular for $L=T$ we have $\wft=\t^*$, and the morphism $\varpi^G_T : G \times_T \wft = G/T \times \t^* \to \wfg$ identifies with the morphism denoted ``$a$'' in \S\ref{subsec2}.
The following diagram commutes by construction:
\begin{equation}
\label{eqn:diagram-varpi}
\vcenter{
\xymatrix@C=2cm{
G \times_T \wft \ar[r]_-{G \times_L \varpi^L_T} \ar@/^1.3pc/[rr]^{\varpi^G_T} & G \times_L \wfl \ar[r]_-{\varpi^G_L} & \wfg.
}
}
\end{equation}

We have $(\varpi^G_L)^{-1}(\wfgr) \subset G \times_L \wflr$. (Note that this inclusion is strict in general.) This open subset is $W_L$-invariant (for the $W_L$-action on $G \times_L \wflr$ induced by the action on $\wflr$), and the morphism $(\varpi^G_L)^{-1}(\wfgr) \to \wfgr$ induced by $\varpi^G_L$ is $W_L$-equivariant.

Adjunction for the morphism $\varpi^G_L$ induces an injective morphism
\beq{eqn:restriction-varpi}
\Gamma(\wfg,\cO_{\wfg}(\lambda)) \ \hookrightarrow \ \Gamma(G \times_L \wfl, \cO_{G \times_L \wfl}(\lambda)) \cong \Ind_L^G \bigl( \Gamma(\wfl, \cO_{\wfl}(\lambda)) \bigr).
\eeq
For simplicity, in the next statement we forget about the gradings (i.e.~the $\C^\times$-actions).

\begin{lem}
\label{lem:Sigma-restriction-Levi}

The following diagram commutes for any $V$ in $\Rep(G)$, $\lambda \in \bX$ and $w \in W_L$, where vertical maps are induced by \eqref{eqn:restriction-varpi}:
\[
\xymatrix@C=2cm{
\bigl(V \otimes \Gamma(\wfg,\cO_{\wfg}(\lambda)) \bigr)^G \ar[r]^-{\sigma^{V,\lambda}_w} \ar@{^{(}->}[d] & {}^w \bigl( V \otimes \Gamma(\wfg, \cO_{\wfg}(w \lambda)) \bigr)^G \ar@{^{(}->}[d] \\
\bigl(V_{|L} \otimes \Gamma(\wfl,\cO_{\wfl}(\lambda)) \bigr)^L \ar[r]^-{\sigma^{V_{|L},\lambda}_w} & {}^w \bigl( V_{|L} \otimes \Gamma(\wfl, \cO_{\wfl}(w \lambda)) \bigr)^L.
}
\]

\end{lem}

\begin{proof}
It is enough to prove the result when $w=s_{\alpha}$ is a simple reflection associated with a simple root $\alpha \in I$. Let $\mathfrak{sl}_{2,\alpha} \subset \fg$ be the Lie subalgebra generated by $\g_{\alpha}$ and $\g_{-\alpha}$, and consider the open subset
\[
\wfg_{\alpha\mathrm{-r}} := \{(g \times_B \eta) \in \wfg \mid \eta_{|\mathfrak{sl}_{2,\alpha}} \neq 0\} \ \subset \ \wfg.
\]
By \cite[Lemma 1.9.1]{br} we have $\wfgr \subset \wfg_{\alpha\mathrm{-r}}$, and
it follows from \cite[Lemma 2.9.1]{br} that the action map $\theta_{s_{\alpha}}$ is the restriction of an isomorphism (denoted similarly) $\theta_{s_{\alpha}} : \wfg_{\alpha\mathrm{-r}} \xrightarrow{\sim} \wfg_{\alpha\mathrm{-r}}$. Moreover, isomorphism \eqref{eqn:isom-line-bundles} is (a shift of) the restriction of an isomorphism of $G \times \C^\times$-equivariant line bundles
\[
\varsigma^{G,\lambda}_{s_{\alpha}} : (\theta_{s_{\alpha}})^* \cO_{\wfg_{\alpha\mathrm{-r}}}(\lambda) \xrightarrow{\sim} \cO_{\wfg_{\alpha\mathrm{-r}}}(s_{\alpha} \lambda) \langle -2 \la(\alv) \rangle.
\]
The same assertions are of course true for the Levi $L$, and we obtain a similar isomorphism $\varsigma^{L,\lambda}_{s_{\alpha}}$.

The morphism $\varpi^G_L$ restricts to a morphism $G \times_L \wfl_{\alpha\mathrm{-r}} \to \wfg_{\alpha\mathrm{-r}}$ (denoted similarly) which satisfies $\varpi^G_L \circ (G \times_L \theta_{s_\alpha}^L) = \theta_{s_\alpha}^G \circ \varpi^G_L$. What we have to prove is that the two isomorphisms
\[
(\varpi^G_L)^* \varsigma^{G,\lambda}_{s_{\alpha}}, \ \Ind_L^G(\varsigma^{L,\lambda}_{s_{\alpha}}) : (G \times_L \theta^L_{s_{\alpha}})^* \cO_{G \times_L \wfl_{\alpha\mathrm{-r}}}(\lambda) \xrightarrow{\sim} \cO_{G \times_L \wfl_{\alpha\mathrm{-r}}}(s_{\alpha} \lambda)
\]
coincide. As both isomorphisms are $G \times \bbc^{\times}$-equivariant, and as
\[
\Gamma(G \times_L \wfl_{\alpha\mathrm{-r}},\cO_{G \times_L \wfl_{\alpha\mathrm{-r}}})^{G \times \bbc^{\times}} \cong \Gamma(G \times_L \wfl,\cO_{G \times_L \wfl})^{G \times \bbc^{\times}} \cong \bbc,
\]
we know that these isomorphisms coincide up to multiplication by a scalar $c \in \C^\times$. Consider the following diagram:
\begin{equation}
\label{eqn:diagram-evaluation}
\vcenter{
\xymatrix@C=1.5cm{
\Gamma(\wfg,\cO_{\wfg}(\lambda)) \ar[r]^-{\sim} \ar@{^{(}->}[d]^-{\eqref{eqn:restriction-varpi}} \ar@/_80pt/[dd]_-{\mathrm{ev}_{\eta_\alpha}^{\la,G}} & \Gamma(\wfg,\cO_{\wfg}(s_{\alpha}\lambda)) \ar@{^{(}->}[d]^-{\eqref{eqn:restriction-varpi}} \ar@/^80pt/[dd]^-{\mathrm{ev}_{\eta_\alpha}^{s_\alpha \la,G}}\\
\Gamma(G \times_L \wfl,\cO_{G \times_L \wfl}(\lambda)) \ar[r]^-{\sim} \ar[d]_-{{}'\mathrm{ev}_{\eta_\alpha}^{\la,L}} & \Gamma(G \times_L \wfl,\cO_{G \times_L \wfl}(s_{\alpha} \lambda)) \ar[d]^-{{}'\mathrm{ev}_{\eta_\alpha}^{s_\alpha \la,L}} \\
\bbc \ar@{=}[r] & \bbc,
}
}
\end{equation}
where the upper horizontal morphism is induced by $\varsigma^{G,\lambda}_{s_{\alpha}}$, the middle horizontal morphism is induced by $\Ind_L^G(\varsigma^{L,\lambda}_{s_{\alpha}})$, and the lower vertical maps are induced by evaluation at $(1 \times_{B_L} \eta_\alpha)$. We know that the upper square commutes up to multiplication by $c$, that the lower square and the exterior square both commute and finally that the morphism $\mathrm{ev}_{\eta_\alpha}^{\la,G}$ is non-zero (see Lemma \ref{lem:evaluation-sigma} and its proof). We deduce that $c=1$, which finishes the proof.
\end{proof}

\subsection{Relation to the operators $\Phi$}

Recall that if $\la \in \bX$ there is a natural morphism $\tla \dh(\X)_\la \to \Gamma(\wfg, \cO_{\wfg}(\la))$ sending $\hb$ to $0$ (see e.g. \S\ref{ss:Fourier-classical}).

\begin{prop}
\label{prop:Phi-sigma}
Let $V$ in $\Rep(G)$, $\la \in \bX$ and $w \in W$. The following diagram commutes, where vertical maps are the natural morphisms sending $\hb$ to $0$:
\[
\xymatrix@C=2cm{
\bigl( V \o \tla \dh(\X)_\la \bigr)^G \ar[r]^-{\Phi^{V,\la}_{w}} \ar[d] & \w \bigl( V \o {}^{(w\la)} \hspace{-1pt} \dh(\X)_{w\la} \bigr)^G \ar[d] \\
\bigl( V \o \Gamma(\wfg, \cO_{\wfg}(\la)) \bigr)^G \ar[r]^-{\sigma^{V,\la}_w} & \w \bigl( V \o \Gamma(\wfg, \cO_{\wfg}(w\la)) \bigr)^G
}
\]
\end{prop}

\begin{proof}
First, using relations \eqref{eqn:cocycle-Phi} and \eqref{eqn:cocycle-sigma} it is enough to prove the lemma when $w=s_\alpha$ is a simple reflection. Then using the compatibility of our constructions with restriction to a Levi subgroup (see Corollary \ref{cor:Phi-rest} and Lemma \ref{lem:Sigma-restriction-Levi}) and the injectivity of morphism \eqref{eqn:restriction-varpi}, it is enough to prove the lemma when $G$ has semisimple rank one (with unique simple root $\alpha$). In this case it is proved in Corollary \ref{cor:Phi-sigma-rk1} below.
\end{proof}

\begin{rem}
\label{rk:Phi-sigma}
By Lemma \ref{lem:hbar=0}, the vertical arrows in the diagram of Proposition \ref{prop:Phi-sigma} induce isomorphisms
\[
\bigl( V \o \tla \dh(\X)_\la \bigr)^G / \langle \hb \rangle \simto \bigl( V \o \Gamma(\wfg, \cO_{\wfg}(\la)) \bigr)^G, \quad \bigl( V \o {}^{(s\la)} \hspace{-1pt} \dh(\X)_{s\la} \bigr)^G / \langle \hb \rangle \simto \bigl( V \o \Gamma(\wfg, \cO_{\wfg}(s\la)) \bigr)^G.
\]
Hence the proposition implies that the operators $\sigma^{V,\la}_w$ can be completely recovered from the operators $\Phi^{V,\la}_{w}$ (or equivalently the operators $\Theta^{V,\la}_{w}$, see Proposition \ref{prop:Phi-Theta}).
\end{rem}

\subsection{Geometric interpretation: $W$-action on the regular part of $T^* \X$}
\label{ss:action-T*X}

The results in this subsection will not be used in the rest of the paper.

Consider the natural morphism $T^*\X = G \times_U (\g/\u)^* \to \g^*$, and denote by $(T^*\X)_{\mathrm{r}}$ the inverse image of the open subset of regular elements in $\g^*$. Note that the $G \times T$-action on $\X$ defined in \S\ref{subsec2} induces an action on $T^*\X$ which stabilizes $(T^*\X)_{\mathrm{r}}$, and also a moment map $T^*\X \to \t^*$.

The existence of the collection of isomorphisms \eqref{eqn:isom-line-bundles} has the following quite surprising consequence. This construction will be re-interpreted and studied further in \cite{gk}.

\begin{prop}
There exists an action of $W$ on $(T^*\X)_{\mathrm{r}}$ (which depends on the choice of the $e_\alpha$'s), denoted $\odot$, which satisfies the following properties: 
\begin{enumerate}
\item
for any $w \in W$ the morphism $w \odot (-)$ is $G$-equivariant;
\item 
for $x \in (T^*\X)_{\mathrm{r}}$ and $t \in T$ we have
$w \odot (t \cdot x) = w(t) \cdot (w \odot x)$;
\item 
the natural morphism $(T^*\X)_{\mathrm{r}} \to \wfgr$ is $W$-equivariant; 
\item
the restriction $(T^*\X)_{\mathrm{r}} \to \t^*$ of the moment map is $W$-equivariant.
\end{enumerate}
\end{prop}

\begin{proof}
The morphism $p_{\mathrm{r}} : (T^*\X)_{\mathrm{r}} \to \wfgr$ is a $T$-torsor; in particular it is affine. Hence to prove the proposition it is enough to construct of collection of isomorphisms of $G$-equivariant sheaves of algebras
\[
(\theta_{w^{-1}})^* \bigl( (p_{\mathrm{r}})_* \oo_{(T^*\X)_{\mathrm{r}}} \bigr) \ \cong \ (p_{\mathrm{r}})_* \oo_{(T^*\X)_{\mathrm{r}}}
\]
for all $w \in W$, which are
compatible with composition in $W$. Now we have a natural isomorphism of $G$-equivariant sheaves of algebras
\[
(p_{\mathrm{r}})_* \oo_{(T^*\X)_{\mathrm{r}}} \ \cong \ \bigoplus_{\la \in \bX} \, \oo_{\wfgr}(\la),
\]
hence it is enough to construct a collection of isomorphisms of $G$-equivariant line bundles
\[
(\theta_{w^{-1}})^* \oo_{\wfgr}(\la) \ \cong \ \oo_{\wfgr}(w\la)
\]
compatible with tensor product and composition in $W$. However one can easily check that the collection constructed in \eqref{eqn:isom-line-bundles} satisfies these requirements.
\end{proof}

\section{Reminder on the Satake equivalence}
\label{sec:Satake}

In Sections \ref{sec:Satake}--\ref{sec:applications} we let $\Gv$ be a complex connected reductive group. We choose a maximal torus $\Tv \subset \Gv$ (with Lie algebra ${\check \t}$), and a Borel subgroup $\Bv \subset \Gv$. We let $\Uv$, resp.~$\Uv^-$, be the unipotent radical of $\Bv$, resp.~of the opposite Borel subgroup (with respect to $\Tv$). We set $\bX:=X_*(\Tv)$, and let $\bX^+ \subset \bX$ be the sub-semigroup of dominant coweights of $\Tv$ (where positive roots of $\Gv$ are those appearing in $\mathrm{Lie}(\Uv)$). We also set $\sh:=\sym({\check \t}^*)[\hb]$, considered as a graded algebra where elements of ${\check \t}^*$ and $\hb$ are in degree $2$. Finally we denote by $2\rhov \in X^*(\Tv)$ the sum of the positive roots.

\subsection{Satake equivalence}
\label{ss:Satake}

Recall (see \S\ref{ss:Gr-intro}) that the affine Grassmannian attached to $\Gv$ is the ind-variety
\[
\Gr_\Gv := \Gv(\KK)/\Gv(\OO)
\]
(equipped with the \emph{reduced} scheme structure). This ind-variety is equipped with an action of the group scheme $\Gv(\OO)$. Recall (see \cite{gi,mv}) that the category
\[
\Perv_{\Gv(\OO)}(\Gr_\Gv)
\]
of $\Gv(\OO)$-equivariant perverse sheaves on $\Gr_\Gv$ (with coefficients in $\C$) can be endowed with a natural convolution product $\star$ which makes it a tensor category, and that the functor\index{FG@$\sfF_{\Gv}$}%
\[
\sfF_{\Gv} := \coH^\hdot(\Gr_\Gv,-) : \Perv_{\Gv(\OO)}(\Gr_\Gv) \to \mathsf{Vect}(\C)
\]
(where $\mathsf{Vect}(\C)$ is the category of finite dimensional $\C$-vector spaces) is a tensor functor. (As usual, perverse sheaves on $\Gr_{\Gv}$ are assumed to be supported on a finite union of $\Gv(\OO)$-orbits.) We let
\[
G:=\mathrm{Aut}^\star(\sfF_\Gv)
\]
be the $\C$-group scheme of automorphisms of this tensor functor. It is well known (see \cite{gi,mv}) that $G$ is a (complex) connected reductive group, with root datum dual to that of $\Gv$. Moreover, the functor $\sfF_\Gv$ lifts to an equivalence of tensor categories
\[
\cS_\Gv : \Perv_{\Gv(\OO)}(\Gr_\Gv) \simto \Rep(G)
\]
known as the geometric Satake equivalence.

In \S\ref{ss:Gr-intro} we have defined the embedding $\bX=\Gr_\Tv \hookrightarrow \Gr_\Gv$, the points $\bla$ (for $\la \in \bX$), the orbits $\Gr^\la_\Gv$ (for $\la \in \bX^+$), the semi-infinite orbits $\fT_\la$ (for $\la \in \bX$), and the morphisms $i_\lambda$ and $t_\lambda$.

Using the identification of $\Gr_{\Tv}$ with $\bX$, the group $T$ (of automorphisms of the tensor functor $\sfF_{\Tv}$) is identified with the torus $\Hom_{\Z}(\bX,\C^{\times})$. In particular, the character lattice $X^*(T)$ is canonically identified with $\bX$, hence the category $\Rep(T)$ identifies with the category of finite dimensional $\bX$-graded vector spaces. Define the functor
\[
\sfF^{\bX} := \bigoplus_{\lambda \in \bX} \coH^{\la(2\rhov)} \bigl( \fT_{\lambda}, t_{\lambda}^!(-) \bigr) : \Perv_{\Gv(\OO)}(\Gr_{\Gv}) \to \Rep(T).
\]
By~\cite[Theorems 3.5 and 3.6]{mv}, we have a canonical isomorphism
\begin{equation}
\label{eqn:fibre-functor-T}
\For^{T} \circ \sfF^{\bX} \simto \sfF_{\Gv},
\end{equation}
where $\For^T : \Rep(T) \to \mathrm{Vect}(\C)$ is the forgetful functor.
Moreover, $\sfF^{\bX}$ is a tensor functor, and \eqref{eqn:fibre-functor-T} is an isomorphism of tensor functors. So $\sfF^{\bX}$ is the composition of $\cS_\Gv$ with a tensor functor $\Rep(G) \to \Rep(T)$ compatible with forgetful functors. By \cite[Corollary 2.9]{dm}, the latter functor is induced by a group morphism $T \to G$. It is proved in \cite{mv} that this morphism is injective, and identifies $T$ with a maximal torus of $G$. Hence from now on we will consider $T$ as a subgroup of $G$. By construction, if $\la \in \bX$ and if $\cF$ is in $\Perv_{\Gv(\OO)}(\Gr_{\Gv})$ we have a canonical isomorphism of $\C$-vector spaces
\beq{eqn:weight-space-Satake}
\coH^{\la(2\rhov)} \bigl( \fT_{\lambda}, t_{\lambda}^! \cF \bigr) \cong \bigl( \cS_{\Gv}(\cF) \bigr)_\la.
\eeq

The choice of $\Bv \subset \Gv$ determines a set of positive roots for $\Gv$, hence also a set of positive roots for $G$. We denote by $B$ the Borel subgroup of $G$ containing $T$ associated with this set of positive roots.

\subsection{Equivariant cohomology}
\label{ss:equiv-cohomology}

For the results stated below, see e.g.~\cite[\S 1]{lu}.

For any complex algebraic variety $X$ endowed with an algebraic action of an algebraic group $H$, recall that the equivariant cohomology (respectively Borel--Moore homology) is defined by
\[
\coH_H^{\hdot}(X):=\Ext^{\hdot}_{\cD_{H}(X)}(\underline{\bbc}_X,\underline{\bbc}_X), \quad \coH^H_{\idot}(X):=\Ext^{\hdot-2\dim(X)}_{\cD_H(X)}(\underline{\bbc}_X,\underline{\mathbb{D}}_X),
\]
where $\cD_H(X)$ is the $H$-equivariant derived category of $X$,  $\underline{\bbc}_X$ is the (equivariant) constant sheaf on $X$ and $\underline{\mathbb{D}}_X$ is the (equivariant) dualizing sheaf. Then $\coH_H^{\hdot}(X)$ is a (graded-commutative) algebra for the Yoneda product (or cup product) and $\coH^H_{\idot}(X)$ is a right module over this algebra, again for the Yoneda product. If $X$ is smooth, then this module is free of rank $1$.

Let now $K$ be a torus, with Lie algebra $\fk$. Let $\lambda \in X^*(K)$, and consider the $1$-dimensional $K$-module $\C_{\lambda}$, considered as a $K$-variety. Consider the $K$-equivariant morphisms
\[
\{\pt\} \xrightarrow{\iota} \C_{\lambda} \xrightarrow{\pi} \{\pt\},
\]
(where $\iota(\pt)=0$), and the induced morphisms in equivariant homology
\[
\xymatrix{
\coH_{\idot}^K(\pt) \ar[r]^-{\iota_!} & \coH_{\idot+2}^K(\C_{\lambda})  \ar[r]^-{(\pi^*)^{-1}}_-{\sim} & \coH_{\idot+2}^K(\pt),
}
\]
which are morphisms of right $\coH^{\hdot}_{K}(\pt)$-modules. Since the right $\coH^{\hdot}_K(\pt)$-module $\coH_{\idot}^K(\pt)$ is free of rank $1$, there exists a unique element $c(\lambda) \in \coH^2_K(\pt)$ such that the composition above is the action of $c(\lambda)$, and $c : X^*(K) \to \coH^{\hdot}_K(\pt)$ is a morphism of abelian groups. There exists a unique isomorphism $a : \fk^* \xrightarrow{\sim} \coH^2_K(\pt)$ such that the following diagram commutes:
\[
\xymatrix@R=0.5cm{
& X^*(K) \ar[rd]^-{c} \ar[ld]_-{d} &  \\
\fk^* \ar[rr]^-{a}_-{\sim} & & \coH^2_K(\pt),
}
\]
where $d$ is the differential. Moreover, this isomorphism extends to an isomorphism of graded $\bbc$-algebras
\begin{equation}
\label{eqn:torus-equivariant-cohomology}
\mathrm{S}(\fk^*) \xrightarrow{\sim} \coH^{\hdot}_K(\pt),
\end{equation}
where in the left-hand side $\fk^*$ is in degree $2$. We will use this isomorphism throughout the paper without further details. In particular we can identify
\[
\coH^{\hdot}_A(\pt) \cong \sh, \qquad \coH^{\hdot}_{\Tv}(\pt) \cong \sym(\t),
\]
where $\hb \in \sh$ corresponds to the natural generator of $\coH_{\C^\times}^2(\pt) \cong \C$.

If $V$ is any $K$-module, with $K$-weights $\lambda_1, \cdots, \lambda_n$ (counted with multiplicities), then if as above $\iota$ denotes inclusion of $0$ and $\pi$ projection to $0$, via isomorphism \eqref{eqn:torus-equivariant-cohomology}, the composition
\[
\xymatrix{
\coH_{\idot}^K(\pt) \ar[r]^-{\iota_!} & \coH_{\idot+2n}^K(V) \ar[r]^-{(\pi^*)^{-1}}_-{\sim} & \coH_{\idot+2n}^K(\pt)
}
\]
identifies with the action of $d(\lambda_1) \cdots d(\lambda_n)$. Note that the action map induces an isomorphism of $\coH^{\hdot}_K(\pt)$-modules $\coH^K_0(V) \otimes \coH^{\hdot}_K(\pt) \xrightarrow{\sim} \coH_{\idot}^K(V)$, and that the forgetful map $\coH^K_0(V) \to \coH_0(V)$ is an isomorphism. Hence we obtain a canonical isomorphism $\coH_{\idot}^K(V) \cong \coH_0(V) \otimes \coH^{\hdot}_K(\pt)$. The Borel--Moore homology $\coH_0(V)$ contains the canonical class $[V]$, which can therefore be viewed as a generator of $\coH_{\idot}^K(V)$. Similarly we have the canonical class $[\pt] \in \coH_{\idot}^K(\pt)$. Then we have $\pi^*[\pt]=[V]$, hence the morphism $\iota_!$ has the property that
\[
\iota_!([\pt]) \ = \ [V] \cdot d(\lambda_1) \cdots d(\lambda_n).
\]

We can now give a proof of Lemma \ref{lem:equiv-cohomology-first-properties}.

\begin{proof}[Proof of Lemma {\rm \ref{lem:equiv-cohomology-first-properties}}]
To fix notation we treat the case of $A$; the case of $\Tv$ is similar. Let $\cF$ in $\Perv_{\Gv(\OO)}(\Gr_{\Gv})$ and $\la \in \bX$.

(1) 
Using the Leray--Serre spectral sequence for an appropriate fibration,
there exists a spectral sequence which computes $\coH^{\hdot}_{A}(i_{\lambda}^! \cF)$ and with $E_2$-term 
\begin{equation}
\label{eqn:spectral-sequence-equiv-cohomology}
E_2^{p,q} \ = \ \coH^p_{A}(\pt) \o  \coH^{q}(i_{\lambda}^! \cF)
\end{equation}
(see e.g.~\cite[Proof of Proposition 7.2]{lu}).
It is well known also (see \cite[Theorem 5.5]{kl} or \cite[Corollaire 2.10]{sp1}) that the ordinary cohomology $\coH^{\hdot}(i_{\lambda}^! \cF)$ is concentrated in degrees of constant parity. Hence the spectral sequence \eqref{eqn:spectral-sequence-equiv-cohomology} degenerates at $E_2$. It follows that there exists a non-canonical isomorphism of graded $\sh$-modules
\[
\coH^{\hdot}_A(i_{\lambda}^! \cF) \cong \coH^{\hdot}_{A}(\pt) \o \coH^{\hdot}(i_{\lambda}^! \cF).
\]
In particular, the left-hand side is free over $\sh = \coH^{\hdot}_{A}(\pt)$.

(2) By \cite[Theorem 3.4]{mv}, $\coH^{\hdot}(\fT_{\lambda}, t_{\lambda}^! \cF)$ is concentrated in degree $\lambda(2\rhov)$. Hence the same spectral sequence argument as in (1) shows that $\coH^{\hdot}_{A}(\fT_{\lambda}, t_{\lambda}^! \cF)$ is free over $\coH^{\hdot}_A(\pt)$. Moreover, the morphism
\[
\coH^{\hdot}_{A}(\fT_{\lambda}, t_{\lambda}^! \cF) \to \coH^{\hdot}(\fT_{\lambda}, t_{\lambda}^! \cF)
\]
induced by the forgetful functor is an isomorphism in degree $\lambda(2\rhov)$. Hence we obtain an inclusion $\bigl( \cS_{\Gv}(\cF) \bigr)_{\lambda} \overset{\eqref{eqn:weight-space-Satake}}{\cong} \coH^{(\lambda,2\rhov)}(\fT_{\lambda}, t_{\lambda}^! \cF) \hookrightarrow \coH^{\hdot}_{A}(\fT_{\lambda}, t_{\lambda}^! \cF)$. Using again the spectral sequence argument, the morphism
\[
\coH^{\hdot}_A(\pt) \o \bigl( \cS_{\Gv}(\cF) \bigr)_{\lambda} \ \to \ \coH^{\hdot}_{A}(\fT_{\lambda}, t_{\lambda}^! \cF)
\]
induced by the cup product is an isomorphism of $\coH^{\hdot}_A(\pt)$-modules, which finishes the proof.
\end{proof}

We deduce the following result from Lemma \ref{lem:equiv-cohomology-first-properties}.

\begin{cor}
\label{cor:kgeom-injective}
For $\la \in \bX$ and $\cF$ in $\Perv_{\Gv(\OO)}(\Gr_\Gv)$,
The morphism $\kgeom_{\cF,\la}$ is injective.
\end{cor}

\begin{proof}
By the localization theorem in equivariant cohomology, the morphism
\[
\qh \o_{\sh} \coH^\hdot_A( i_\la^! \cF ) \to \qh \o_{\sh} \coH^\hdot_A( \fT_\la, i_\la^! \cF )
\]
induced by $(\iota_\la)_!$
is an isomorphism (since $\bla$ is the only $A$-fixed point in $\fT_\la$). As $\coH^\hdot_A( i_\la^! \cF )$ is free over $\sh$ (see Lemma \ref{lem:equiv-cohomology-first-properties}(1)), we deduce that $(\iota_\la)_!$
is injective, which implies that $\kgeom_{\cF,\la}$ is also injective.
\end{proof}

We will also need the following result, which again follows from the fact that the spectral sequence \eqref{eqn:spectral-sequence-equiv-cohomology} degenerates.

\begin{lem}
\label{lem:forget-hbar}

Let $\cF$ in $\Perv_{\Gv(\OO)}(\Gr_{\Gv})$ and $\la \in \bX$.

The forgetful morphism $\coH^{\hdot}_{A}(i_{\lambda}^! \cF) \to \coH^{\hdot}_{\Tv}(i_{\lambda}^! \cF)$ induces an isomorphism
\[
\coH^{\hdot}_{A }(i_{\lambda}^! \cF) \big/ \big( \hbar \cdot \coH^{\hdot}_{A}(i_{\lambda}^! \cF) \bigr) \ \xrightarrow{\sim} \ \coH^{\hdot}_{\Tv}(i_{\lambda}^! \cF).
\]

\end{lem}

\subsection{Restriction to a Levi subgroup}
\label{ss:Satake-restriction}

Below we will make extensive use of the geometric description of the functor of restriction to a Levi subgroup, due to Mirkovi{\'c}--Vilonen \cite{mv} in the (crucial) case of the maximal torus, and to Beilinson--Drinfeld \cite{bd} in the general case.

Let $\Pv \subset \Gv$ be a parabolic subgroup containing $\Bv$, and let $\Lv \subset \Pv$ be the (unique) Levi factor containing $\Tv$. Let also $\Pv^-$ be the opposite parabolic subgroup. Note that $\Tv$ is also a maximal torus of $\Lv$, and $\Bv_{\Lv}:=\Bv \cap \Lv$ is a Borel subgroup containing $\Tv$.
We have dual groups $G$ and $L$ defined as in \S\ref{ss:Satake}. Consider the diagram
\begin{equation} \label{eqn:affine-Grassmannian-diagram}
\xymatrix@C=2cm{
\Gr_{\Lv} & \Gr_{\Pv^-} \ar[l]_-{q_{\Pv^-}} \ar[r]^-{i_{\Pv^-}} & \Gr_{\Gv}
}
\end{equation}
where $q_{\Pv^-}$ is induced by the projection $\Pv^- \twoheadrightarrow \Lv$ whose kernel is the unipotent radical of $\Pv^-$, and $i_{\Pv^-}$ is induced by the embedding $\Pv^- \hookrightarrow \Gv$. Define the functor
\[
{}' \ResGr^{\Gv}_{\Lv} := (q_{\Pv^-})_* \circ (i_{\Pv^-})^! : \cDbc(\Gr_{\Gv}) \to \cDbc(\Gr_{\Lv}),
\]
where $\cDbc(\Gr_\Gv)$ is the derived category of constructible complexes of $\C$-vector spaces on the ind-variety $\Gr_{\Gv}$ which are supported on a finite union of $\Gv(\OO)$-orbits, and similarly for $\cDbc(\Gr_{\Lv})$.
The functor ${}' \ResGr^{\Gv}_{\Lv}$ does not map the subcategory $\Perv_{\Gv(\OO)}(\Gr_{\Gv})$ of $\cDbc(\Gr_{\Gv})$ into the subcategory $\Perv_{\Lv(\OO)}(\Gr_{\Lv})$ of $\cDbc(\Gr_{\Lv})$; however, the following modification of this functor has this property. 

Recall that the connected components of $\Gr_{\Lv}$ are parametrized by the quotient $\bX/(\Z {\check R}_{\Lv})$, see \cite[Proposition 4.5.4]{bd}. (Here ${\check R}_{\Lv}$ denotes the coroots of $\Lv$, and $\Z {\check R}_{\Lv}$ is the lattice they generate.) If $M$ is in $\cDbc(\Gr_{\Lv})$ and $\chi \in \bX/(\Z {\check R}_{\Lv})$, we denote by $M_{\chi}$ the restriction of $M$ to the corresponding connected component. Define the functor $\ResGr^{\Gv}_{\Lv} : \cDbc(\Gr_{\Gv}) \to \cDbc(\Gr_{\Lv})$\index{ResGr@$\ResGr^{\Gv}_{\Lv}$} by the formula
\[
\ResGr^{\Gv}_{\Lv}(M) = \bigoplus_{\chi \in \bX/(\Z {\check R}_{\Lv})} \, \bigl( {}' \ResGr^{\Gv}_{\Lv}(M) \bigr)_{\chi}[\chi(2\rho_{\Gv} -2\rho_{\Lv} )],
\]
where $\rho_{\Gv}$ and $\rho_{\Lv}$ are the half sums of positive roots of $\Gv$ and $\Lv$. 
It is proved in \cite[Proposition 5.3.29]{bd} that $\ResGr^{\Gv}_{\Lv}$ restricts to a functor
\[
\ResGr^{\Gv}_{\Lv} : \Perv_{\Gv(\OO)}(\Gr_{\Gv}) \to \Perv_{\Lv(\OO)}(\Gr_{\Lv}).
\]
Moreover, it is explained in \cite[\S5.3.30]{bd} that this functor is a tensor functor.

Using the base change theorem one can easily construct an isomorphism of tensor functors
\[
\sfF_{\Gv} \simto \sfF_{\Lv} \circ \ResGr^{\Gv}_{\Lv},
\]
see e.g.~\cite[\S 4.1]{ahr} for details. We deduce a morphism of algebraic groups
\[
L = \mathrm{Aut}^\star(\sfF_{\Lv}) \to \mathrm{Aut}^\star(\sfF_{\Lv} \circ \ResGr^{\Gv}_{\Lv}) \cong \mathrm{Aut}^\star(\sfF_{\Gv}) = G.
\]
It is known that this morphism is injective, and identifies $L$ with the Levi subgroup of $G$ whose root system is the system of coroots of $\Lv$. Hence we will consider $L$ as a subgroup of $G$. By construction the following diagram is commutative:
\[
\xymatrix@C=2cm{
\Perv_{\Gv(\OO)}(\Gr_{\Gv}) \ar[r]^-{\ResGr^{\Gv}_{\Lv}} \ar[d]_-{\cS_{\Gv}} & \Perv_{\Lv(\OO)}(\Gr_{\Lv}) \ar[d]^-{\cS_{\Lv}} \\
\Rep(G) \ar[r]^-{V \mapsto V_{|L}} & \Rep(L)
}
\]

Note that by the constructions of \S\ref{ss:Satake} we have identified $T$ with a subgroup of $L$, but also with a subgroup of $G$. These identifications are compatible with the identification of $L$ as a subgroup of $G$, see e.g.~\cite[\S 4.1]{ahr}. Moreover, by the base change theorem there exists a canonical isomorphism of functors
\beq{eqn:transitivity-ResGr}
\ResGr^{\Lv}_{\Tv} \circ \ResGr^{\Gv}_{\Lv} \ \cong \ \ResGr^{\Gv}_{\Tv}.
\eeq

\subsection{Restriction to a Levi subgroup: cofibers}
\label{ss:Satake-restriction-cofiber}

Let $\Pv, \Pv^-, \Lv$ be as in \S\ref{ss:Satake-restriction}. Let also $\la \in \bX$ and $\cF$ in $\Perv_{\Gv(\OO)}(\Gr_\Gv)$. We want to compare the $\sh$-modules
\[
\coH^\hdot_A\bigl( (i_\la^\Gv)^! \cF \bigr) \qquad \text{and} \qquad \coH^\hdot_A \bigl( (i_\la^\Lv)^! \ResGr^{\Gv}_{\Lv} (\cF) \bigr),
\]
where $i_\la^\Gv$ and $i_\la^\Lv$ are the inclusions of $\bla$ is $\Gr_\Gv$ and $\Gr_\Lv$ respectively. Let $t_\la^{\Pv}$ be the inclusion of $(q_{\Pv^-})^{-1}(\bla)$ in $\Gr_\Gv$. Then by the base-change theorem there is a canonical isomorphism of graded $\sh$-modules
\[
\coH^\hdot_A \bigl( (i_\la^\Lv)^! \ResGr^{\Gv}_{\Lv} (\cF) \bigr) \ \cong \ \coH^\hdot_A \bigl( (q_{\Pv^-})^{-1}(\bla); (t_\la^\Pv)^! \cF \bigr) \langle \la(2\rhov_L - 2 \rhov_G) \rangle.
\]
As $\{\bla\}$ is closed in $(q_{\Pv^-})^{-1}(\bla)$, the $({}_!, {}^!)$-adjunction for the inclusion $\{\bla\} \hookrightarrow (q_{\Pv^-})^{-1}(\bla)$ induces a canonical morphism of $\sh$-modules
\beq{eqn:restriction-geom}
\coH^\hdot_A\bigl( (i_\la^\Gv)^! \cF \bigr) \to \coH^\hdot_A \bigl( (i_\la^\Lv)^! \ResGr^{\Gv}_{\Lv}(\cF) \bigr) \langle \la(2\rhov_G-2\rhov_L) \rangle.
\eeq

If $\Pv=\Bv$ (so that $\Lv=\Tv$) we have canonical isomorphisms
\[
\coH^\hdot_A \bigl( (i_\la^\Tv)^! \ResGr^{\Gv}_{\Tv}(\cF) \bigr) \ \cong \ \coH^\hdot_A(t_\la^! \cF) \lan -\la(2\rhov_G) \ran \ \cong \ \bigl( \cS_\Gv(\cF) \bigr)_\la \o \sh.
\]
Via this isomorphism, \eqref{eqn:restriction-geom} identifies with the morphism $\kgeom_{\cF,\la}$ defined in \eqref{eqn:equivariant-cohomology-morphism}.
Moreover, one can easily check that the following diagram is commutative:
\[
\xymatrix@C=1.5cm{
\coH^\hdot_A\bigl( (i_\la^\Gv)^! \cF \bigr) \ar@/^1.5pc/[rr]^-{{}^\Gv \hspace{-1pt} \kgeom_{\cF,\la}} \ar[r]_-{\eqref{eqn:restriction-geom}} & \coH^\hdot_A \bigl( (i_\la^\Lv)^! \ResGr^{\Gv}_{\Lv}(\cF) \bigr) \langle \la(2\rhov_G-2\rhov_L) \rangle \ar[r]_-{{}^\Lv \hspace{-1pt} \kgeom_{\cF,\la}} & \bigl( \cS_\Gv(\cF) \bigr)_\la \o \sh \lan \la(2\rhov_G) \ran
}
\]
(where for simplicity we neglect shifts of morphisms).
As the morphism ${}^\Gv \hspace{-1pt} \kgeom_{\cF,\la}$ is injective (see Corollary \ref{cor:kgeom-injective}), we deduce that \eqref{eqn:restriction-geom} is also injective.

%

We will also need the following result, which is clear by construction.

\begin{lem}
\label{lem:Xi-restriction}
Let $\la \in \bX$ and $\cF$ in $\Perv_{\Gv(\OO)}(\Gr_\Gv)$. For any $w \in W_L \subset W_G$ the following diagram commutes:
\[
\xymatrix@C=2cm{
\coH^\hdot_A\bigl( (i_\la^\Gv)^! \cF \bigr) \ar[d]_-{{}^{\Gv} \hspace{-1pt} \Xi^{\cF,\la}_w} \ar[r]^-{\eqref{eqn:restriction-geom}} & \coH^\hdot_A \bigl( (i_\la^\Lv)^! \ResGr^{\Gv}_{\Lv}(\cF) \bigr) \ar[d]^-{{}^{\Lv} \hspace{-1pt} \Xi^{\ResGr^{\Gv}_{\Lv} (\cF),\la}_w} \\
\w \coH^\hdot_A\bigl( (i_{w \la}^\Gv)^! \cF \bigr) \ar[r]^-{\eqref{eqn:restriction-geom}} & \coH^\hdot_A \bigl( (i_\la^\Lv)^! \ResGr^{\Gv}_{\Lv}(\cF) \bigr).
}
\]
\end{lem}

\subsection{Construction of root vectors}
\label{ss:root-vectors}

There are several ways to construct simple root vectors in $\g:=\Lie(G)$ out of the equivalence $\cS$ using \cite{gi} and \cite{mv}, see e.g.~\cite{yz, baumann}. Here we recall a simple version essentially explained in \cite{vasserot}, which will be sufficient for our purposes.

Let $\alpha$ be a simple root. Let $\Pv^\alpha$ be the minimal parabolic subgroup of $\Gv$ containing $\Bv$ associated with $\alpha$, and let $\Lv^\alpha$ be the Levi factor of $\Pv^\alpha$ containing $\Tv$. Then as explained in \S\ref{ss:Satake-restriction}, the ``dual'' group $L^\alpha$ associated with $\Lv^\alpha$ can be canonically identified with the Levi subgroup of $G$ whose Lie algebra is $\g_{-\alpha} \oplus \t \oplus \g_\alpha$, so to construct a root vector in $\g_\alpha$ it is enough to construct a root vector in the $\alpha$-weight space of the Lie algebra $\l^\alpha$ of $L^\alpha$.

Let $\Gr_{\Lv^\alpha}^\circ$ be the connected component of $\boldsymbol{0}$ in $\Gr_{\Lv^\alpha}$. Then the subcategory $\Perv_{\Lv^\alpha(\OO)}(\Gr_{\Lv^\alpha}^\circ)$ of $\Perv_{\Lv^\alpha(\OO)}(\Gr_{\Lv^\alpha})$ is closed under convolution. If we denote by $M^\alpha$ the group of automorphisms of the restriction $\mathsf{F}_{\Lv^\alpha}^\circ$ of the fiber functor $\mathsf{F}_{\Lv^\alpha}$ to this subcategory, then by definition we have a natural morphism $L^\alpha \to M^\alpha$, which induces a morphism $\l^\alpha \to \mathfrak{m}^\alpha$ whose restriction to $\alpha$-weight spaces is an isomorphism. Hence to construct a root vector in $\g_\alpha$ it is enough to construct a root vector in the $\alpha$-weight space of $\mathfrak{m}^\alpha$.

Now let $\cL$ be the positive generator of the Picard group of $\Gr_{\Lv^\alpha}^\circ$. (See \cite[\S 1.4]{vasserot} or \cite[\S 3.3]{baumann} for the explicit construction of this line bundle.) The cup product with the first Chern class of $\cL$ defines an endomorphism of the functor $\mathsf{F}_{\Lv^\alpha}^\circ$. By the arguments of \cite[\S 3.4]{yz}, this endomorphism defines an element in the $\alpha$-weight space of $\mathfrak{m}^\alpha$, which finishes the construction.

This construction of the root vectors is clearly compatible with restriction to a Levi subgroup in the sense of \S\ref{ss:Satake-restriction}.


\section{Proofs of the main results}
\label{sec:proofs}

In Sections \ref{sec:proofs}--\ref{sec:applications} we will use the results of Sections \ref{sec:DX}--\ref{sec:wfg} for the datum $T \subset B \subset G$ constructed in \S\ref{ss:Satake}, and for the root vectors constructed in \S\ref{ss:root-vectors}. Note that all the objects which are denoted by the same symbol in Sections \ref{sec:DX}--\ref{sec:wfg} and in Section \ref{sec:Satake} (e.g.~$\bX$, $\bX^+$, $\sh$) get identified canonically.

%
%
%
%
%

\subsection{Preliminaries on Verma modules}
\label{ss:preliminaries-Verma}

Recall the Verma modules $\V(\mu)$ defined in \S\ref{ss:reminder-Verma}.

Let $\mu \in \t^*$, and assume that for all $\alpha \in R$, $\mu(\alv) \notin \Z$. We claim that for $\nu \in \ft^*$ we have
\begin{equation}
\label{eqn:morphisms-Verma-regular}
\Hom_{U(\g)}(\V(\nu),\V(\mu)) \ = \ \bbc^{\delta_{\mu,\nu}}.
\end{equation}
Indeed, if there exists a non-zero morphism $\V(\nu) \to \V(\mu)$ then by \cite[\S 3.4]{hu} $\nu$ must be in $W_{[\mu]} \bullet \mu$, where $W_{[\mu]}=\{w \in W \mid w(\mu) - \mu \in \Z R\}$.
However, by \cite[Theorem 3.4]{hu} and our assumption on $\mu$, $W_{[\mu]}=\{1\}$, which implies \eqref{eqn:morphisms-Verma-regular}. 

If $\lambda \in \bX$, then the $\t$-weights of the $U(\b)$-module $\V(\mu + \lambda) \o  \bbc_{-\mu}$ (where the $U(\b)$-action is diagonal, the action on $\V(\mu + \lambda)$ being the restriction of the $U(\g)$-action) are in $\bX$, and the action of $U(\u)$ is locally finite. We deduce that the $U(\b)$-action can be integrated to a $B$-action.

\begin{lem}
\label{lem:Verma-induced}

Let $\mu \in \t^*$ be such that $\mu(\alv) \notin \Z$ for all $\alpha \in R$, and let $\lambda \in \bX$. Then there exists a natural isomorphism of $B$-modules
\[
\V(\mu + \lambda) \o  \bbc_{-\mu} \ \simto \ \Ind_T^B(\lambda).
\]

\end{lem}

\begin{proof}
First, let us explain how this morphism in constructed. 
The projection of $\t$-modules $\V(\mu+\la) \to \C_{\mu+\lambda}$ (with kernel $\V(\mu+\la)_-$) induces a morphism of $T$-modules $\V(\mu + \lambda) \o  \C_{-\mu} \to \C_\lambda$. By Frobenius reciprocity we deduce a morphism of $B$-modules 
as in the statement of the lemma.

Now we prove that this morphism is injective. For this, it suffices to prove that its restriction to the socle of the left-hand side (as a $B$-module, or equivalently as a $U(\b)$-module) is injective. We claim that this socle is isomorphic to $\C_{\lambda}$, and has a basis consisting of the vector $1_{\lambda+\mu} \otimes 1 \in \V(\mu + \la) \o  \bbc_{-\mu}$. Indeed, this is equivalent to saying that the socle of $\V(\mu + \lambda)$ is isomorphic to $\bbc_{\mu + \lambda}$. However, if there exists a non-zero morphism $\bbc_{\nu} \to \V(\mu + \lambda)$ for some $\nu \in \t^*$ then 
we obtain a non-zero morphism of $U(\g)$-modules $\V(\nu) \to \V(\mu + \lambda)$. By \eqref{eqn:morphisms-Verma-regular}, this implies that $\nu=\mu+\lambda$, proving the claim and the injectivity of the morphism.

Now, it is easy to see that the $T$-modules $\V(\mu + \lambda) \o  \bbc_{-\mu}$ and $\Ind_T^B(\lambda)$ have the same weights, with the same (finite) multiplicities. Hence our morphism must be an isomorphism.
\end{proof}

Below we will use the standard order on $\t^*$, defined by $\nu \leq \mu$ iff $\mu-\nu \in \Z_{\geq 0} R^+$.

Let $\alpha$ be a simple root, $P^{\alpha} \subset G$ the corresponding minimal parabolic subgroup containing $B$, and $L^\alpha$ the Levi factor of $P^\alpha$ containing $T$. Let also $B^{\alpha} := L^{\alpha} \cap B$, and $\b^{\alpha}:=\mathrm{Lie}(B^{\alpha})$. Let $\mu \in \ft^*$ be such that $\mu({\check \beta}) \notin \Z$ for any $\beta \in R \smallsetminus \{\pm \alpha\}$. We claim that for $\nu \in \t^*$ we have
\begin{equation}
\label{eqn:morphisms-Verma-regular-alpha}
\Hom_{U (\g)}(\V(\nu),\V(\mu)) \ \cong \ 
\begin{cases}
\C & \text{if } \nu=\mu; \\
\C & \text{if } \nu=s_{\alpha} \bullet \mu \text{ and } \nu < \mu; \\
0 & \text{otherwise}.
\end{cases}
\end{equation}
Indeed, if there exists a non-zero morphism $\V(\nu) \to \V(\mu)$ with $\nu \neq \mu$ then, with the same notation as above we must have $\nu \in W_{[\mu]} \bullet \mu$ and $\nu < \mu$. Again by \cite[Theorem 3.4]{hu}, this implies that $\nu=s_{\alpha} \bullet \mu$. On the other hand, if $\nu=s_{\alpha} \bullet \mu$ and $\nu < \mu$, then the $\nu$-weight space of $\V(\mu)$ is one-dimensional, and consists of singular vectors by \cite[Proposition 1.4]{hu}.

In the following lemma, for $\mu \in \t^*$ we denote by $\V^{\alpha}(\mu)$ the Verma module associated with $\mu$ for the reductive group $L^{\alpha}$ with Borel subgroup $B^\alpha$.

\begin{lem}
\label{lem:Verma-induced-alpha}

Let $\mu \in \ft^*$ be such that $\mu({\check \beta}) \notin \Z$ for all $\beta \in R \smallsetminus \{\pm \alpha\}$, and let $\lambda \in \bX$. Then there exists a natural isomorphism of $B$-modules
\[
\V(\mu+\lambda) \o \bbc_{-\mu} \ \simto \ \Ind_{B^{\alpha}}^B \bigl( \V^{\alpha}(\mu+\lambda) \o \bbc_{-\mu} \bigr).
\]

\end{lem}

\begin{proof}
First, let us explain the construction of this morphism. As in the proof of Lemma \ref{lem:Verma-induced} $\V(\mu+\la) \o \C_{-\mu}$ has a natural structure of $B$-module, and $\V^{\alpha}(\mu+\lambda) \o \bbc_{-\mu}$ has a natural structure of $B^\alpha$-module. The subspace of $\V(\mu+\la) \o \C_{-\mu}$ spanned by weight spaces whose weight is not in $\lambda + \Z\alpha$ is stable under the action of $B^{\alpha}$, and the quotient by this subspace is clearly isomorphic to $\V^{\alpha}(\mu+\la) \otimes \bbc_{-\mu}$. Hence we have constructed a morphism of $B^{\alpha}$-modules $\V(\nu+\lambda) \otimes \bbc_{-\mu} \to \V^{\alpha}(\mu+\la) \otimes \bbc_{-\mu}$. Using Frobenius reciprocity we obtain the desired morphism of $B$-modules.

Now we prove that this morphism is injective. As in Lemma \ref{lem:Verma-induced}, it is enough to prove that its restriction to the socle of the left-hand side is injective. But it follows from \eqref{eqn:morphisms-Verma-regular-alpha} that this socle has dimension $1$ or $2$, and injectivity is clear by construction.

Finally, one can deduce surjectivity as in the proof of Lemma \ref{lem:Verma-induced} by comparing characters.
\end{proof}

\subsection{Generic and sub-generic situations: classical case}

Let $\t^*_{\rs}$\index{t*rs@$\t^*_{\rs}$} and $(\g/\u)^*_{\rs}$ be the sets of elements in $\t^*$ and $(\g/\u)^*$ which are regular semisimple (as elements of $\g^*$). Then $(\g/\u)^*_{\rs} = (\g/\u)^* \times_{\t^*} \t^*_{\rs}$, and the action of $B$ on $(\g/\u)^*$ induces an isomorphism of $B$-varieties
\begin{equation}
\label{eqn:classical-case-generic-geometry}
B/T \times \t^*_{\rs} \ \simto \ (\g/\u)^*_{\rs}
\end{equation}
where the $B$-action on the left-hand side is trivial on $\t^*_{\rs}$, and given by left multiplication on $B/T$
(see e.g.~\cite[p.~188]{jantzen}).
In particular, we deduce that for any $\lambda \in \bX$ there is a natural isomorphism of $B$-modules
\begin{equation}
\label{eqn:classical-case-generic}
\C[\t^*_{\rs}] \otimes_{\mathrm{S}(\t)} \mathrm{S}(\g/\u) \o  \C_{-\lambda} \ \cong \ \C[\t^*_{\rs}] \o  \Ind_T^B(-\lambda).
\end{equation}

Now let $\alpha$ be a simple root, and let $P^\alpha$, $L^{\alpha}$, $B^{\alpha}$ be defined as in \S\ref{ss:preliminaries-Verma}. Let also $\l^{\alpha}$ be the Lie algebra of $L^{\alpha}$, and $\u^{\alpha}:=\l^{\alpha} \cap \u$. Let $\t^*_{\alpha-\rs}$ be the complement in $\t^*$ of the collection of hyperplanes defined by the equations ${\check \beta}$ for $\beta \in R \smallsetminus \{\pm \alpha\}$. Let also $(\g/\u)^*_{\alpha-\rs}=(\g/\u)^* \times_{\t^*} \t^*_{\alpha-\rs}$, and $(\l^{\alpha}/\u^{\alpha})^*_{\alpha-\rs}=(\l^{\alpha}/\u^{\alpha})^* \times_{\t^*} \t^*_{\alpha-\rs}$. 

\begin{lem}
\label{lem:geometry-subregular}
The (coadjoint) action of $B$ on $(\g/\u)^*$ induces an isomorphism of $B$-varieties
\[
B \times_{B^{\alpha}} (\l^{\alpha}/\u^{\alpha})^*_{\alpha-\rs} \ \simto \ (\g/\u)^*_{\alpha-\rs}.
\]

\end{lem}

\begin{proof}
Both varieties under consideration are smooth complex varieties, hence it is enough to prove that the map is bijective. We use the Killing form to get identifications $(\g/\u)^* \cong \b$, $(\l^{\alpha}/\u^{\alpha})^* \cong \b^{\alpha}$ (where $\b^{\alpha}:=\Lie(B^\alpha)$), and define $\b_{\alpha-\rs}$, $\b^\alpha_{\alpha-\rs}$ in an obvious way. 

First, let $x \in \fb_{\alpha-\rs}$, and consider the Jordan decomposition $x=s+n$. There exists $u \in B$ such that $u \cdot s \in \ft$, and we must have $u \cdot s \in \ft_{\alpha-\rs}$. Then $u \cdot n \in Z_{\fb}(u \cdot s) \subset \fb^{\alpha}$, which implies that $u \cdot x \in \fb^{\alpha}_{\alpha-\rs}$. This proves surjectivity of our map.

Next, we prove injectivity. Let $u_1, u_2 \in B$ and $x_1,x_2 \in \b^{\alpha}_{\alpha-\rs}$, and assume that $u_1 \cdot x_1 = u_2 \cdot x_2$. Consider the Jordan decompositions $x_1=s_1+n_1$, $x_2=s_2+n_2$, so that $u_1 \cdot s_1=u_2 \cdot s_2$. Conjugating if necessary $x_1$ by an element of $B^{\alpha}$ (and modifying $u_1$ accordingly), one can assume that $s_1 \in \t_{\alpha-\rs}$. Then the fact that $(u_2^{-1} u_1) \cdot s_1 = s_2 \in \b^{\alpha}$ implies that $u_2^{-1} u_1 \in B^{\alpha}$. We deduce that
\[
(u_1 \times_{B^\alpha} x_1) = (u_2 (u_2^{-1} u_1) \times_{B^\alpha} (u_1^{-1} u_2) \cdot x_2) = (u_2 \times_{B^\alpha} x_2)
\]
in $B \times_{B^\alpha} \b^\alpha_{\alpha-\rs}$, which finishes the proof.
\end{proof}

In particular it follows from Lemma \ref{lem:geometry-subregular} that for any $\lambda \in \bX$ there is a natural isomorphism of $B$-modules
\begin{equation}
\label{eqn:classical-case-subgeneric}
\bbc[\t^*_{\alpha-\rs}] \o_{\mathrm{S}(\t)} \mathrm{S}(\g/\u) \o  \bbc_{-\lambda} \ \cong \ \Ind_{B^{\alpha}}^B \bigl( \bbc[\t^*_{\alpha-\rs}] \o_{\mathrm{S}(\t)} \mathrm{S}(\l^{\alpha}/\u^{\alpha}) \o \C_{-\lambda} \bigr).
\end{equation}

\begin{rem}
The same arguments as in Lemma \ref{lem:geometry-subregular} can be used to prove the following more general claim, which will not be used in this paper. Let $I$ be a set of simple roots, and consider the associated Levi subgroup $L^I$ containing $T$. Let $B^I=B \cap L^I$, $\u^I=\mathrm{Lie}(U \cap L^I)$. Let $\ft^*_{I-\rs}$ be the complement in $\ft^*$ of the collection of hyperplanes defined by the coroots associated to roots in $R \smallsetminus \Z I$. Let $(\g/\u)^*_{I-\rs}=(\g/\u)^* \times_{\t^*} \t^*_{\rs}$, and similarly for $(\l^I/\u^I)^*_{I-\rs}$. Then the coadjoint action of $B$ induces an isomorphism of varieties $B \times_{B^I} (\l^I/\u^I)^*_{I-\rs} \ \xrightarrow{\sim} \ (\g/\u)^*_{I-\rs}$.
\end{rem}

\subsection{Generic and sub-generic situations: quantum case}
\label{ss:generic-quantum}

Set $\a:=\Lie(A) = \t \times \bA^1$.\index{a@$\a$} We will identify $\a^*$ with $\t^* \times \bA^1$ in the natural way. We denote by $\C[\a^*_{\rs}]$\index{Ca*rs@$\C[\a^*_{\rs}]$} the localization of $\sh$ with respect to the collection $\{\alv + n \hbar \mid \alpha \in R, \, n \in \Z\}$. If $(\nu,a) \in \a^*$, we denote by $\C_{\sh}(\nu,a)$ the one-dimensional $\sh$-module associated with $(\nu,a)$.

\begin{lem}
\label{lem:quantum-case-generic}

For any $\lambda \in \bX$, there exists a natural isomorphism of $B$-modules and $\bbc[\a^*_{\rs}]$-modules
\[
\bbc[\a^*_{\rs}] \otimes_{\sh} \M(\la) \ \simto \ \bbc[\a^*_{\rs}] \o \Ind_T^B(-\la).
\]

\end{lem}

\begin{proof}
Recall that there exists a morphism of $T$-modules and $\sh$-modules $\M^G(\lambda) \to \M^T(\lambda) = {\sh} \o  \bbc_{-\lambda}$, see \eqref{eqn:Verma-restriction}. Using Frobenius reciprocity we deduce a morphism as in the statement of the lemma.

First, we claim that for any $(\nu,a) \in \a^*$ such that $\nu(\alv) \notin a \Z$ for all $\alpha \in R$, the induced morphism of $B$-modules
\begin{equation}
\label{eqn:morphism-generic-quantum}
\bbc_{\sh}(\nu,a) \otimes_{\sh} \M(\lambda) \ \to \ \Ind_T^B(-\lambda)
\end{equation}
is an isomorphism. 
Indeed, if $a=0$, then $\nu \in \t^*_{\rs}$, and morphism \eqref{eqn:morphism-generic-quantum} can be identified with the specialization of isomorphism \eqref{eqn:classical-case-generic} at $\nu$, hence we are done.
Now, assume $a \neq 0$. Then there is an algebra isomorphism $\uh(\g)^{\mathrm{op}} \otimes_{\bbc[\hbar]} \bbc_a \cong U (\g)$ which maps $x \in \fg \subset \uh(\g)$ to $-a x$ (see \S\ref{ss:asymptotic-Verma}). Using this isomorphism, the left-hand side of \eqref{eqn:morphism-generic-quantum} identifies with the $U(\b)$-module $\V(-\frac{1}{a} \nu - \lambda-\rho) \o  \bbc_{\frac{1}{a} \nu+\rho}$. By our assumption $-\frac{1}{a} \nu - \rho$ satisfies the assumptions of Lemma \ref{lem:Verma-induced}, hence \eqref{eqn:morphism-generic-quantum} is an isomorphism in this case also.

Now we deduce that our morphism of $\bbc[\a^*_{\rs}]$-modules is an isomorphism. As this morphism is $B$-equivariant, it is sufficient to prove that its restriction to each $T$-weight space is an isomorphism. It is easily checked that for any $\lambda \in \bX$ the $\lambda$-weight spaces of both modules are free $\bbc[\a^*_{\rs}]$-modules of the same rank. Choose a basis of each of these spaces, and consider the determinant $d_\la \in \bbc[\a^*_{\rs}]$ of the restriction of our morphism in these bases. Then for each $(\nu,a) \in \a^*$ such that $\nu(\alv) \notin a \Z$ for all $\alpha \in R$, we have $d_\la(\nu,a) \neq 0$. However a polynomial $P \in \sh$ which does not vanish on any hyperplane $\alv + n \hb=0$ (with $\alpha \in R$, $n \in \Z$) is necessarily a scalar multiple of a product of polynomials of the form $\alv + n \hb$. Hence $d_\la$ is invertible in the algebra $\bbc[\a^*_{\rs}]$, which proves that our morphism is indeed an isomorphism.
\end{proof}

Let now $\alpha$ be a simple root, and let $L^\alpha$ be the Levi subgroup defined in \S\ref{ss:preliminaries-Verma}, with its Borel subgroup $B^\alpha$. For $\lambda \in \bX$, we denote by $\M^{\alpha}(\lambda)$ the asymptotic universal Verma module associated with $\la$ for the group $L^\alpha$ and its Borel subgroup $B^\alpha$. Let $\C[\a^*_{\alpha-\rs}]$\index{Ca*alphars@$\C[\a^*_{\alpha-\rs}]$} be the localization of $\sh$ with respect to the collection $\{ {\check \beta} + n \hbar \mid \beta \in R \smallsetminus \{\pm \alpha\}, n \in \Z\}$. 

\begin{lem}
\label{lem:quantum-case-subgeneric}

For any $\lambda \in \bX$, there exists a natural isomorphism of $B$-modules and $\bbc[\a^*_{\alpha-\rs}]$-modules
\[
\bbc[\a^*_{\alpha-\rs}] \otimes_{\sh} \M(\lambda) \ \to \ \Ind_{B^{\alpha}}^B \bigl( \bbc[\a^*_{\alpha-\rs}] \otimes_{\sh} \M^{\alpha}(\lambda) \bigr)
\]
where in the right-hand side $B^\alpha$ acts trivially on $\bbc[\a^*_{\alpha-\rs}]$.

\end{lem}

\begin{proof}
As in Lemma \ref{lem:quantum-case-generic}, the morphism in question is constructed using the morphism $\M(\lambda) \to \M^{\alpha}(\lambda)$ of  \eqref{eqn:Verma-restriction} and Frobenius reciprocity.

To prove that this morphism is an isomorphism
it is enough to prove that for any $(\nu,a) \in \a^*$ such that $\nu({\check \beta}) \notin a \Z$ for all $\beta \in R \smallsetminus \{\pm \alpha\}$, the natural morphism of $B$-modules
\begin{equation}
\label{eqn:morphism-subgeneric-quantum}
\bbc_{\sh}(\nu,a) \otimes_{\sh} \M(\lambda) \ \to \ \Ind_{B^{\alpha}}^B \bigl( \bbc_{\sh}(\nu,a) \otimes_{\sh} \M^{\alpha}(\lambda) \bigr)
\end{equation}
is an isomorphism. (Note that the functor $\Ind_{B^{\alpha}}^B$ is exact by \cite[Corollary I.5.13]{ja}, hence specialization commutes with induction here.) 

If $a=0$, then morphism \eqref{eqn:morphism-subgeneric-quantum} can be identified with the specialization of isomorphism \eqref{eqn:classical-case-subgeneric} at $\nu$, hence we are done. If $a \neq 0$, by the same arguments as in the proof of Lemma \ref{lem:quantum-case-generic}, morphism \eqref{eqn:morphism-subgeneric-quantum} can be identified with the morphism
\[
\V \left( -\frac{1}{a}\nu-\lambda-\rho \right) \o \bbc_{\frac{1}{a}\nu + \rho} \ \to \ \Ind_{B^{\alpha}}^B \left( \V^{\alpha} \left( -\frac{1}{a}\nu-\lambda-\rho \right) \o \bbc_{\frac{1}{a}\nu + \rho} \right)
\]
considered in Lemma \ref{lem:Verma-induced-alpha}. Hence it is an isomorphism in this case also.
\end{proof}

\subsection{Generic and sub-generic situations: isomorphism}
\label{ss:generic-isomorphism}

Fix $\la \in \bX$ and $\cF$ in $\Perv_{\Gv(\OO)}(\Gr_\Gv)$. To simplify notation we set $V:=\cS_\Gv(\cF)$. 

As a first step towards proving Theorem \ref{thm:loop-equivariant-version}, we construct a canonical isomorphism of $\C[\a^*_\rs]$-modules
\beq{eqn:isom-regular}
\C[\a^*_\rs] \o_{\sh} \bigl( V \o \M(\la) \bigr)^B \ \cong \ \C[\a^*_\rs] \o_{\sh} \coH_A^\hdot (i_\la^! \cF).
\eeq
In fact, we deduce this isomorphism from the fact that the morphism
\beq{eqn:k-regular}
\C[\a^*_\rs] \o_{\sh} \coH_A^\hdot (i_\la^! \cF) \to \C[\a^*_\rs] \o V_\la, \quad \text{resp.} \quad
\C[\a^*_\rs] \o_{\sh} \bigl( V \o \M(\la) \bigr)^B  \to \C[\a^*_\rs] \o V_\la
\eeq
induced by $\kgeom_{\cF,\la}$, resp.~$\kalg_{V,\la}$, is an isomorphism. For the first morphism this follows from the localization theorem in equivariant cohomology. Let us now prove that the second morphism is also an isomorphism. In fact we have a series of isomorphisms
\begin{multline*}
\bbc[\a^*_{\rs}] \o_{\sh} \bigl( V \otimes \M(\la) \bigr)^B
\ \cong \ \bigl( V \otimes ( \bbc[\a^*_{\rs}] \o_{\sh} \M(\la)) \bigr)^B 
\ \cong \ \bigl( V \o \bbc[\a^*_{\rs}] \o \Ind_T^B(-\la) \bigr)^B \\
\cong \bbc[\a^*_{\rs}] \o \bigl( \Ind_T^B(V \otimes \bbc_{-\lambda}) \bigr)^B \ \cong \bbc[\a^*_{\rs}] \o V_{\lambda}.
\end{multline*}
Here the second isomorphism follows from Lemma \ref{lem:quantum-case-generic}, the third one from the tensor identity,
and the last one from the isomorphism $\bigl( \Ind_T^B(V \otimes \bbc_{-\lambda}) \bigr)^B = (V \otimes \bbc_{-\lambda})^T$ given by Frobenius reciprocity. By construction the composition of these isomorphisms is precisely the right-hand morphism in \eqref{eqn:k-regular}.

Now, let $\alpha$ be a simple root. As a second step towards proving Theorem \ref{thm:loop-equivariant-version}, we want to show that \eqref{eqn:isom-regular} restricts to a canonical isomorphism of $\C[\a^*_{\alpha-\rs}]$-modules
\beq{eqn:isom-subregular}
\C[\a^*_{\alpha-\rs}] \o_{\sh} \bigl( V \o \M(\la) \bigr)^B \ \cong \ \C[\a^*_{\alpha-\rs}] \o_{\sh} \coH_A^\hdot (i_\la^! \cF).
\eeq
Let $\Pv^\alpha$ be the minimal parabolic subgroup of $\Gv$ containing $\Bv$ associated with $\alpha$, and let $\Lv^\alpha$ be the unique Levi factor of $\Pv^\alpha$ containing $\Tv$. By the constructions of \S\ref{ss:Satake-restriction}, these data determine a Levi factor $L^\alpha$ in $G$, hence the corresponding minimal parabolic subgroup $P^\alpha$ containing $B$.

Consider first the right-hand side of \eqref{eqn:isom-subregular}. We will use the constructions of \S\ref{ss:Satake-restriction-cofiber} for the Levi subgroup $\Lv^\alpha$.
By the localization theorem in equivariant cohomology, the morphism
\[
\C[\a^*_{\alpha-\rs}] \o_{\sh} \coH_A^\hdot (i_\la^! \cF) \ \to \ \C[\a^*_{\alpha-\rs}] \o_{\sh} \coH^\hdot_A \bigl( (i_\la^{\Lv^\alpha})^! \ResGr^{\Gv}_{\Lv^\alpha}(\cF) \bigr)
\]
induced by \eqref{eqn:restriction-geom} is an isomorphism.

Consider now the left-hand side of \eqref{eqn:isom-subregular}. Here we will use the constructions of \S\ref{ss:generic-quantum}. We have a series of isomorphisms
\begin{multline*}
\bbc[\a^*_{\alpha-\rs}] \otimes_{\sh} \bigl( V \otimes \M(\la) \bigr)^B 
\ \cong \ \bigl( V \otimes (\bbc[\a^*_{\alpha-\rs}] \otimes_{\sh} \M(\la)) \bigr)^B 
\ \cong \ \bigl( V \o \Ind^B_{B^\alpha} ( \bbc[\a^*_{\alpha-\rs}] \otimes_{\sh} \M^{\alpha}(\la) ) \bigr)^B \\
\cong \bbc[\a^*_{\alpha-\rs}] \otimes_{\sh} \bigl( \Ind^B_{B^\alpha} ( V_{|L^{\alpha}} \o \M^{\alpha}(\la) ) \bigr)^{B}
\ \cong \ \bbc[\a^*_{\alpha-\rs}] \otimes_{\sh} \bigl( V_{|L^{\alpha}} \o \M^{\alpha}(\lambda) \bigr)^{B^{\alpha}}.
\end{multline*}
Here the second isomorphism follows from Lemma \ref{lem:quantum-case-subgeneric}, the third one from the tensor identity,
and the last one from the isomorphism $\bigl( \Ind^B_{B^\alpha} ( V_{|L^{\alpha}} \o \M^{\alpha}(\la) ) \bigr)^{B}
\cong \bigl( V_{|L^{\alpha}} \o \M^{\alpha}(\lambda) \bigr)^{B^{\alpha}}$ given by Frobenius reciprocity.

Using these isomorphisms, to construct \eqref{eqn:isom-subregular} we only have to construct a canonical isomorphism of $\sh$-modules
\[
\bigl( V_{|L^{\alpha}} \o \M^{\alpha}(\lambda) \bigr)^{B^{\alpha}} \ \cong \ \coH^\hdot_A \bigl( (i_\la^{\Lv^\alpha})^! \ResGr^{\Gv}_{\Lv^\alpha}(\cF) \bigr).
\]
As $\Lv^\alpha$ has semisimple rank one, this isomorphism is constructed in \S\ref{ss:thm-cofibers-rk1} below.

One can check that the isomorphism obtained from \eqref{eqn:isom-subregular} by extension of scalars to $\C[\a^*_\rs]$ coincides with \eqref{eqn:isom-regular}. In other words, \eqref{eqn:isom-subregular} is the restriction of \eqref{eqn:isom-regular} to
\[
\bbc[\a^*_{\alpha-\rs}] \otimes_{\sh} \bigl( V \otimes \M(\la) \bigr)^B \subset \bbc[\a^*_{\rs}] \otimes_{\sh} \bigl( V \otimes \M(\la) \bigr)^B.
\]

\subsection{Equivariance}

Recall the operators $\Omega^{V,\la}_w$ defined in \S\ref{ss:relation-Phi-Theta}.

\begin{lem}
\label{lem:isom-regular-equivariance}
Let $\la \in \bX$ and $\cF$ in $\Perv_{\Gv(\OO)}(\Gr_\Gv)$.
Isomorphism \eqref{eqn:isom-regular} is $W$-equivariant, in the sense that for any $w \in W$ the following diagram commutes:
\[
\xymatrix@C=2cm{
\C[\a^*_\rs] \o_{\sh} \bigl( \cS(\cF) \o \M(\la) \bigr)^B \ar[r]^-{\eqref{eqn:isom-regular}} \ar[d]_-{\C[\a^*_\rs] \o_{\sh} \Omega^{\cS(\cF),\la}_w} & \C[\a^*_\rs] \o_{\sh} \coH_A^\hdot (i_\la^! \cF) \ar[d]^-{\C[\a^*_\rs] \o_{\sh} \Xi^{\cF,\la}_w} \\
\C[\a^*_\rs] \o_{\sh} \w \Bigl( \bigl( \cS(\cF) \o \M(w\la) \bigr)^B \Bigr) \ar[r]^-{\eqref{eqn:isom-regular}} & \C[\a^*_\rs] \o_{\sh} \w \Bigl( \coH_A^\hdot (i_{w \la}^! \cF) \Bigr)
}
\]
\end{lem}

\begin{proof}
It is enough to prove the commutativity when $w$ is a simple reflection. So let $\alpha$ be a simple root. Recall that isomorphism \eqref{eqn:isom-regular} is the restriction of isomorphism \eqref{eqn:isom-subregular} to $\a^*_\rs$. Now isomorphism \eqref{eqn:isom-subregular} is deduced from the isomorphism of Theorem \ref{thm:loop-equivariant-version} for $\Lv^\alpha$, proved (directly) in \S\ref{ss:thm-cofibers-rk1} below. Moreover, by Corollary \ref{cor:Phi-rest} (or Lemma \ref{lem:Theta-rest-Levi}) and Lemma \ref{lem:Xi-restriction}, the operators $\Omega$ and $\Xi$ for the group $\Gv$ can also be constructed from the similar operators for the group $\Lv^\alpha$. Hence the commutativity for $s_\alpha$ follows from Theorem \ref{thm:W-symmetry} for $\Lv^\alpha$, which is proved (directly) in \S\ref{ss:W-symmetry-rk1} below.
\end{proof}

\subsection{Proof of the main results: quantum case}
\label{ss:proof-main-results}

If $\alpha \in R$ is any root, we define the localization $\C[\a^*_{\alpha-\rs}]$ of $\sh$ by the same recipe as for simple roots (see \S\ref{ss:generic-quantum}). Then for any $w \in W$ we have $w(\C[\a^*_{\alpha-\rs}]) = \C[\a^*_{w(\alpha)-\rs}]$ as subalgebras of $\C[\a^*_{\rs}]$
To finish the proofs we will need the following obvious lemma.

\begin{lem}
\label{lem:morphisms-sh}
Let $M$ and $N$ be free $\sh$-modules of finite rank, and let 
\[
\varphi : \C[\a^*_{\rs}] \o_{\sh} M \simto \C[\a^*_{\rs}] \o_{\sh} N
\]
be an isomorphism of $\C[\a^*_{\rs}]$-modules. Assume that, for any $\alpha \in R$, $\varphi$ restricts to an isomorphism of $\C[\a^*_{\alpha-\rs}]$-modules $\C[\a^*_{\alpha-\rs}] \o_{\sh} M \simto \C[\a^*_{\alpha-\rs}] \o_{\sh} N$. Then $\varphi$ restricts to an isomorphism $M \simto N$.
\end{lem}

\begin{proof}[Proof of Theorem {\rm \ref{thm:loop-equivariant-version}}]
Let $\la \in \bX$ and $\cF$ in $\Perv_{\Gv(\OO)}(\Gr_\Gv)$, and set $V:=\cS(\cF)$.

Injectivity of $\kalg_{V,\la}$ follows from Lemma \ref{lem:restH-injective}, while injectivity of $\kgeom_{\cF,\la}$ is proved in Corollary \ref{cor:kgeom-injective}. By Lemma \ref{lem:equiv-cohomology-first-properties}(1), $\coH_A^\hdot (i_\la^! \cF)$ is a free $\sh$-module. It follows from Proposition \ref{prop:freeness} and the first isomorphism in Lemma \ref{lem:fixed-points-morphisms} that the same is true for $\bigl( V \o \M(\la) \bigr)^B$. In \eqref{eqn:isom-regular} we have constructed an isomorphism between the extensions of scalars of these $\sh$-modules to $\C[\a^*_\rs]$. By \eqref{eqn:isom-subregular} this isomorphism restricts to an isomorphism bewteen extensions of scalars to $\C[\a^*_{\alpha-\rs}]$ for any simple root $\alpha$. From Lemma \ref{lem:isom-regular-equivariance} we deduce that the same property is true for \emph{any} root. Hence the isomorphism follows from Lemma \ref{lem:morphisms-sh}.
\end{proof}

\begin{proof}[Proof of Theorem {\rm \ref{thm:W-symmetry}}]
As our $\sh$-modules are free, it is enough to prove commutativity after restriction to $\a^*_\rs \subset \a^*$. In this setting the claim follows from Lemma \ref{lem:isom-regular-equivariance}.
\end{proof}

\subsection{Classical analogues}
\label{ss:proof-main-results-classical}

\begin{proof}[Proof of Theorem {\rm \ref{thm:equiv-coh-classical}}]
One can prove Theorem \ref{thm:equiv-coh-classical} using exactly the same strategy as for Theorem \ref{thm:loop-equivariant-version}. Alternatively, isomorphism $\overline{\zeta}_{\cF,\la}$ can be deduced from isomorphism $\zeta_{\cF,\la}$ of Theorem \ref{thm:loop-equivariant-version} using Lemma \ref{lem:hbar=0} and Lemma \ref{lem:forget-hbar}. (The statements about injectivity are easy.)
\end{proof}

\begin{proof}[Proof of Theorem {\rm \ref{thm:W-symmetry-classical}}]
One can prove Theorem \ref{thm:W-symmetry-classical} using exactly the same strategy as for Theorem \ref{thm:W-symmetry}. Alternatively, one can deduce Theorem \ref{thm:W-symmetry-classical} from Theorem \ref{thm:W-symmetry} using Proposition \ref{prop:Phi-sigma}, Lemma \ref{lem:hbar=0} and Lemma \ref{lem:forget-hbar} (see also Remark \ref{rk:Phi-sigma}).
\end{proof}

\section{Complementary results and applications}
\label{sec:applications}

\subsection{Convolution}
\label{ss:convolution}

In this subsection we construct the morphisms $\mathsf{Conv}$ considered in \S\ref{ss:intro-conv}, and prove their compatibility with the isomorphisms of Corollary \ref{key_cor}, thereby finishing the proof of the theorem stated in \S\ref{ss:intro-main-thm}.

First, consider the ``geometric'' setting. Multiplication induces a morphism $\dh(\X) \o_{\C[\hb]} \dh(\X) \to \dh(\X)$. One can check that for $\la,\mu \in \bX$ this morphism induces a morphism of graded $\sh$-modules
\[
{}^{(\la)} \hspace{-1pt} \dh(\X)_\la \o_{\sh} {}^{(\la+\mu)} \hspace{-1pt} \dh(\X)_\mu \to {}^{(\la+\mu)} \hspace{-1pt} \dh(\X)_{\la+\mu}.
\]
Hence for $V,V'$ in $\Rep(G)$ we obtain a morphism of graded $\sh$-modules
\begin{multline*}
\mathsf{Conv}^{\operatorname{geom}}_{V,V',\la,\mu} : {}^{(\la)} \hspace{-1pt} \bigl( V \o \dh(\X)_\la \bigr)^G \langle \la(2\rhov) \rangle \o_{\sh} {}^{(\la+\mu)} \hspace{-1pt} \bigl( V' \o \dh(\X)_\mu \bigr)^G \langle \mu(2\rhov) \rangle \\
\to {}^{(\la+\mu)} \hspace{-1pt} \bigl( V \o V' \o \dh(\X)_{\la+\mu} \bigr)^G \langle (\la+\mu)(2\rhov) \rangle.
\end{multline*}

Now, consider the ``algebraic'' setting. Note that if $\la,\mu \in \bX$ we have a canonical isomorphism of $(\sh,\uh(\g))$-bimodules
\[
\sh \llangle \la \rrangle \o_{\sh} \M(\mu) \ \cong \ \M(\la+\mu)
\]
which sends $1 \o \bv_\mu$ to $\bv_{\la+\mu}$. In particular if $V$ is in $\Rep(G)$ and $\phi : \M(0) \to V \o \M(\mu)$ is a morphism of $(\sh,\uh(\g))$-bimodules, then we can consider $\sh \llangle \la \rrangle \o_{\sh} \phi$ as a morphism of bimodules $\M(\la) \to V \o \M(\la+\mu)$.

Fix $V,V'$ in $\Rep(G)$ and $\la,\mu \in \bX$. Then following \cite[\S 8.4]{abg} we define the morphism of graded $\sh$-modules
\begin{multline*}
\mathsf{Conv}^{\operatorname{alg}}_{V,V',\la,\mu} : \Hom \bigl( \M(0), V \o \M(\la) \bigr) \langle \la(2\rhov) \rangle \o_{\sh} {}^{(\la)} \Hom \bigl( \M(0), V' \o \M(\mu) \bigr) \langle \mu(2\rhov) \rangle \\
\to \Hom \bigl( \M(0), V \o V' \o \M(\la+\mu) \bigr) \langle (\la+\mu)(2\rhov) \rangle
\end{multline*}\index{Convalg@$\mathsf{Conv}^{\operatorname{alg}}_{V,V',\la,\mu}$}%
(where as usual we consider morphisms of $(\sh,\uh(\g))$-bimodules)
which sends a pair $(\phi,\psi)$ to the following composition:
\[
\xymatrix@C=3cm{
\M(0) \ar[r]^-{\phi} & V \o \M(\la) \ar[r]^-{\id_V \o (\sh \llangle \la \rrangle \o_{\sh}\psi)} & V \o V' \o \M(\la+\mu).
}
\]

Finally we consider the ``topological'' setting. Here again we follow \cite[\S 8.7]{abg} (though we have to be more careful because we use \emph{equivariant} cohomology.) Let $\cF,\cG$ in $\Perv_{\Gv(\OO)}(\Gr)$, and let $\la,\mu \in \bX$. The convolution $\cF \star \cG$ is defined in terms of the ``convolution diagram''
\[
\mathsf{mult} : \Gv(\KK) \times_{\Gv(\OO)} \Gr \to \Gr
\]
induced by left multiplication of $\Gv(\KK)$ on $\Gr$. More precisely, $\cF \star \cG = \mathsf{mult}_* (\cF \,\widetilde{\boxtimes} \, \cG)$ where $\cF \, \widetilde{\boxtimes} \, \cG$ is the twisted external product, as defined e.g.~in \cite[\S 4]{mv}.


Let $\nu:=\la+\mu$. Consider the cartesian square
\[
\xymatrix@C=2cm{
\mathsf{mult}^{-1}(\bnu) \ar@{^{(}->}[r]^-{j_{\nu}} \ar[d]_-{\mathsf{mult}_{\nu}} & \Gv(\KK) \times_{\Gv(\OO)} \Gr \ar[d]^-{\mathsf{mult}} \\
\bnu \ar@{^{(}->}[r]^-{i_{\nu}} & \Gr.
}
\]
This diagram is $A$-equivariant if we consider $\Gv(\KK) \times_{\Gv(\OO)} \Gr$ as an $A$-variety where $\Tv$ acts by left multiplication on $\Gv(\KK)$, and $\C^\times$ acts diagonally by loop rotation.
By the base-change theorem we have an isomorphism
\[
i_{\nu}^! (\cF \star \cG) = i_{\nu}^! \mathsf{mult}_* (\cF \, \widetilde{\boxtimes} \, \cG) \cong (\mathsf{mult}_{\nu})_* j_{\nu}^! (\cF \, \widetilde{\boxtimes} \, \cG),
\]
so that we obtain an isomorphism
\[
\coH^\hdot_A \bigl( i_{\nu}^! (\cF \star \cG) \bigr) \ \cong \ \coH^\hdot_A \bigl(\mathsf{mult}^{-1}(\bnu), j_{\nu}^! (\cF \, \widetilde{\boxtimes} \, \cG) \bigr).
\]
Now let $k_{\la,\mu} : \{\bla \times_{\Gv(\OO)} \bmu\} \hookrightarrow \Gv(\KK) \times_{\Gv(\OO)} \Gr$ be the obvious embedding. The $({}_!,{}^!)$-adjunction for the embedding $\{\bla \times_{\Gv(\OO)} \bmu\} \hookrightarrow \mathsf{mult}^{-1}(\bnu)$ induces a morphism
\[
\coH^\hdot_A \bigl( k_{\la,\mu}^! (\cF \, \widetilde{\boxtimes} \, \cG) \bigr) \ \to \ \coH^\hdot_A \bigl(\mathsf{mult}^{-1}(\bnu), j_{\nu}^! (\cF \, \widetilde{\boxtimes} \, \cG) \bigr),
\]
which can be reinterpreted as a morphism
\beq{eqn:conv-top}
\coH^\hdot_A \bigl( k_{\la,\mu}^! (\cF \, \widetilde{\boxtimes} \, \cG) \bigr) \ \to \ \coH^\hdot_A \bigl( i_{\la+\mu}^! (\cF \star \cG) \bigr).
\eeq

\begin{lem}
\label{lem:cohomology-tensor-product}
There exists a canonical isomorphism of graded $\sh$-modules
\[
\coH^\hdot_A \bigl( k_{\la,\mu}^! (\cF \, \widetilde{\boxtimes} \, \cG) \bigr) \ \cong \ \coH^\hdot_A(i_\la^! \cF) \o_{\sh} {}^{(\la)} \hspace{-1pt} \coH^\hdot_A(i_\mu^! \cG).
\]
\end{lem}

\begin{proof}
Choose some closed finite union of $\Gv(\OO)$-orbits $Y \subset \Gr$ such that $\cG$ is supported on $Y$, and a closed normal subgroup $H \lhd \Gv(\OO)$ of finite codimension $c$, which acts trivially on $Y$. Let $f : \Gv(\KK)/H \times Y \to \Gv(\KK) \times_{\Gv(\OO)} Y$ and $g : \Gv(\KK)/H \to \Gr$ be the natural projections. Then by definition of $\cF \, \widetilde{\boxtimes} \, \cG$ we have a canonical isomorphism $f^* \bigl(\cF \, \widetilde{\boxtimes} \, \cG \bigr) \cong (g^*\cF) \boxtimes \cG$ hence (since both $f$ and $g$ are smooth morphisms of relative dimension $c$) a canonical isomorphism
\beq{eqn:cofiber-convolution}
f^! \bigl(\cF \, \widetilde{\boxtimes} \, \cG \bigr) \ \cong \ (g^! \cF) \boxtimes \cG.
\eeq

We will consider $\Gv(\KK)/H \times Y$ as an $A$-variety where any $t \in \Tv$, resp.~$a \in \C^\times$ acts by 
\[
t \cdot (gH,x) := (tgt^{-1} H,t \cdot x), \quad \text{resp.} \quad a \cdot (gH,x):= \bigl( (a \cdot g) \la(a)^{-1} H, \la(a) \cdot (a \cdot x)).
\]
With this definition, the morphism $f$ is $A$-equivariant, and the point $(\tilde{\bla},\bmu) \in \Gv(\KK)/H \times Y$ is $A$-stable. (Here $\tilde{\bla}:=\hat{\bla} H/H$, where $\hat{\bla}$ is $\lambda$ considered as an element of $\Gv(\KK)$.) Now $k_{\la,\mu}$ factors as the composition
\[
\xymatrix@C=1.5cm{
\{(\bla,\bmu)\} \ar[r]^{\sim} & \{(\tilde{\bla},\bmu) \} \ar@{^{(}->}[r]^-{l_{\la,\mu}} & \Gv(\KK)/H \times Y \ar[r]^-{f} & \Gv(\KK) \times_{\Gv(\OO)} Y,
}
\]
hence using \eqref{eqn:cofiber-convolution} we obtain an isomorphism
\[
\coH^\hdot_A \bigl( k_{\la,\mu}^! (\cF \, \widetilde{\boxtimes} \, \cG) \bigr) \, \cong \, \coH_A^\hdot \bigl( l_{\la,\mu}^! (g^! \cF \boxtimes \cG) \bigr).
\]
Now we have
\[
\coH_A^\hdot \bigl( l_{\la,\mu}^! (g^! \cF \boxtimes \cG) \bigr) \ \cong \ \coH^\hdot_A(i_\la^! \cF) \o_{\sh} {}^{(\la)} \hspace{-1pt} \coH^\hdot_A(i_\mu^! \cG)
\]
(where we use the fact that ${}^{(\la)} \hspace{-1pt} \coH^\hdot_A(i_\mu^! \cG)$ is free over $\sh$ by Lemma \ref{lem:equiv-cohomology-first-properties}, so that we have a K{\"u}nneth formula in equivariant cohomology.)
This finishes the proof.
\end{proof}

Using the isomorphism of Lemma \ref{lem:cohomology-tensor-product} we can reinterpret \eqref{eqn:conv-top} as a morphism\index{Convtop@$\mathsf{Conv}^{\operatorname{top}}_{\cF,\cG,\la,\mu}$}%
\[
\mathsf{Conv}^{\operatorname{top}}_{\cF,\cG,\la,\mu} : \coH^\hdot_A(i_\la^! \cF) \o_{\sh} {}^{(\la)} \hspace{-1pt} \coH^\hdot_A(i_\mu^! \cG) \to \coH^\hdot_A \bigl( i_{\la+\mu}^! (\cF \star \cG) \bigr).
\]

\begin{rem}
In case $\la,\mu$ are antidominant, the morphism $\mathsf{Conv}^{\operatorname{top}}_{\cF,\cG,\la,\mu}$ can also be described in terms of Wakimoto sheaves, as in \cite[diagram on p.~652]{abg}.
\end{rem}

The following result will be proved in \S\ref{ss:proof-convolution}

\begin{prop}
\label{prop:convolution}
Let $\cF,\cG$ be in $\Perv_{\Gv(\OO)}(\Gr)$ and $\la,\mu \in \bX$. Under the isomorphisms of Corollary \ref{key_cor}, the morphisms $\mathsf{Conv}^{\operatorname{geom}}_{\cS(\cF),\cS(\cG),\la,\mu}$, $\mathsf{Conv}^{\operatorname{alg}}_{\cS(\cF),\cS(\cG),\la,\mu}$ and $\mathsf{Conv}^{\operatorname{top}}_{\cF,\cG,\la,\mu}$ match.
\end{prop}

\begin{rem}
\label{rk:convolution}
\begin{enumerate}
\item
Note that, unlike most of the constructions in the paper, the morphisms $\mathsf{Conv}$ are \emph{not} compatible with restriction to a Levi subgroup (see \S\ref{ss:restriction-DX}, \S\ref{ss:Levi-DX} or \S\ref{ss:Satake-restriction-cofiber}) in the obvious strong sense. We will use a weaker compatibility result in the proof of Proposition \ref{prop:convolution}.
\item
Proposition \ref{prop:convolution} has an obvious ``classical analogue'', that we do not state for simplicity, but which can be proved by the same methods.
\item
Using similar constructions one can define, for $V,V'$ in $\Rep(G)$ and $\la,\mu \in \bX$ a ``geometric'' convolution morphism
\[
\bigl( V \o \M(\la) \bigr)^G \o_{\sh} {}^{(\la)} \hspace{-1pt} \bigl( V'_\mu \o \sh \bigr)^G \to\bigl( (V \o V')_{\la + \mu} \o \sh \bigr)^G
\]
(where for simplicity we disregard the grading) and, for $\cF,\cG$ in $\Perv_{\Gv(\OO)}(\Gr)$ and $\la,\mu \in \bX$, a ``topological'' convolution morphism
\[
\coH^\hdot_A(i_\la^! \cF) \o_{\sh} {}^{(\la)} \hspace{-1pt} \coH^\hdot_A(t_\mu^! \cG) \to \coH^\hdot_A \bigl( t_{\la+\mu}^! (\cF \star \cG) \bigr).
\]
(These maps ``extend'' $\mathsf{Conv}^{\operatorname{alg}}_{V,V',\la,\mu}$ and $\mathsf{Conv}^{\operatorname{top}}_{\cF,\cG,\la,\mu}$ in the natural sense, using the inclusions $\kalg_{V',\mu}$, $\kalg_{V \o V',\la+\mu}$, $(\imath_\mu)_!$ and $(\imath_{\la+\mu})_!$.)
Then the same arguments as for Proposition \ref{prop:convolution} show that the following diagram commutes:
\[
\xymatrix{
\bigl( V \o \M(\la) \bigr)^G \o_{\sh} {}^{(\la)} \hspace{-1pt} \bigl( V'_\mu \o \sh \bigr)^G \ar[r] \ar[d]^-{\wr}_-{\mathrm{Th.}~\ref{thm:loop-equivariant-version}~\&~\text{Lem.}~\ref{lem:equiv-cohomology-first-properties}} & \bigl( (V \o V')_{\la + \mu} \o \sh \bigr)^G \ar[d]^-{\wr}_-{\text{Lem.}~\ref{lem:equiv-cohomology-first-properties}} \\
\coH^\hdot_A(i_\la^! \cF) \o_{\sh} {}^{(\la)} \hspace{-1pt} \coH^\hdot_A(t_\mu^! \cG) \ar[r] & \coH^\hdot_A \bigl( t_{\la+\mu}^! (\cF \star \cG) \bigr).
}
\]
In particular, for $V=V'=\C[G]$ (the regular representation of $G$, considered as an ind-object in $\Rep(G)$) and taking the direct sum over $\la$ and $\mu$, one can check that composing the ``algebraic'' map with the morphism induced by the multiplication morphism $\mathsf{m} : \C[G] \o \C[G] \to \C[G]$ one obtains an action of the algebra $\dh(\X)$ on the $\C[\hb]$-algebra $\dh^{\mathrm{rel}}$ of \emph{relative} asymptotic differential operators along the fibers of the projection $G \to G/T$. This action can be realized ``topologically'' using the composition
\[
\coH^\hdot_A(i_\la^! \cR) \o_{\sh} {}^{(\la)} \hspace{-1pt} \coH^\hdot_A(t_\mu^! \cR) \longrightarrow \coH^\hdot_A \bigl( t_{\la+\mu}^! (\cR \star \cR) \bigr) \xrightarrow{\cS(\mathsf{m})} \coH^\hdot_A \bigl( t_{\la+\mu}^! \cR \bigr),
\]
where $\cR:=\cS(\C[G])$.

In this construction one can also replace $T$ by a Levi factor of a parabolic subgroup; details are left to the reader.
\end{enumerate}
\end{rem}

\subsection{Proof of Proposition \ref{prop:convolution}}
\label{ss:proof-convolution}

For $\la,\mu \in \bX$ we set
\[
\fT_{\la,\mu} := \{(n \hat{\bla} \times_{\Gv(\OO)} m\bmu) \mid n,m \in \Nv^-(\KK)\} \subset \Gv(\KK) \times_{\Gv(\OO)} \Gr,
\]
and we denote by $t_{\la,\mu} : \fT_{\la,\mu} \hookrightarrow \Gv(\KK) \times_{\Gv(\OO)} \Gr$ the inclusion. (The notation $\hat{\bla}$ is defined in the proof of Lemma \ref{lem:cohomology-tensor-product}.) Then we have decompositions
\[
\Gv(\KK) \times_{\Gv(\OO)} \Gr = \bigsqcup_{\la,\mu \in \bX} \fT_{\la,\mu}, \qquad \mathsf{mult}^{-1}(\fT_{\nu}) = \bigsqcup_{\la+\mu=\nu} \fT_{\la,\mu}.
\]
In fact, $\fT_{\la,\mu}$ is the inverse image of $\fT_\la \times \fT_{\la+\mu} \subset \Gr \times \Gr$ under the isomorphism
$\Gv(\KK) \times_{\Gv(\OO)} \Gr \simto \Gr \times \Gr$ sending $(g_1 \times_{\Gv(\OO)} g_2 \Gv(\OO))$ to $(g_1 \Gv(\OO),g_1 g_2 \Gv(\OO))$.

\begin{lem}
\label{lem:cohomology-tensor-product-2}
For any $\cF,\cG$ in $\Perv_{\Gv(\OO)}(\Gr)$, there exists a natural isomorphism of graded $\sh$-modules
\[
\coH^{\hdot}_A \bigl( \fT_{\la,\mu}, t_{\la,\mu}^! (\cF \, \widetilde{\boxtimes} \, \cG) \bigr) \ \cong \ \coH^{\hdot}_A(\fT_\la, t_\la^! \cF) \otimes_{\sh} {}^{(\la)} \hspace{-1pt} \coH^{\hdot}_A(\fT_\mu, t_\mu^! \cG).
\]
\end{lem}

\begin{proof}
If we set $\widetilde{\fT}_\la := \tilde{\bla} \cdot \Nv^-(\OO^-)_1 \subset \Gv(\KK)$ (where $\Nv^-(\OO^-)_1 \subset \Nv^-(\OO^-)$ is the kernel of the evaluation at $z=\infty$), then the composition $\widetilde{\fT}_\la \hookrightarrow \Gv(\KK) \twoheadrightarrow \Gr$ induces an isomorphism $\widetilde{\fT}_\la \simto \fT_\la$. Then the same arguments as in the proof of Lemma \ref{lem:cohomology-tensor-product} prove our claim.
\end{proof}

Now we fix $\nu \in \bX$ and $\cF,\cG$ in $\Perv_{\Gv(\OO)}(\Gr)$. For $\la \in \bX$ we set
\[
\fT_{\nu}^{\geq \la} := \bigsqcup_{\genfrac{}{}{0pt}{}{\la'+\mu=\nu}{\la' \geq \la}} \fT_{\la',\mu} \subset \mathsf{mult}^{-1}(\fT_{\nu}),
\]
and denote by $t_\nu^{\geq \lambda} : \fT_{\nu}^{\geq \la} \hookrightarrow \Gv(\KK) \times_{\Gv(\OO)} \Gr$ the inclusion.
Then $\fT_{\nu}^{\geq \la}$
is closed in $\mathsf{mult}^{-1}(\fT_{\nu})$, and $\fT_{\la,\nu-\la}$ is open in $\fT_{\nu}^{\geq \la}$. It follows in particular from Lemma \ref{lem:cohomology-tensor-product-2} and Lemma \ref{lem:equiv-cohomology-first-properties} that for any $\la,\mu$ such that $\la+\mu=\nu$, the cohomology $\coH^{\hdot}_A \bigl( \fT_{\la,\mu}, t_{\la,\mu}^! (\cF \, \widetilde{\boxtimes} \, \cG) \bigr)$ is concentrated in degrees of the same parity as $\nu(2\rhov)$. From this parity vanishing observation, one can deduce that the long exact sequence associated with the decomposition $\fT_{\nu}^{\geq \la}=\fT_{\nu}^{> \la} \sqcup \fT_{\la,\nu-\la}$ (where $\fT_{\nu}^{> \la}$ has the obvious definition) for the object $(t_\nu^{\geq \lambda})^! (\cF \, \widetilde{\boxtimes} \, \cG)$ breaks into a family of short exact sequences.
And then (using the base change theorem) we deduce that the graded $\sh$-module
\[
\coH^{\hdot}_A \bigl( \fT_\nu, t_\nu^!(\cF \star \cG) \bigr)
\]
admits a deacreasing $\bX$-filtration with part bigger than $\la$ isomorphic to 
$
\coH^{\hdot}_A(\fT_{\nu}^{\geq \la}, (t_\nu^{\geq \lambda})^! (\cF \, \widetilde{\boxtimes} \, \cG)),
$
and with associated graded
\[
\bigoplus_{\la + \mu=\nu} \coH^{\hdot}_A \bigl( \fT_{\la,\mu}, t_{\la,\mu}^! (\cF \, \widetilde{\boxtimes} \, \cG) \bigr) \, \cong \,
\bigoplus_{\la + \mu=\nu} \coH^{\hdot}_A(\fT_\la, t_\la^! \cF) \o_{\sh} {}^{(\la)} \coH^{\hdot}_A(\fT_\mu, t_\mu^! \cG).
\]
(Here the isomorphism is provided by Lemma \ref{lem:cohomology-tensor-product-2}.)

\begin{lem}

Under the isomorphisms
\[
\coH^{\hdot}_A \bigl( \fT_\nu, t_\nu^!(\cF \star \cG) \bigr) \ \cong \ \bigl( \cS(\cF) \otimes \cS(\cG) \bigr)_\nu \o \sh \langle \nu(2\rhov) \rangle,
\]
and
\[
\coH^{\hdot}_A(\fT_\la, t_\la^! \cF) \o_{\sh} {}^{(\la)} \coH^{\hdot}_A(\fT_\mu, t_\mu^! \cG) \ \cong \ (\cS(\cF)_\la \o \cS(\cG)_\mu) \o \sh \langle \nu(2\rhov) \rangle
\]
provided by Lemma {\rm \ref{lem:equiv-cohomology-first-properties}}, the ``topological'' filtration on $\coH^{\hdot}_A \bigl( \fT_\nu, t_\nu^!(\cF \star \cG) \bigr)$ considered above
is induced by the filtration on $\bigl( \cS(\cF) \o \cS(\cG) \bigr)_\nu$ by the subspaces
\[
\bigoplus_{\la' \geq \la} \cS(\cF)_{\la'} \o \cS(\cG)_{\nu-\la'}.
\]
\end{lem}

\begin{proof}[Sketch of proof]
By construction of the isomorphisms in Lemma \ref{lem:equiv-cohomology-first-properties} it is sufficient to prove the analogous claim for ordinary cohomology; in other words one can forget about $A$-equivariance. Now recall the  construction of the tensor structure in \cite[Proposition 6.4]{mv}; in particular let $X$ be a smooth (algebraic) curve, and consider the local analogues of $\Gr$ and the convolution diagram over $X^2$ as in \cite[Equation (5.2)]{mv}. In the proof of \cite[Proposition 6.4]{mv}, the authors define a global counterpart $\fT_\nu(X^2) \subset \Gr_{X^2}$ of $\fT_\nu$. Then one can consider
\[
\mathsf{mult}^{-1} \bigl( \fT_\nu(X^2) \bigr) \subset \Gr_X \widetilde{\times} \Gr_X
\]
in the ``global analogue'' of $\Gv(\KK) \times_{\Gv(\OO)} \Gr$. One can define locally closed ind-subvarieties $\fT_{\la,\mu}(X^2)$ inside this inverse image (when $ \la+\mu=\nu$) which, over points in the diagonal copy of $X$ in $X^2$, coincide with our subvarieties $\fT_{\la,\mu}$ and which, over points outside the diagonal, coincide with $\fT_\la \times \fT_\mu$. Then one has a filtration as above, but this time globally over $X^2$. Over points in the diagonal, this filtration coincides with the one considered above by construction. And over points outside of the diagonal the variety $\fT_\nu(X^2)$ is a disjoint union $\bigsqcup_{\la+\mu=\nu} \fT_\la \times \fT_\mu$, and the filtration is obtained from the decomposition of the appropriate cohomology sheaves as a direct sum as e.g.~in \cite[Equation (6.25b)]{mv}. This implies the claim.
\end{proof}

Using these remarks we are now ready to give a proof of Proposition \ref{prop:convolution}.

\begin{proof}[Proof of Proposition {\rm \ref{prop:convolution}}]
First it is easy to check, by explicit computation, that the isomorphism between the left-hand side and the right-hand side of the equation in Corollary \ref{key_cor} is compatible with the morphisms $\mathsf{Conv}^{\operatorname{geom}}$ and $\mathsf{Conv}^{\operatorname{alg}}$.

Set $\nu:=\la+\mu$, $V:=\cS(\cF)$, $V':=\cS(\cG)$.
To finish the proof we have to prove that the left square in the following diagram commutes, where the isomorphisms are as in Theorem \ref{thm:loop-equivariant-version}:
\[
\xymatrix@C=1.7cm{
\mathsf{H}^\hdot_A(i_{\lambda}^! \mathcal{F} ) \o_{\sh} {}^{(\lambda)} \hspace{-1pt} \mathsf{H}^\hdot_A(i_\mu^! \mathcal{G}) \ar[r]^-{\mathsf{Conv}^{\operatorname{top}}_{\cF,\cG,\la,\mu}} \ar[d]^-{\wr} & \mathsf{H}^\hdot_A(i_{\nu}^! (\mathcal{F} \star \mathcal{G} )) \ar@{^{(}->}[r]^-{(\imath_\nu)_!} \ar[d]^-{\wr} & \mathsf{H}^\hdot_A(\mathfrak{T}_{\nu}, t_{\nu}^! (\mathcal{F} \star \mathcal{G} )) \ar[d]^-{\wr} \\
\bigl( V \otimes \mathbf{M}(\lambda) \bigr)^B \o_{\sh} {}^{(\lambda)} \hspace{-1pt} \bigl( V' \otimes \mathbf{M}(\mu) \bigr)^B \ar[r]^-{\mathsf{Conv}^{\operatorname{alg}}_{V,V',\la,\mu}} & \bigl( V \otimes V' \otimes \mathbf{M}(\nu) \bigr)^B \ar@{^{(}->}[r]^-{\kalg_{V \o V', \nu}} & (V \otimes V')_{\nu} \otimes \sh.
}
\]
Of course it is sufficient to prove that the outer square commutes.

Now recall the filtrations on $\mathsf{H}^\hdot_A(\mathfrak{T}_{\nu}, t_{\nu}^! (\mathcal{F} \star \mathcal{G} ))$ and $(V \otimes V')_{\nu} \otimes \mathrm{S}_\hbar$ considered above. Then each of our morphisms factors through the part of the filtration bigger than $\la$: for the first line this follows from the fact that $(\hat{\bla} \times_{\Gv(\OO)} \bmu) \in \fT_\nu^{\geq \la}$, and for the second line this can be checked by explicit computation. Moreover, the restriction to the image of these morphisms of the projection to $\coH^{\hdot}_A(\fT_\la, t_\la^! \cF) \o_{\sh} {}^{(\la)} \coH^{\hdot}_A(\fT_\mu, t_\mu^! \cG)$, resp.~$(V_\la \o V_\mu) \o \sh$, is injective. Hence it is enough to check that the corresponding diagram:
\[
\xymatrix@C=1.7cm{
\mathsf{H}^\hdot_A(i_{\lambda}^! \mathcal{F} ) \otimes_{\mathrm{S}_\hbar} {}^{(\lambda)} \hspace{-1pt} \mathsf{H}^\hdot_A(i_\mu^! \mathcal{G}) \ar[d]^-{\wr} \ar[r] & \coH^{\hdot}_A(\fT_\la, t_\la^! \cF) \o_{\sh} {}^{(\la)} \coH^{\hdot}_A(\fT_\mu, t_\mu^! \cG) \ar[d]^-{\wr} \\
\bigl( V \otimes \mathbf{M}(\lambda) \bigr)^B \otimes_{\mathrm{S}_\hbar} {}^{(\lambda)} \hspace{-1pt} \bigl( V' \otimes \mathbf{M}(\mu) \bigr)^B  \ar[r] & (V_\la \o V_\mu) \o \sh
}
\]
commutes. However the upper line is induced by $(\imath_\la)_! \o (\imath_\mu)_!$, and the bottom line is induced by $\kalg_{V,\la} \o \kalg_{V',\mu}$, hence this claim is clear.
\end{proof}

\subsection{Reminder on dynamical Weyl groups}
\label{ss:definition-dwg}

Let us fix $V$ in $\Rep(G)$, $\la \in \bX$ and $w \in W$. For any $\mu \in \bX^-$ sufficiently large we consider the morphism of $\C[\hb,\hb^{-1}]$-modules
\[
\dwalg_{V,\la,w,\mu} : \C[\hb,\hb^{-1}] \o V_\la \to \C[\hb,\hb^{-1}] \o V_{w\la}, \quad \dwalg_{V,\la,w,\mu} := \mathsf{E}^{V,\la}_{w\mu} \circ \Psi^{V,\la}_{w,\mu} \circ (\mathsf{E}^{V,\la}_\mu)^{-1}.
\]
This morphism is well defined by Lemma \ref{lem:intertwiners-quantum}(1). We will sometimes extend this morphism to a morphism of $\C(\hb)$-modules $\C(\hb) \o V_\la \to \C(\hb) \o V_{w\la}$ in the obvious way (and denote the extension also by $\dwalg_{V,\la,w,\mu}$).

Recall that for $\mu \in \t^*$ we have defined a morphism $P \mapsto P(\mu)$ in \S\ref{ss:specialization}. We denote similarly the induced morphisms $\sh \o V_\la \to \C[\hb,\hb^{-1}] \o V_\la$ or $\sh \o V_{w\la} \to \C[\hb,\hb^{-1}] \o V_{w\la}$.

\begin{lem}
\label{lem:definition-dwg}
There exists a unique isomorphism of $\C(\hb)$-modules
\[
\dwalg_{V,\la,w} : \qh \o V_\la \simto \qh \o V_{w\la}
\]
such that for any $x \in \qh \o V_\la$ the following property holds: for any $\mu \in \bX^-$ sufficiently large such that $x(\mu)$ is defined, we have
\[
\bigl( \dwalg_{V,\la,w} (x) \bigr)(w\mu) = \dwalg_{V,\la,w,\mu} \bigl( x(\mu) \bigr)
\]
(in particular, the left-hand side is defined).
This morphism induces an isomorphism of $\qh$-modules
\[
\dwalg_{V,\la,w} : \qh \o V_\la \simto \w \qh \o V_{w\la}.
\]
\end{lem}

\begin{proof}
This claim is proved in \cite{tv, ev}. For later use, let us explain how it can be deduced from our constructions.
Unicity is clear (see e.g.~\S\ref{ss:specialization}). Let us prove existence.

It is enough to treat the case where $w=s$ is the reflection associated with a simple root $\alpha$ (provided we only require $\mu \in \bX$ to satisfy $\mu(\alv) \leq 0$, not necessarily to be antidominant), see in particular  \eqref{eqn:composition-Psi}. Define the isomorphism of $\C(\hb)$-modules
\[
{}' \dwalg_{V,\la,s} : \qh \o V_\la \simto \qh \o V_{s \la}
\]
to be the composition
\begin{multline*}
\qh \o V_\la \xleftarrow[\sim]{\kalg_{V,\la}} \qh \o_{\sh} \bigl( V \o \M(\la) \bigr)^B \xrightarrow[\sim]{\Omega^{V,\la}_{s}} \qh \o_{\sh} \bigl( V \o \M(s \la) \bigr)^B \\ 
\xrightarrow[\sim]{\kalg_{V,s\la}} \qh \o_{\sh} {}^{s} \bigl( V_{s\la} \o \sh \bigr) \xrightarrow[\sim]{q \o v \o p \mapsto s(q)p \o v} \qh \o V_{s \la}.
\end{multline*}
(The fact that the first and third arrows are invertible was proved in the course of the proof of Theorem \ref{thm:loop-equivariant-version}, see \S\ref{ss:generic-isomorphism}.)
Then for $x \in \qh \o V_\la$ we set
\[
\dwalg_{V,\la,s}(x) = \begin{cases}
\frac{1}{(-\alv ) ( -\alv + \hb ) \cdots (-\alv + (\la(\alv)-1)\hb )} \cdot {}' \dwalg_{V,\la,s} (x) & \text{if } \la(\alv) \geq 0; \\
(-\alv-\hb ) ( -\alv-2\hb ) \cdots (-\alv+\la(\alv)\hb ) \cdot {}' \dwalg_{V,\la,s} (x) & \text{if } \la(\alv) \leq 0
\end{cases}
\]
in $\qh \o V_{s \la}$. It follows from Proposition \ref{prop:definition-Theta} that the morphism $\dwalg_{V,\la,s}$ satisfies our requirements. 
\end{proof}

%

The operators $\dwalg_{V,\la,w}$ form the \emph{dynamical Weyl group} as considered (in the non-asymptotic case) in \cite{tv,ev} and (in the asymptotic case) in \cite{brf}.


\subsection{Transverse slices and semi-infinite orbits}
\label{ss:transverse-slices}

Recall the definition of $\W_\la$, $s_\la$, and $n_\la$ in \S\ref{ss:dwg}. The following result contains in particular Lemma \ref{lem:transverse-slices}. It is essentially proved in \cite{brf}; we explain the details for the reader's convenience.

\begin{lem}
\label{lem:transversal-slice}

Let $\cF$ in $\Perv_{\Gv(\OO)}(\Gr)$ and $\la \in \bX$.

\begin{enumerate}
\item 
There exists a canonical isomorphism of graded ${\sh}$-modules
\[
\coH_{A}^{\hdot}(\W_\la \cap \fT_\la, s_{\la}^! \cF) \ \cong \ \bigl( \cS(\cF) \bigr)_\la \o {\sh} \lan n_\la \ran.
\]
\item 
Under the isomorphism of {\rm (1)} and \eqref{zeta}, the natural morphism
\[
\coH_{A}^{\hdot}(\W_\la \cap \fT_\la, s_{\la}^! \cF) \ \to \ \coH_A^{\hdot}(\fT_\la, t_{\la}^! \cF)
\]
induced by the (closed) inclusion $\fT_{\la} \cap \W_{\la} \hookrightarrow \fT_{\la}$ identifies with multiplication by
\[
\prod_{\genfrac{}{}{0pt}{}{\alpha>0,}{(\la,\alv)<0}} \prod_{j=0}^{- \la(\alv ) -1} (-\alv + j\hbar)
\]
on $\bigl( \cS(\cF) \bigr)_{\la} \o {\sh}$.
\end{enumerate}

\end{lem}

\begin{proof}
There exists a closed subvariety $V \subset \Gv(\OO)$ isomorphic to a (finite dimensional) affine space, which is stable under conjugation by $A$, and which satisfies the following conditions:
\begin{itemize}
\item the morphism $V \to \Gr^\la$, $u \mapsto u \cdot \bla$, is an open embedding;
\item the morphism $V \times \W_\la \to \Gr$ induced by the $\Gv(\OO)$-action on $\Gr$ is an open embedding.
\end{itemize}
Indeed, by $N_\Gv(\Tv)$-equivariance it is enough to treat the case $\la$ is dominant. And in this case one can e.g.~take as $V$ the subvariety $J^\la$ considered in \cite[Lemme 2.2]{np}. We will identify $V \times \W_\la$ with its image in $\Gr$, and denote it by $O$.

We have canonical isomorphisms
\begin{align*}
\coH^{\hdot}_{A}(\fT_\la, t_{\la}^! \cF) \ \cong \ \coH^{\hdot}_{A}(\imath_\la^* t_{\la}^! \cF) \ \cong \ \coH^{\hdot}_{A}((\{\bla\} \hookrightarrow O \cap \fT_{\la})^* (O \cap \fT_{\la} \hookrightarrow \Gr)^! \cF).
\end{align*}
Here the first isomorphism follows from \cite[Proposition 2.3]{fw} (since $\bla$ is an attractive fixed point of $A$ on $\fT_\la$), and the second one from the fact that $O$ is open in $\Gr$.

As $\cF$ is $\Gv(\OO)$-equivariant, there exists a canonical isomorphism
\begin{equation}
\label{eqn:restriction-O}
(O \hookrightarrow \Gr)^* \cF \ \cong \ \underline{\bbc}_V [2\la(2\rhov)] \boxtimes \cF_{(\la)} \quad \text{where} \quad \cF_{(\la)} := (\W_{\la} \hookrightarrow \Gr)^* \cF [-2\la(2\rhov)].
\end{equation}
As $V$ is smooth of dimension $\dim(\Gr^\la)=\la(2\rhov)$, we also have canonically
\[
\cF_{(\la)} \cong (\W_{\la} \hookrightarrow \Gr)^! \cF. 
\]

By definition we have
\[
\fT_{\la} \ = \ \{x \in \Gr \mid \lim_{s \to \infty} 2\rho(s) \cdot x = \bla \}.
\]
(Here $\rho$ is considered as a cocharacter of $\Tv$.)
If follows that we have an isomorphism
\beq{eqn:O-T}
O \cap \fT_{\la} \ \cong \ (V \cap \fT_{\la}) \times (\W_{\la} \cap \fT_{\la}).
\eeq
The variety $V \cap \fT_{\la}$ is open in $\Gr^{\la} \cap \fT_{\la}$, and by \cite[Equation (3.6)]{mv} we have
\beq{eqn:Gr-T}
\Gr^{\la} \cap \fT_{\la} \ = \ \Uv^-(\fO) \cdot \bla \ \cong \ \prod_{\genfrac{}{}{0pt}{}{\alpha>0,}{(\la,\alv)<0}} \prod_{k=0}^{-(\la,\alv)-1} \bbc_A(-\alv+k\hbar)
\eeq
as $A$-varieties. In particular, $V \cap \fT_{\la}$ is smooth.

Using \eqref{eqn:restriction-O} we obtain (under the identification \eqref{eqn:O-T}) a canonical isomorphism
\[
(O \cap \fT_{\la} \hookrightarrow \Gr)^! \cF \ \cong \ \underline{\bbc}_{V \cap \fT_{\la}}[2\dim(V \cap \fT_{\la})] \boxtimes (\W_{\la} \cap \fT_{\la} \hookrightarrow \W_{\la})^! \cF_{(\la)}.
\]
We deduce a canonical isomorphism
\begin{equation*}
\coH^{\hdot}_A(\fT_\la, t_{\la}^! \cF) \ \cong \ \coH^{\hdot}_{A}((\{\bla\} \hookrightarrow \W_{\la} \cap \fT_{\la})^* s_{\la}^! \cF) \langle -2 \dim(V \cap \fT_\la) \rangle.
\end{equation*}
Again by \cite[Proposition 2.3]{fw}, we have a canonical isomorphism
\[
\coH^{\hdot}_{A}(\W_\la \cap \fT_\la, s_{\la}^! \cF) \ \cong \ \coH^{\hdot}_{A}((\{\bla\} \hookrightarrow \W_{\la} \cap \fT_{\la})^* s_{\la}^! \cF).
\]
Using the first isomorphism in \eqref{zeta}, we obtain finally an isomorphism
\[
\coH^{\hdot}_{A}(s_{\la}^! \cF) \cong \bigl( \cS(\cF) \bigr)_\la \o \sh \langle \la(2\rhov)+2\dim(V \cap \fT_\la) \rangle.
\]
Now the dimension $\dim(V \cap \fT_\la)$ can be computed using \eqref{eqn:Gr-T}, and (1) follows.

To prove (2) we have to understand the natural morphism
\[
\coH^A_\hdot(\{\bla\}) \to \coH^A_\hdot(V \cap \fT_\la) \langle -2\dim(V \cap \fT_\la) \rangle.
\]
However this morphism factorizes as the following composition:
\[
\coH^A_\hdot(\{\bla\}) \to \coH^A_\hdot(\Gr^\la \cap \fT_\la) \langle -2\dim(\Gr^\la \cap \fT_\la) \rangle \simto \coH^A_\hdot(V \cap \fT_\la) \langle -2\dim(V \cap \fT_\la) \rangle,
\]
where the first morphism is induced by the inclusion $\{\bla\} \hookrightarrow \Gr^\la \cap \fT_\la$, and the second morphism is given by restriction to the open subvariety $V \cap \fT_\la$. Now the first morphism can be computed using \eqref{eqn:Gr-T} and the reminder in \S\ref{ss:equiv-cohomology}, and the result follows.
\end{proof}

\subsection{Geometric realization of dynamical Weyl groups}
\label{ss:dwg-proof}

Now we are in a position to prove Proposition \ref{prop:dwg}. 

\begin{proof}[Proof of Proposition {\rm \ref{prop:dwg}}]
Fix $\cF$ in $\Perv_{\Gv(\OO)}(\Gr)$, and set $V:=\cS(\cF)$. Then it is enough to prove that for any simple root $\alpha$ and any $\la \in \bX$ such that $\la(\alv) \geq 0$ the following diagram commutes, where the vertical isomorphisms are induced by those of Lemma \ref{lem:transversal-slice}(1) and where $s:=s_\alpha$:
\[
\xymatrix@C=2.5cm{
\qh \o_{\sh} \coH^{\hdot}_A(\W_\la \cap \fT_\la, s_\la^! \cF) \ar[r]^-{\dwgeom_{\cF,\la,s}} \ar[d]_-{\wr} & \qh \o_{\sh} \coH^{\hdot}_A(\W_{s\la} \cap \fT_{s\la}, s_{s \la}^! \cF) \ar[d]^-{\wr} \\
\qh \o V_\la \ar[r]^-{\dwalg_{V,\la,s}} & \qh \o V_{s \la}
}
\]

Consider the isomorphism of $\C(\hb)$-modules
\[
{}' \dwgeom_{\cF,\la,s} : \qh \o_{\sh} \coH^{\hdot}_A(\fT_\la, t_\la^! \cF) \simto \qh \o_{\sh} \coH^{\hdot}_A(\fT_{s\la}, t_{s \la}^! \cF)
\]
defined by the composition
\begin{multline*}
\qh \o \coH^{\hdot}_A(\fT_\la, t_\la^! \cF) \xleftarrow[\sim]{(\imath_\la)_!} \qh \o \coH_A^{\hdot}(i_\la^! \cF) \xrightarrow[\sim]{\Xi^{V,\la}_{s}} \qh \o {}^s \hspace{-1pt} \coH_A^{\hdot}(i_{s \la}^! \cF) \\ 
\xrightarrow[\sim]{(\imath_{s\la})_!} \qh \o_{\sh} {}^s \hspace{-1pt} \coH^{\hdot}_A(\fT_{s\la}, t_{s \la}^! \cF) \xrightarrow[\sim]{s \o 1} \qh \o_{\sh} \coH^{\hdot}_A(\fT_{s\la}, t_{s \la}^! \cF).
\end{multline*}
(Here the first and third arrows are invertible by the localization theorem in equivariant cohomology.)
Using the notation of \S\ref{ss:definition-dwg},
it follows from Theorems \ref{thm:loop-equivariant-version} and \ref{thm:W-symmetry} that the following diagram commutes:
\[
\xymatrix@C=2.5cm{
\qh \o_{\sh} \coH^{\hdot}_A(\fT_\la, t_\la^! \cF) \ar[r]^-{{}' \dwgeom_{\cF,\la,s}} \ar[d]_-{\wr}^-{\eqref{zeta}} & \qh \o_{\sh} \coH^{\hdot}_A(\fT_{s\la}, t_{s \la}^! \cF) \ar[d]_-{\eqref{zeta}}^-{\wr} \\
\qh \o V_\la \ar[r]^-{{}' \dwalg_{V,\la,s}} & \qh \o V_{s \la}
}
\]

As explained in the proof of Lemma \ref{lem:definition-dwg}, for any $x \in \qh \o V_\la$ we have
\[
{}' \dwalg_{V,\la,s}(x) =
\bigl(-\alv \bigr) \bigl( -\alv + \hb \bigr) \cdots \bigl(-\alv + (\la(\alv)-1)\hb \bigr) \cdot \dwalg_{V,\la,s} (x)
\]
in $\qh \o V_{s\la}$.
On the other hand, it follows from Lemma \ref{lem:transversal-slice}(2) that if we identify the $\qh$-modules $\qh \o_{\sh} \coH^{\hdot}_A(\fT_\la, t_\la^! \cF)$ and $\qh \o_{\sh} \coH^{\hdot}_A(\W_\la \cap \fT_\la, s_\la^! \cF)$ with $\qh \o V_\la$, and $\qh \o_{\sh} \coH^{\hdot}_A(\fT_{s\la}, t_{s\la}^! \cF)$ and $\qh \o_{\sh} \coH^{\hdot}_A(\W_{s\la} \cap \fT_{s\la}, s_{s\la}^! \cF)$ with $\qh \o V_{s\la}$ by the isomorphisms of Lemma \ref{lem:equiv-cohomology-first-properties}(2) and Lemma \ref{lem:transversal-slice}(1), we have for any $x \in \qh \o V_\la$
\[
{}' \dwgeom_{V,\la,s}(x) =
\bigl(-\alv \bigr) \bigl( -\alv + \hb \bigr) \cdots \bigl(-\alv + (\la(\alv)-1)\hb \bigr) \cdot \dwgeom_{V,\la,s} (x)
\]
in $\qh \o V_{s\la}$. The proposition follows.
\end{proof}

\subsection{Brylinski--Kostant filtration}
\label{ss:BK-filtration}


We have a natural isomorphism of algebras $\sym(\g/\u) = \C[(\g/\u)^*]$ so that for any $\varphi \in \t^*$ there is a natural surjective algebra morphism
\beq{eqn:restriction-phi}
\sym(\g/\u) \to \C[\varphi + (\g/\b)^*]
\eeq
associated with the inclusion $\varphi + (\g/\b)^* \hookrightarrow (\g/\u)^*$. (Here, as usual we consider $\varphi$ as a linear form on $\g$ trivial on $\u \oplus \u^-$.) 

\begin{lem}
\label{lem:specialization-algebra}

Let $V$ in $\Rep(G)$ and $\la \in \bX$. For any $\varphi \in \t^*_{\rs}$, the morphism
\[
\bbc_{\varphi} \o_{\sym(\t)} \bigl( V \o \sym(\g/\u) \o \bbc_{-\lambda} \bigr)^B \to \bigl( V \o \bbc[\varphi + (\g/\b)^*] \o \bbc_{-\lambda} \bigr)^B
\]
induced by \eqref{eqn:restriction-phi}
is an isomorphism.

\end{lem}

\begin{proof}
We have
\begin{multline*}
\bbc_{\varphi} \o_{\sym(\t)} \bigl( V \o \sym(\g/\u) \o \bbc_{-\lambda} \bigr)^B \ \cong \ \bbc_{\varphi} \o_{\C[\t^*_{\rs}]} \C[\t^*_{\rs}] \o_{\sym(\t)} \bigl( V \o \sym(\g/\u) \o \bbc_{-\lambda} \bigr)^B \\
\cong \ \bbc_{\varphi} \o_{\C[\t^*_{\rs}]} \Bigl( V \o \bigl( \C[\t^*_{\rs}] \o_{\sym(\t)} \sym(\g/\u) \o \bbc_{-\lambda} \bigr) \Bigr)^B \ \overset{\eqref{eqn:classical-case-generic}}{\cong} \bbc_{\varphi} \o_{\C[\t^*_{\rs}]} \bigl( V \o \C[\t^*_{\rs}] \o \Ind_T^B(-\la) \bigr)^B \\ 
\cong \ \bigl( V \o \C_{-\la} \o \C[B/T] \bigr)^B.
\end{multline*}
On the other hand we have $\varphi + (\fg/\fb)^* = B \cdot \varphi \cong B/T$ as a $B$-variety, hence there is a natural isomorphism of $B$-modules $\bbc[\varphi + (\g/\b)^*] \cong \C[B/T]$. This implies our claim.
\end{proof}

Let us now fix $\la \in \bX$ and $\cF$ in $\Perv_{\Gv(\OO)}(\Gr)$. To simply notation we set $V:=\cS(\cF)$.

Recall that for $\varphi \in \t^*$ we have defined the filtered vector space
$\coH_\varphi(i_\la^! \cF)$ in \S\ref{ss:dwg}. On the other hand, the algebra $\bbc[\varphi + (\g/\b)^*]$ is also naturally filtered: the filtration is defined so that for any $\psi \in \varphi + (\g/\b)^*$, the algebra isomorphism $\C[\varphi+(\g/\b)^*] \cong \C[(\g/\b)^*]$ induced by the isomorphism 
\[
(\g/\b)^* \simto \varphi + (\g/\b)^*, \qquad f \mapsto \psi+f
\]
is an isomorphisms of \emph{filtered} algebras, where the filtration on $\C[(\g/\b)^*]$ is the one induced by the grading such that the vectors in $\g/\b \subset \C[(\g/\b)^*]$ are in degree $2$. Hence we have an induced filtration on the vector space
$
\bigl( V \o \bbc[\varphi + (\fg/\fb)^*] \otimes \bbc_{-\lambda} \bigr)^B$.

Combining Theorem \ref{thm:equiv-coh-classical} and Lemma \ref{lem:specialization-algebra}, one obtains the following.

\begin{cor}
\label{cor:specialized-cohomology}

For any $\varphi \in \ft^*_{\mathrm{rs}}$, there exists a canonical isomorphism
\[
\coH_{\varphi}(i_{\lambda}^! \cF) \ \cong \ \bigl( V \o \bbc[\varphi + (\fg/\fb)^*] \otimes \bbc_{-\lambda} \bigr)^B.
\]
This isomorphism is an isomorphism of \emph{filtered} vector spaces, where the filtration on the right-hand side is shifted by $\la(2\rhov)$, i.e.~for any $j \in \Z$ it restricts to an isomorphism
\[
\mathsf{F}_j \Bigl( \coH_{\varphi}(i_{\lambda}^! \cF) \Bigr) \ \cong \ \mathsf{F}_{j-\la(2\rhov)} \Bigl( \bigl( V \o \bbc[\varphi + (\fg/\fb)^*] \otimes \bbc_{-\lambda} \bigr)^B \Bigr).
\]

\end{cor}

Now we are in a position to give a proof of Proposition \ref{prop:BK}. In fact, this proposition is a consequence of Corollary \ref{cor:specialized-cohomology} and the following result (which is essentially proved in \cite{bry}; we reproduce the proof for the reader's convenience).

\begin{prop}
Let $e \in \u$ be a sum of non-zero simple root vectors. If $\varphi \in \t^* \smallsetminus \{0\}$ satisfies $(\mathrm{ad}^* e)(\varphi)=0$, then evaluation at $\varphi$ induces an isomorphism
\beq{eqn:evaluation-phi}
\bigl( V \o \bbc[\varphi + (\fg/\fb)^*] \otimes \bbc_{-\lambda} \bigr)^B \simto V_\la.
\eeq
For any $j \in \Z$, this isomorphism restricts to an isomorphism
\[
\mathsf{F}_{2j} \Bigl( \bigl( V \o \bbc[\varphi + (\fg/\fb)^*] \otimes \bbc_{-\lambda} \bigr)^B \Bigr) = \mathsf{F}_{2j+1} \Bigl( \bigl( V \o \bbc[\varphi + (\fg/\fb)^*] \otimes \bbc_{-\lambda} \bigr)^B \Bigr) \simto \FBK_j (V_\la).
\]
\end{prop}

\begin{proof} 
Let us denote morphism \eqref{eqn:evaluation-phi} by $\Lambda_\la$.
The fact that $\Lambda_\la$ is an isomorphism follows from the proof of Lemma \ref{lem:specialization-algebra}. What remains is to prove that this isomorphism is compatible with filtrations. Note that $\Lambda_\la$ is the restriction to the $\la$-weight spaces of the isomorphism of $T$-modules
\[
\Lambda :\ (V \otimes \bbc[\varphi + (\g/\b)^*])^U \xrightarrow{\sim} V
\]
given again by evaluation at $\varphi$. By construction the filtration on the left-hand side of \eqref{eqn:evaluation-phi} has jumps only in even degrees, which justifies the equality $\mathsf{F}_{2j}=\mathsf{F}_{2j+1}$. For $j \in \Z$ we set
\[
\mathsf{F}^{\operatorname{fib}}_j(V_\la) := \Lambda_\la \Bigl( \mathsf{F}_{2j} \bigl( ( V \o \bbc[\varphi + (\fg/\fb)^*] \otimes \bbc_{-\lambda} )^B \bigr) \Bigr).
\]
Hence what we have to check is that $\mathsf{F}^{\operatorname{fib}}_{\hdot}$ coincides with $\FBK_{\hdot}$.

It will be convenient to work in $\g$ rather than $\g^*$. Hence we choose a $G$-equivariant isomorphism $\g \cong \g^*$; it restricts to a $B$-equivariant
isomorphism $(\g/\u)^* \cong \b$. Let $h \in \b$ be the image of $\varphi$, so that we obtain an identification $\varphi + (\g/\b)^* \cong h + \u$. Consider the isomorphisms
\[
\bigl( V \otimes \bbc[\varphi + (\fg/\fb)^*] \o \bbc_{-\lambda} \bigr)^B \ \cong \ \bigl( V \otimes \bbc[h + \u] \o \bbc_{-\lambda} \bigr)^B \ \cong \ \mathrm{Hom}^B \bigl( V^* \otimes \bbc_{\lambda},\bbc[h + \u] \bigr).
\]
Under this identification, $\Lambda_\la$ sends a morphism $f : V^* \otimes \bbc_{\lambda} \to \bbc[h + \u]$ to the linear form on $(V_{\lambda})^*$ given by $\psi \mapsto f(\psi \otimes 1)(h)$. In fact the image of $\Lambda_{\lambda}(f)$ in $V$ is $\Lambda(f)$, which can be described as the linear form on $V^*$ given by $\psi \mapsto f(\psi \o 1)(h)$.

Fix some $f \in \mathrm{Hom}^B \bigl( V^* \otimes \bbc_{\lambda},\bbc[h + \fn] \bigr)$, and let $v=\Lambda_{\lambda}(f)$ be its image in $V_{\lambda}$ (or in $V$). By definition, $v$ is in $\mathsf{F}^{\mathrm{fib}}_j ( V_{\lambda})$ iff for any $x \in \u$ and any $\psi \in V^*$, the polynomial in $t$ given by 
\[
f(\psi \o 1)(h+tx)
\]
has degree $\leq j$. Fix $\psi \in V^*$, and choose $x \in \u$ such that this degree is maximal. By density of regular nilpotent elements, one can assume that $x$ is regular. Then there exists $b \in B$ such that $x=b \cdot e$. (Indeed, it is well known that there exists $g \in G$ such that $x=g \cdot e$. Then $x \in g \cdot \b$; as a regular nilpotent element is contained in only one Borel subalgebra we deduce that $\b=g \cdot \b$, which implies that $g \in B$.) And we have
\[
f(\psi \o 1)(h+tx) = f(\psi \o 1)(b \cdot (b^{-1} \cdot h + te)) = f(b^{-1} \cdot (\psi \o 1))(b^{-1} \cdot h + te). 
\]
Now, by \cite[Lemma 4.2]{bry}, the degree of the polynomial on the right-hand side is the same as the degree of the polynomial $f(b^{-1} \cdot (\psi \o 1))(h + te)$. Moreover, we have
\[
f(b^{-1} \cdot (\psi \o 1))(h + te) = f(b^{-1} \cdot (\psi \o 1))(\exp(te) \cdot h) = f(\exp(-te) b^{-1} \cdot (\psi \o 1))(h).
\]
Hence, finally, $v$ is in $\mathsf{F}^{\operatorname{fib}}_j (V_{\lambda})$ iff for any $\psi \in V^*$, the polynomial 
\[
f(\exp(-te) \cdot (\psi \o 1))(h)
\]
has degree $\leq j$.

Now, the linear form on $V^*$ given by $\psi \mapsto f(\exp(-te) \cdot (\psi \o 1))(h)$ is $\exp(te) \cdot v$. Hence $v$ is in $\mathsf{F}^{\operatorname{fib}}_j ( V_{\lambda} )$ iff $e^{j+1} \cdot v=0$, i.e.~iff $v \in \FBK_j(V_\la)$.
\end{proof}

\subsection{Ordinary cohomology of cofibers}

As a consequence of Corollary \ref{key_cor-classical} we can prove the following result, which is equivalent to \cite[Theorem 8.5.2]{abg}. (More precisely, \cite[Theorem 8.5.2]{abg} also contains a claim about compatibility with convolution, which can be deduced from Proposition \ref{prop:convolution}.)

\begin{prop}
For $\cF$ in $\Perv_{\Gv(\OO)}(\Gr)$ and $\la \in \bX^+$, there exists a canonical isomorphism of graded vector spaces
\[
\coH^\hdot(i_\la^! \cF) \ \cong \ \Bigl( \cS(\cF) \o \Gamma \bigl( \wcN, \oo_{\wcN}(\la) \bigr) \Bigr)^G \langle \la(2\rhov) \rangle.
\]
\end{prop}

\begin{proof}
By the arguments in the proof of Lemma \ref{lem:equiv-cohomology-first-properties} the forgetful functor induces an isomorphism
\[
\C \o_{\sym(\t)} \coH^\hdot_\Tv(i_\la^! \cF) \simto \coH(i_\la^! \cF)
\]
(where, in the left-hand side, $\sym(\t)$ acts trivially on $\C$).
On the other hand, 
as observed in the proof of Lemma \ref{lem:T*Xfg}, restriction induces an isomorphism
\[
\C \o_{\sym(\t)} \Gamma \bigl( \wfg,\oo_{\wfg}(\la) \bigr) \simto  \Gamma \bigl( \wcN, \oo_{\wcN}(\la) \bigr).
\]
Then the result follows from Corollary \ref{key_cor-classical}, using the fact that the functor of $G$-fixed points is exact.
\end{proof}

\begin{rem}
One can obtain in a similar way a description of $\coH^\hdot_{\C^\times}(i_\la^! \cF)$ (in case $\la$ is dominant) in terms of asymptotic $\dd$-modules on $\B$. We omit the details.
\end{rem}

\subsection{Equivariant cohomology of spherical perverse sheaves}

In this subsection we explain the relation between our description of equivariant cohomology of cofibers of spherical perverse sheaves on $\Gr_{\Gv}$ and the description of the full equivariant cohomology of these perverse sheaves given in \cite{bf}. Details (and generalization to all reductive groups) will be discussed in a future publication.

From now on, for simplicity we assume that $\Gv$ is quasi-simple and simply connected, so that $G$ is simple (of adjoint type). The Killing form determines an isomorphism of $G$-modules $\kappa : \g \simto \g^*$. Let us choose an element $e \in \u$ which is a sum of non-zero simple root vectors. Then $e$ is regular nilpotent, and it can be completed to an $\mathfrak{sl}_2$-triple $(e,\rhov,f)$. (Note that $f$ is uniquely determined by $e$, and that the different choices for $e$ are all conjugate under the adjoint action of $T$.) We consider the Kostant slice\index{Sigmae@$\Sigma_e$}%
\[
\Sigma_e:= \kappa(e + \g_f) \subset \g^*.
\]
It is well known that $\Sigma_e$ is included in $\g^*_{\mathrm{r}}$, and that the (co-)adjoint quotient $\g^* \to \g^*/G \cong \t^*/W$ restricts to an isomorphism $\Sigma_e \simto \t^*/W$. We denote by $\widetilde{\Sigma}_e$\index{sigmaetilde@$\widetilde{\Sigma}_e$} the inverse image of $\Sigma_e$ under the projection $\pi : \wfg \to \g^*$. Then $\widetilde{\Sigma}_e \subset \wfgr$, and the morphism $\widetilde{\Sigma}_e \to \Sigma_e$ is $W$-equivariant, where $W$ acts on $\widetilde{\Sigma}_e$ by the restriction of the action on $\wfgr$, and trivially on $\Sigma_e$. It is also known that the natural morphism $\wfg \to \g^* \times_{\t^*/W} \t^*$ restricts to an isomorphism of algebraic varieties $\wfgr \to \g^*_{\mathrm{r}} \times_{\t^*/W} \t^*$ (see \cite[Remark 4.2.4(i)]{gi-kostant}). We deduce that the morphism $\delta : \wfg \to \t^*$ restricts to an isomorphism $\widetilde{\Sigma}_e \simto \t^*$.

\begin{lem}
\label{lem:section}
The projection $T^* \X \to \wfg$ admits a canonical section $\omega_e : \widetilde{\Sigma}_e \hookrightarrow T^* \X$ over $\widetilde{\Sigma}_e$. This section is independent of the choice of $e$ in the following sense: if $t \in T$ then the following diagram commutes:
\[
\xymatrix@C=2cm{
\widetilde{\Sigma}_e \ar[r]^{t \cdot (-)}_-{\sim} \ar@{^{(}->}[d]^-{\omega_e} & \widetilde{\Sigma}_{t \cdot e} \ar@{^{(}->}[d]^-{\omega_{t \cdot e}} \\
T^*\X \ar[r]^{t \cdot (-)}_-{\sim} & T^*\X.
}
\]
The morphism $\omega_e$ is also $W$-equivariant, if $W$ acts on $(T^*\X)_{\mathrm{r}}$ via the action of {\rm \S\ref{ss:action-T*X}}, for the choice of simple root vectors given by the components of $e$ on each $\g_\alpha$.
\end{lem}

\begin{proof}
Given a Borel subalgebra $\b_0 \subset \g$ with unipotent radical $\u_0$, the ``universal Cartan subalgebra'' $\b_0/\u_0$ acts on $(\u_0 / [\u_0,\u_0])^*$; under this action $(\u_0 / [\u_0,\u_0])^*$ decomposes as a direct sum of $1$-dimensional eigenspaces, and the eigenvalues can be naturally identified with the negative simple roots under the canonical isomorphism $\t^* \cong (\b_0/\u_0)^*$. Following \cite{gk} we denote by $\mathbb{O}(\b_0) \subset (\u_0 / [\u_0,\u_0])^*$ the subset of vectors whose component in each eigenspace is non-zero. We also set $\widetilde{\B}:=\{(\b_0,x) \mid \b_0 \in \B, \, x \in \mathbb{O}(\b_0)\}$. Then our choice of simple root vectors defines a $G$-equivariant isomorphism $\X \simto \widetilde{\B}$ sending $U/U$ to $(\b, \psi)$, where $\psi:=\kappa(f)_{|\u} \in \mathbb{O}(\b)$.

To define the section we need, given some $\eta \in \Sigma_e$ and $\b_0 \in \B$ with unipotent radical $\u_0$ such that $\eta_{|\u_0}=0$, to define a lift of $\b_0$ to $\X$, or equivalently to $\widetilde{\B}$. However, by \cite{gk} the Borel subalgebras $\b_0$ and $\b$ are in general relative position. Hence the Killing form induces a non-degenerate pairing between $\u$ and $\u_0$, hence an isomorphism $\u \simto \u_0^*$. If we denote by $\eta_0 \in \u_0^*$ the image of $e$ under this isomorphism, then one can check that $\eta_0 \in \mathbb{O}(\b_0)$, hence the pair $(\b_0,\eta_0)$ provides the desired lifting of $\b_0$ to $\widetilde{\B}$.
\end{proof}

From now on we fix a choice for $e$.
Let use denote by\index{i@$i$}\index{t@$t$}%
\[
i : \Gr_{\Tv} \hookrightarrow \Gr_{\Gv} \qquad \text{and} \qquad t: \Gr_{\Bv^-} \hookrightarrow \Gr_{\Gv}
\]
the inclusions. We have $\Gr_{\Tv} = \bigsqcup_{\la \in \bX} \{ \bla \}$, and $i$ identifies with $\bigsqcup_{\la \in \bX} i_\la$. Similarly one can identify $\Gr_{\Bv^-}$ with $\bigsqcup_{\la \in \bX} \fT_\la$, and then $t$ with $\bigsqcup_{\la \in \bX} t_\la$. Let $\cF$ be in $\Perv_{\Gv(\OO)}(\Gr_{\Gv})$ and consider the diagram
\[
\xymatrix@C=2cm{
\coH_{\Tv}^{\hdot}(\Gr_{\Bv^-}, t^! \cF) & \coH_{\Tv}^{\hdot}(\Gr_{\Tv}, i^! \cF) \ar@{^{(}->}[r] \ar@{_{(}->}[l] & \coH_{\Tv}^{\hdot}(\Gr_{\Gv}, \cF).
}
\]
Here the left-hand morphism is induced by the inclusion $\Gr_\Tv \hookrightarrow \Gr_{\Bv^-}$, and the right-hand morphism is induced by $i$ as in \S\ref{ss:Gr-intro}. By the localization theorem in equivariant cohomology, both morphisms become isomorphisms when we extend scalars from $\mathrm{S}(\t)$ to its fraction field $\mathrm{Q}$,\index{Q@$\mathrm{Q}$} which provides a canonical isomorphism
\beq{eqn:total-equiv-coh-0}
\mathrm{Q} \o_{\mathrm{S}(\t)} \coH_{\Tv}^{\hdot}(\Gr_{\Gv}, \cF) \ \cong \ \mathrm{Q} \o_{\mathrm{S}(\t)} \coH_{\Tv}^{\hdot}(\Gr_{\Tv}, t^! \cF).
\eeq
By Lemma \ref{lem:equiv-cohomology-first-properties} there exists a canonical isomorphism of $\mathrm{S}(\t)$-modules $
\coH_{\Tv}^{\hdot}(t^! \cF) \cong \cS(\cF) \o \mathrm{S}(\t)$;
hence we obtain a canonical isomorphism
\beq{eqn:total-equiv-coh-1}
\mathrm{Q} \o \coH_{\Tv}^{\hdot}(\Gr_{\Gv}, \cF) \simto \mathrm{Q} \o \cS(\cF).
\eeq

On the other hand, let $V$ in $\Rep(G)$, and consider the diagram
\[
\xymatrix@C=2cm{
V \o \mathrm{S}(\t) & \bigl( V \o \C[T^* \X] \bigr)^G \ar@{^{(}->}[r] \ar@{_{(}->}[l]  & V \o \C[\widetilde{\Sigma}_e].
}
\]
Here the left-hand morphism is induced by inverse image with respect to the natural morphism $G \times \t^* \to T^* \X$, and the right-hand morphism is induced by inverse image with respect to the morphism $\omega_e : \widetilde{\Sigma}_e \hookrightarrow T^*\X$ of Lemma \ref{lem:section}. One can check that both morphisms in this diagram become isomorphisms when we extend scalars from $\mathrm{S}(\t)$ to $\mathrm{Q}$, so that we obtain a canonical isomorphism
\beq{eqn:total-equiv-coh-2}
\mathrm{Q} \o_{\sym(\t)} \bigl( V \o \C[\widetilde{\Sigma}_e] \bigr) \simto \mathrm{Q} \o V.
\eeq

As far as we understand, the isomorphism
like our \eqref{eqn:total-equiv-coh-0} which is implicitly used in \cite{bf} (see in particular the proof of Theorem 6 in \emph{loc}.~\emph{cit}.) is the
one we are defining here. With this interpretation,  the ``quasi-classical limit'' (or ``classical analogue'') of \cite[Theorem 6]{bf} says the following.

\begin{prop}
\label{prop:total-equiv-cohomology}
Let $\cF$ be in $\Perv_{\Gv(\OO)}(\Gr_{\Gv})$.

The image of $\coH_{\Tv}^{\hdot}(\Gr_{\Gv}, \cF)$ in $\mathrm{Q} \o \cS(\cF)$ under isomorphism \eqref{eqn:total-equiv-coh-1} coincides with the image of $\cS(\cF) \o \C[\widetilde{\Sigma}_e]$ under isomorphism \eqref{eqn:total-equiv-coh-2}, which provides a canonical isomorphism 
\beq{eqn:isom-total-cohomology}
\coH_{\Tv}^{\hdot}(\Gr_{\Gv}, \cF) \ \cong \ \cS(\cF) \o \C[\widetilde{\Sigma}_e].
\eeq
\end{prop}

Combined with Lemma \ref{lem:equiv-cohomology-first-properties} and Corollary \ref{key_cor-classical}, this implies that we have the following commutative diagram:
\[
\xymatrix@C=2cm{
\coH_{\Tv}^{\hdot}(\Gr_{\Bv^-}, t^! \cF) \ar[d]^-{\wr}_-{{\rm Lem.~\ref{lem:equiv-cohomology-first-properties}}} & \coH_{\Tv}^{\hdot}(\Gr_{\Tv}, i^! \cF) \ar@{^{(}->}[r] \ar@{_{(}->}[l] \ar[d]^-{\wr}_-{{\rm Cor.~\ref{key_cor-classical}}} & \coH_{\Tv}^{\hdot}(\Gr_{\Gv}, \cF) \ar[d]^-{\wr}_-{\eqref{eqn:isom-total-cohomology}} \\
\cS(\cF) \o \mathrm{S}(\t) & \bigl( \cS(\cF) \o \C[T^* \X] \bigr)^G \ar@{^{(}->}[r] \ar@{_{(}->}[l]  & \cS(\cF) \o \C[\widetilde{\Sigma}_e].
}
\]

Isomorphism \eqref{eqn:isom-total-cohomology} is $W$-equivariant, where the $W$-action on the left-hand side is defined in a way similar to the definition of isomorphisms $\Xi^{\cF,\la}_w$ in \S\ref{ss:W-symmetries}, and the action on the right-hand side is induced by the $W$-action on $\widetilde{\Sigma}_e$. Hence taking fixed points we obtain the following result, also proved in \cite{bf} (see in particular [\emph{loc}.~\emph{cit}., Lemma 9]).

\begin{cor}
There exists a canonical isomorphism $\coH_{\Gv(\OO)}^{\hdot}(\Gr_{\Gv}, \cF) \cong \cS(\cF) \o \C[\Sigma_e]$.
\end{cor}

\appendix

\section{Computations in rank one}
\label{sec:rank1}

In \S\S\ref{ss:Verma-rk1}--\ref{ss:sigma-rk1} we use the notation of Sections \ref{sec:DX}--\ref{sec:wfg}, assuming in addition that $G$ has semisimple rank one. We let $\alpha$ be the unique simple root, and set $s:=s_\alpha$, $e:=e_\alpha$, $f:=f_\alpha$.


\subsection{Asymptotic universal Verma modules}
\label{ss:Verma-rk1}

For any $\nu \in \bX^+$, we denote by $V^{\nu}$\index{Vnu@$V^{\nu}$} the corresponding simple $G$-module. We choose a basis $(v^\nu_\nu, v^\nu_{\nu-\alpha}, \cdots, v^\nu_{\nu-\nu(\alv) \alpha})$\index{vnulambda@$v^{\nu}_{\lambda}$} of $V^\nu$ such that the following formulas are satisfied for $k=0, \cdots, \nu(\alv)$:
\beq{eqn:modules-SL2}
\alv \cdot v^\nu_{\nu-k\alpha} = \bigl( \nu(\alv) - 2k \bigr) v^\nu_{\nu-k\alpha}, \qquad e \cdot v^\nu_{\nu-k\alpha} = k v^\nu_{\nu-(k-1)\alpha}.
\eeq
Such a basis exists and is unique up to a constant. We also set $v_\la^\nu = 0$ if $\la$ is not a weight of $V^\nu$. The following commutation relation in the asymptotic enveloping algebra $\uh(\g)$ is easily checked by induction:
\begin{equation}
\label{eqn:commutation-sl2}
f^k e \ = \ e f^k - k \hbar \cdot \alv f^{k-1} - k(k-1)\hbar^2 \cdot f^{k-1}.
\end{equation}

\begin{lem}
\label{lem:algebra-sl2}

Let $\nu \in \bX^+$.
\begin{enumerate}
\item 
If $\lambda$ is not a weight of $V^\nu$, then $
\bigl( V^{\nu} \o \M(\la) \bigr)^B=0
$.
\item
If $k \in \{0, \cdots, \nu(\alv)\}$, the ${\sh}$-module
$
\bigl( V^{\nu} \o \M(\nu-k\alpha) \bigr)^B
$
is free of rank one, and generated by
\begin{multline*}
x^\nu_{\nu-k\alpha} := v_{\nu}^{\nu} \o 1 \otimes f^k -
\binom{k}{1} v^{\nu}_{\nu-\alpha} \o \Bigl( \alv+\bigl(\nu(\alv)-k \bigr)\hbar \Bigr)
\otimes f^{k-1} \\  
+ \binom{k}{2} \cdot v^{\nu}_{\nu-2\alpha} \o \Bigl( \alv + \bigl(\nu(\alv)-k \bigr)\hbar \Bigr) \Bigl( \alv+ \bigl(\nu(\alv)-k-1 \bigr)\hbar \Bigr)
\otimes f^{k-2} + \cdots \\  
+ (-1)^k \binom{k}{k} \cdot v^{\nu}_{\nu-k\alpha} \o 
\Bigl( \alv+\bigl(\nu(\alv)-k \bigr)\hbar \Bigr) \cdots \Bigl( \alv + (\nu(\alv)-2k+1)\hbar \Bigr) \otimes 1.
\end{multline*}\index{xnulambda@$x^{\nu}_{\lambda}$}%
\end{enumerate}
\end{lem}

\begin{proof}
(1) follows from the injectivity of $\kalg_{V^\nu,\la}$, see Lemma \ref{lem:restH-injective}.

(2)
We have to decide when an element of the form
\[
v_{\nu-k\alpha}^{\nu} \o P_k(h,\hbar) \otimes 1 + \cdots + v_{\nu}^{\nu} \o P_0(h,\hbar) \otimes f^{k}
\]
is annihilated by $e \in \b \subset U(\b)$. (Here the action of $U(\b)$ is the differential of the $B$-action.) However, the image by $e$ of such an element is given by
\[
e \cdot v_{\nu-k\alpha}^{\nu} \o P_k \otimes 1 + \sum_{i=1}^{k-1} \bigl( e \cdot v^{\nu}_{\nu-i\alpha} \o P_{i} \otimes f^{k-i} - \frac{1}{\hbar} v^{\nu}_{\nu-i\alpha} \o P_{i} \otimes f^{k-i} e \bigr)   - \frac{1}{\hbar} v_{\nu}^{\nu} \o P_0 \otimes f^{k} e,
\]
i.e.~(using \eqref{eqn:modules-SL2} and \eqref{eqn:commutation-sl2}) by
\begin{multline*}
k v_{\nu-(k-1)\alpha}^{\nu} \o P_k \otimes 1 + \sum_{i=1}^{k-1} \Bigl( \bigl( i v^{\nu}_{\nu-(i-1)\alpha} \o P_{i} \otimes f^{k-i} 
+ (k-i) v^{\nu}_{\nu-i\alpha} \o P_{i} \bigl( \alv + (\nu(\alv)-k-i)\hb \bigr) \otimes f^{k-i-1} \Bigr) 
\\
+ k v_{\nu}^{\nu} \o P_0 \bigl( \alv+(\nu(\alv)-k)\hb \bigr) \otimes f^{k-1}.
\end{multline*}
Hence the fact that this element is zero amounts to the following equations:
\[
\begin{array}{rcl}
k \cdot P_k & = & - \bigl( \alv + (\nu(\alv)-2k+1)\hbar \bigr) \cdot P_{k-1}, \\
& \vdots & \\
i \cdot P_i & = & -(k-i+1) \bigl( \alv + (\nu(\alv)-k-i+1)\hbar \bigr) \cdot P_{i-1}, \\
& \vdots & \\
P_1 & = & - k \bigl( \alv+(\nu(\alv)-k) \hbar \bigr) \cdot P_0.
\end{array}
\]
The result follows.
\end{proof}

If $\lambda \in \bX$ is not a weight of $V^\nu$, we set $x_\lambda^\nu=0$. As an immediate consequence of Lemma \ref{lem:algebra-sl2} we obtain the following result.

\begin{cor}
\label{cor:image-algebra-sl2}

For $\nu \in \bX^+$ and $k \in \{0, \cdots, \nu(\alv)\}$, the image of the morphism
\[
\kalg_{V^\nu,\nu-k\alpha} : \bigl( V^{\nu} \o \M(\nu-k\alpha) \bigr)^B \to V^{\nu}_{\nu-k\alpha} \o  {\sh} \cong {\sh}
\]
is generated by $\Bigl( \alv + \bigl( \nu(\alv)-k \bigr) \hbar \Bigr) \cdots \Bigl(\alv + \bigl( \nu(\alv) -2k+1 \bigr) \hbar \Bigr)$.

\end{cor}

\subsection{Operators $\Phi$}

In this subsection we consider the constructions of Section \ref{sec:DX}, with the choice of root vector $f$. In particular, for $\nu \in \bX^+$ and $\la \in \bX$, we have an isomorphism
\[
\Phi_s^{V^\nu,\la} : \bigl(V^\nu \o \tla \dh(\X)_\la \bigr)^G \simto {}^s \hspace{-1pt} \bigl(V^\nu \o {}^{(s\la)} \hspace{-1pt} \dh(\X)_{s \la} \bigr)^G.
\]
Recall 
that, by the first isomorphism in Lemma \ref{lem:fixed-points-morphisms},
$x_\lambda^\nu$ defines an element\index{ynulambda@$y^{\nu}_{\lambda}$}%
\[
y_\lambda^\nu \in \bigl(V^\nu \o \tla \dh(\X)_\la \bigr)^G.
\]

\begin{lem}
\label{lem:Phi-rk1}
For $\nu \in \bX^+$ and $\la \in \bX$ we have
\[
\Phi_s^{V^\nu,\la}(y_\lambda^\nu) = y_{s\lambda}^\nu.
\]
\end{lem}

\begin{proof}
If $\lambda$ is not a weight of $V^\nu$, then $s\lambda$ is not either. Hence $y^\nu_\la = 0 = y^\nu_{s\la}$, and the result is clear.

Now assume $\lambda=\nu-k\alpha$ for some $k \in \{0, \cdots, \nu(\alv)\}$. As $\Phi_s^{V^\nu,\la}$ is an isomorphism of $\sh$-modules, we know that $\Phi_s^{V^\nu,\la}(y_\lambda^\nu) = c \cdot y_{s\lambda}^\nu$ for some $c \in \C^\times$, and we have to prove that $c=1$. Let $z_\la^\nu$ be the image of $y_\la^\nu$ under the composition
\[
\bigl(V^\nu \o \tla \dh(\X)_\la \bigr)^G \hookrightarrow V^\nu \o \dh(\X) \to V^\nu \o \mathscr{A}(\X),
\]
and similarly for $z_{s \la}^\nu$.
Then by definition (see \S\ref{ss:Fourier-classical}) we have $(\id_{V^\nu} \o \mathscr{F}_\alpha)(z_\la^\nu) = c \cdot z_{s \la}^\nu$.

Since, in the notation of \S\ref{ss:partial-Fourier}, one can take as $G^\sc$ a product of $\mathrm{SL}(2,\C)$ and a torus, it suffices to prove the claim in the case $G=\mathrm{SL}(2,\C)$, with $B$ the subgroup of upper-triangular matrices, $T$ the subgroup of diagonal matrices, and $f=\left( \begin{smallmatrix} 0 & 0 \\ 1& 0 \end{smallmatrix} \right)$. Let $v_1,v_2$ be the natural basis of $\C^2$ and $\eta_1,\eta_2$ the dual basis. Then the morphism $gU \mapsto g \cdot v_1$ induces an isomorphism $G/U \simto \C^2 \smallsetminus \{0\}$, and $\mathcal{V}_\alpha = \C^2$, with the symplectic form $\omega$ defined by $\omega(v_1,v_2)=1$. Hence we have an canonical isomorphism $\mathscr{A}(\X) \cong \C[T^* (\C^2)]=\C[\C^2 \times (\C^2)^*]$, and the automorphism $\mathscr{F}_\alpha$ can be computed using Example \ref{ex:Fourier-SL2} below: it is defined by
\[
\eta_1 \mapsto v_2, \quad \eta_2 \mapsto -v_1, \quad v_1 \mapsto -\eta_2, \quad v_2 \mapsto \eta_1.
\]

Consider the morphism $\mathrm{ev} : \mathscr{A}(\X) \to \C$ given by evaluation of functions at $(v_1,\eta_2) \in T^* (\C^2)$. The formulas above imply that $\mathrm{ev} \circ \mathscr{F}_\alpha = \mathrm{ev}$. Hence to conclude we only have to check that 
\[
(\id_{V^\nu} \o \mathrm{ev})(z_\lambda^\nu) = (\id_{V^\nu} \o \mathrm{ev})(z_{s \lambda}^\nu).
\]
However we have
\[
(\id_{V^\nu} \o \mathrm{ev})(z_\lambda^\nu) = v_\nu^\nu \o 1 = (\id_{V^\nu} \o \mathrm{ev})(z_{s \lambda}^\nu),
\]
which finishes the proof.
\end{proof}

\subsection{Operators $\Theta$}

In this subsection we consider the constructions of Section \ref{sec:morphisms}, with the choice $f_\alpha:=f$. In particular, for $\nu \in \bX^+$ and $\la \in \bX$, we have a morphism
\[
\Theta^{V^{\nu},\la}_s : \Hom_{(\sh,\uh(\g))} \bigl( \M(0), V^\nu \o \M(\la) \bigr) \to {}^s \hspace{-1pt} \Hom_{(\sh,\uh(\g))} \bigl( \M(0), V^\nu \o \M(s \la) \bigr).
\]
By Lemma \ref{lem:morphisms-Verma-v}, $x_\la^\nu$ defines a morphism of $(\sh,\uh(\g))$-bimodules\index{phinulambda@$\varphi^{\nu}_{\lambda}$}%
\[
\varphi_\la^\nu : \M(0) \to V^\nu \o \M(\la),
\]
which is a generator of the $\sh$-module $\Hom_{(\sh,\uh(\g))} \bigl( \M(0), V^\nu \o \M(\la) \bigr)$.

\begin{lem}
\label{lem:Theta-rk1}
For $\nu \in \bX^+$ and $\la \in \bX$ we have
\[
\Theta_s^{V^\nu,\la}(\varphi^\nu_\la) = \varphi^\nu_{s\la}.
\]
\end{lem}

\begin{proof}
If $\lambda$ is not a weight of $V^\nu$, then $s\lambda$ is not either. Hence $\varphi^\nu_\la = 0 = \varphi^\nu_{s\la}$, and the result is clear.

Now assume $\lambda=\nu-k\alpha$ for some $k \in \{0, \cdots, \nu(\alv)\}$. Then $s\lambda=\nu-(\nu(\alv)-k)\alpha$. By Lemma \ref{lem:algebra-sl2} we know that $\Theta_s^{V^\nu,\nu-k\alpha}(\varphi^\nu_{\nu-k\alpha})$ is a multiple of $\varphi^\nu_{\nu-(\nu(\alv)-k)\alpha}$ (by an element of $\sh$). 
Hence to prove the lemma it is enough to check that the coefficient of $v_\nu^\nu \o 1 \o f^{\nu(\alv)-k}$ in the element of $\bigl( V^\nu \o \M(\nu-(\nu(\alv)-k)\alpha) \bigr)^B$ corresponding to $\Theta_s^{V^\nu,\nu-k\alpha}(\varphi^\nu_{\nu-k\alpha})$ is the same as the coefficient of $v_\nu^\nu \o 1 \o f^{\nu(\alv)-k}$ in $x^\nu_{\nu-(\nu(\alv)-k)\alpha}$, i.e.~$1$. However this coefficient can be computed using formula \eqref{eqn:formula-Theta-s}, and the result is $1$ as expected. This concludes the proof.
\end{proof}

\begin{cor}
\label{cor:Theta-rk1}
For any $V$ in $\Rep(G)$ and $\la \in \bX$ we have
\[
{}^s \hspace{-1pt} \bigl( \Theta_s^{V,s\la} \bigr) \circ \Theta_s^{V,\la} = \id
\]
as endomorphisms of $\Hom_{(\sh,\uh(\g))} \bigl( \M(0), V \o \M(\la) \bigr)$. In particular, each $\Theta_s^{V,\la}$ is an isomorphism.
\end{cor}

\begin{proof}
By complete reducibility it is enough to prove the result when $V=V^\nu$ for some $\nu \in \bX^+$. In this case, by Lemma \ref{lem:Theta-rk1} we have
\[
{}^s \hspace{-1pt} \bigl( \Theta_s^{V^\nu,s\la} \bigr) \circ \Theta_s^{V^\nu,\la} (\varphi_{\la}^{\nu}) = \varphi_{\la}^{\nu},
\]
which proves the claim since $\varphi_{\la}^{\nu}$ is a generator of $\Hom_{(\sh,\uh(\g))} \bigl( \M(0), V^\nu \o \M(\la) \bigr)$ over $\sh$ by Lemma \ref{lem:algebra-sl2}.
\end{proof}

\subsection{Comparison of $\Theta$ and $\Phi$}

\begin{lem}
\label{lem:comparison-Phi-Theta}
For any $V$ in $\Rep(G)$ and $\la \in \bX$,
the following diagram commutes:
\[
\xymatrix@C=2cm{
\bigl( V \o \tla \dh(\X)_\la \bigr)^G \ar[r]^-{\eqref{eqn:isom-D-Hom}}_-{\sim} \ar[d]_-{\Phi^{V,\la}_s} & \Hom_{(\sh,\uh(\g))} \bigl( \M(0), V \o \M(\la) \bigr) \ar[d]^-{\Theta^{V,\la}_s} \\
{}^s \hspace{-1pt} \bigl( V \o  {}^{(s\la)} \hspace{-1pt} \dh(\X)_{s \la} \bigr)^G \ar[r]^-{\eqref{eqn:isom-D-Hom}}_-{\sim} & {}^s \hspace{-1pt} \Hom_{(\sh,\uh(\g))} \bigl( \M(0), V \o \M(s \la) \bigr).
}
\]
\end{lem}

\begin{proof}
By complete reducibility it is enough to prove the lemma when $V=V^\nu$ for some $\nu \in \bX^+$. In this case it follows from Lemma \ref{lem:Phi-rk1} and Lemma \ref{lem:Theta-rk1}.
\end{proof}

\subsection{Operators $\sigma$}
\label{ss:sigma-rk1}

If $\nu \in \bX^+$ and $\la \in \bX$,
we denote by\index{znulambda@$z^{\nu}_{\lambda}$}%
\[
z_\la^\nu \in \bigl( V^\nu \o \Gamma(\wfg, \cO_{\wfg}(\la)) \bigr)^G
\]
the image of $y_\la^\nu$ under the natural morphism
\[
\bigl( V^\nu \o \tla \dh(\X)_\la \bigr)^G \to \bigl( V^\nu \o \Gamma(\wfg, \cO_{\wfg}(\la)) \bigr)^G
\]
sending $\hb$ to $0$. These elements can be naturally identified with the ones denoted similarly in the proof of Lemma \ref{lem:Phi-rk1}. By Lemma \ref{lem:hbar=0} and Lemma \ref{lem:algebra-sl2}, $z_\la^\nu$ is a generator of the $\mathrm{S}(\t)$-module $\bigl( V^\nu \o \Gamma(\wfg, \cO_{\wfg}(\la)) \bigr)^G$, and $z_\la^\nu \neq 0$ iff $\la$ is a weight of $V^\nu$.

\begin{lem}
\label{lem:sigma-rk1}
For $\nu \in \bX^+$ and $\la \in \bX$ we have
\[
\sigma^{V^\nu,\la}_s(z_\la^\nu) = z_{s\la}^\nu.
\]
\end{lem}

\begin{proof}
The statement is clear if $\la$ is not a weight of $V^\nu$. Now assume that $\la=\nu-k\alpha$ for some $k \in \{0, \cdots, \nu(\alv)\}$. As $\sigma^{V^\nu,\la}_s$ is an isomorphism of $\mathrm{S}(\t)$-modules we must have $\sigma^{V^\nu,\la}_s(z_\la^\nu) = c \cdot z_{s\la}^\nu$ for some $c \in \C^\times$. Now recall the commutative diagram \eqref{eqn:diagram-sigma}. We have 
\[
(\id_{V^\nu} \o \mathrm{ev}^\la_{\eta_0})(z_\la^\nu)= v_\nu^\nu =(\id_{V^\nu} \o \mathrm{ev}^{s\la}_{\eta_0})(z_{s\la}^\nu).
\]
This proves that $c=1$, and finishes the proof.
\end{proof}

\begin{cor}
\label{cor:Phi-sigma-rk1}
Let $V$ in $\Rep(G)$ and $\la \in \bX$. The following diagram commutes, where vertical maps are the natural morphisms sending $\hb$ to $0$:
\[
\xymatrix@C=2cm{
\bigl( V \o \tla \dh(\X)_\la \bigr)^G \ar[r]^-{\Phi^{V,\la}_{s}} \ar[d] & \bigl( V \o {}^{(s\la)} \hspace{-1pt} \dh(\X)_{s\la} \bigr)^G \ar[d] \\
\bigl( V \o \Gamma(\wfg, \cO_{\wfg}(\la)) \bigr)^G \ar[r]^-{\sigma^{V,\la}_s} & \bigl( V \o \Gamma(\wfg, \cO_{\wfg}(s\la)) \bigr)^G
}
\]
\end{cor}

\begin{proof}
By complete reducibility it is enough to prove the claim when $V=V^\nu$ for some $\nu \in \bX^+$. In this case it follows by comparing Lemma \ref{lem:Phi-rk1} and Lemma \ref{lem:sigma-rk1}.
\end{proof}

\subsection{Satake equivalence}

From now on we use the notation of Sections \ref{sec:Satake}--\ref{sec:applications}, assuming in addition that $\Gv$ has semisimple rank one. We denote by $\alpha$ the unique positive coroot of $\Gv$.
Note that $\alv$ and $\hb$ can be considered either as characters of $A=\Tv \times \C^\times$ or as elements of $\a^* = \t \oplus \C$.

To simplify the statements of the next results we introduce the following notation. For $\la \in \bX$ and $k \geq 1$ we define the $A$-module\index{Vlambdak@$V^{\lambda}_k$}%
\[
V^\la_k \ := \ \C_{-\alv-(\la(\alv)+1)\hb} \oplus \C_{-\alv-(\la(\alv)+2)\hb} \oplus \cdots \oplus \C_{-\alv-(\la(\alv)+k)\hb}.
\]
We also set $V^\la_0=\{0\}$. Note that $\dim(V^\la_k)=k$, and that there are natural inclusions $V_k^\la \subset V_{k+1}^\la$.

\begin{lem}
\label{lem:T-rk1}
Let $\la \in \bX$. 
\begin{enumerate}
\item
Assume $\la(\alv) \geq 0$.
For $\nu \in \bX^+$ we have isomorphisms of $A$-varieties
\[
\fT_\la \cap \Gr_\Gv^\nu \ \cong \ \begin{cases}
\{\bla\} & \text{if  $\nu=\la$;} \\
V^\la_k \smallsetminus V^\la_{k-1} & \text{if $\nu = \la + k \alpha$ for some $k \in \Z_{>0}$;} \\
\emptyset & \text{otherwise.}
\end{cases}
\]
\item
Assume $\la(\alv) < 0$.
For $\nu \in \bX^+$ we have isomorphisms of $A$-varieties
\[
\fT_\la \cap \Gr_\Gv^\nu \ \cong \ \begin{cases}
V^\la_{-\la(\alv)} & \text{if  $\nu=\la+(-\la(\alv)) \alpha$;} \\
V^\la_k \smallsetminus V^\la_{k-1} & \text{if $\nu = \la + k \alpha$ for some $k \in \Z_{>-\la(\alv)}$;} \\
\emptyset & \text{otherwise.}
\end{cases}
\]
\end{enumerate}
\end{lem}

\begin{proof}
Each connected component of $\Gr_\Gv$ is isomorphic to a connected component of $\Gr_{\mathrm{PGL}(2,\C)}$, and the action of $\ker(\alv)=Z(\Gv) \subset \Tv$ is trivial. Hence it is enough to prove the isomorphism when $\Gv=\mathrm{PGL}(2,\C)$. In this case we can identify $\bX$ with $\Z$ through $\mu \mapsto \mu(\alv)$, so that $\nu$ and $\la$ can be considered as integers.

\emph{First case: even weights}. Write $\la=2\ell$ with $\ell \in \Z$. Then $\bla$ is the class of the matrix
\[
\left(
\begin{array}{cc}
z^{\ell} & 0 \\
0 & z^{-\ell}
\end{array}
\right),
\]
and $\fT_{\lambda}$ is given by the classes of matrices of the form
\[
M(Q) = \left( \begin{array}{cc}
z^{\ell} & 0 \\
Q(z) & z^{-\ell}
\end{array} \right)
\]
where $Q(z) \in z^{-\ell-1} \C[z^{-1}]$. If $Q(z)=0$, then this point is in $\Gr^{\lambda}$. Otherwise, write $Q(z)=a z^{-m} + \cdots$, where $a \neq 0$, $m > \ell$, and ``$\cdots$'' means terms of degree between $-m+1$ and $-\ell-1$. Assume first that $m+\ell \geq 0$. (This condition is always satisfied when $\ell \geq 0$. Note also that it implies $m \geq 0$.) Then we have
\[
\left(
\begin{array}{cc}
z^{\ell} & 0 \\
Q(z) & z^{-\ell}
\end{array}
\right) \cdot
\left(
\begin{array}{cc}
z^{m-\ell} & R(z) \\
-z^{m} Q(z) & 0
\end{array}
\right) = 
\left(
\begin{array}{cc}
1 & z^{m+\ell}R(z) \\
0 & 1
\end{array}
\right) \cdot
\left(
\begin{array}{cc}
z^{m} & 0 \\
0 & z^{-m}
\end{array}
\right)
\]
where $R(z) \in \OO$ is the inverse to $z^{m}Q(z)$. This equality implies that $M(Q) \in \Gr^{2m}$.

If $\ell<0$, then we also have to consider the case $\ell < m < -\ell$. However, in this case $M(Q)$ is in $\Gv(\OO) \cdot \bla = \Gr_\Gv^{\la+(-\la(\alv)) \alpha}$.
This settles the first case.

\emph{Second case: odd weights}. Write $\la=2\ell+1$ with $\ell \in \Z$. Then $\bla$ is the class of the matrix
\[
\left(
\begin{array}{cc}
z^{\ell+1} & 0 \\
0 & z^{-\ell}
\end{array}
\right),
\]
and $\fT_{\lambda}$ is given by the classes of matrices of the form
\[
N(Q) = \left( \begin{array}{cc}
z^{\ell+1} & 0 \\
Q(z) & z^{-\ell}
\end{array} \right)
\]
where $Q(z) \in x^{-\ell-1} \C[z^{-1}]$. If $Q(z)=0$, then this point is in $\Gr^{\lambda}$. Otherwise, write as above $Q(z)=a z^{-m} + \cdots$, where $a \neq 0$ and $m > \ell$. We have the following equality:
\[
\left(
\begin{array}{cc}
z^{\ell+1} & 0 \\
Q(z) & z^{-\ell}
\end{array}
\right) \cdot
\left(
\begin{array}{cc}
z^{m-\ell} & R(z) \\
-z^{m}Q(z) & 0
\end{array}
\right) = 
\left(
\begin{array}{cc}
1 & z^{\ell+m+1}R(z) \\
0 & 1
\end{array}
\right) \cdot
\left(
\begin{array}{cc}
z^{m+1} & 0 \\
0 & z^{-m}
\end{array}
\right)
\]
where as above $R(z) \in \OO$ is the inverse to $z^{m}Q(z)$. This equality implies that $N(Q) \in \Gr^{2m+1}$ if $\ell+m+1 \geq 0$.

If $\ell<0$ and $\ell+m+1 < 0$, then $N(Q)$ is in $\Gv(\OO) \cdot \bla=\Gr^{\lambda+(-\la(\alv)) \alpha}$. This settles the second case, and finishes the proof.
\end{proof}

The following result is a direct consquence of Lemma \ref{lem:T-rk1}. It is stated without proof in \cite{bf, brf}.

\begin{cor}
\label{cor:intersections-rk1}
Let $\nu \in \bX^+$.
\begin{enumerate}
\item If $\lambda \in \bX$ is not of the form $\nu - k \alpha$ for some $k \in \{0, \cdots, \nu(\alv)\}$, then $\fT_\la \cap \overline{\Gr_\Gv^\nu}=\emptyset$.
\item If $k \in \{0, \cdots, \nu(\alv)\}$ and $\la=\nu-k\alpha$, then there exists an isomorphism of $A$-varieties 
\[
\fT_{\la} \cap \overline{\Gr_\Gv^\nu} \ \cong \ V^{\nu-k\alpha}_k
\]
sending $\bla$ to $0$.
\end{enumerate}
\end{cor}

By the constructions of \S\ref{ss:Satake} we have a dual group $G$ (which is also of semisimple rank $1$) and a maximal torus $T \subset G$ such that $\bX=X^*(T)$.
If $\nu \in \bX^+$, we set $\IC_\nu:=\IC \bigl( \overline{\Gr_\Gv^\nu}, \underline{\C}_{\Gr_\Gv^\nu} \bigr)$.\index{ICnu@$\IC_{\nu}$}%

\begin{cor}
\label{cor:kgeom-rk1}
For $\nu \in \bX^+$ and $k \in \{0, \cdots, \nu(\alv)\}$, the image of the morphism
\[
\kgeom_{\IC_\nu,\nu-k\alpha} : \coH^{\hdot}_A(i_{\nu-k\alpha}^! \IC_\nu) \to \bigl( \cS_\Gv(\IC_\nu) \bigr)_{\nu-k\alpha} \o \sh \cong \sh
\]
is generated by $\Bigl( \alv + \bigl( \nu(\alv)-k \bigr) \hbar \Bigr) \cdots \Bigl(\alv + \bigl( \nu(\alv) -2k+1 \bigr) \hbar \Bigr)$.
\end{cor}

\begin{proof}
It is well known that $\overline{\Gr_\Gv^\nu}$ is rationally smooth. This implies that
\[
\IC_\nu \cong \underline{\C}_{\overline{\Gr_\Gv^\nu}}[\nu(\alv)] \cong \underline{\mathbb{D}}_{\overline{\Gr_\Gv^\nu}}[-\nu(\alv)].
\]
We deduce isomorphisms
\[
\coH_A^\hdot(i_{\nu-k\alpha}^! \IC_\nu) \cong \coH^A_{\hdot-\nu(\alv)}(\{\bla\}), \qquad \coH_A^\hdot(t_{\nu-k\alpha}^! \IC_\nu) \cong \coH^A_{\hdot-\nu(\alv)+2k}(\fT_{\nu-k\alpha} \cap \overline{\Gr_\Gv^\nu}),
\]
and the morphism 
$
\coH_{A}^{\hdot}(i_{\nu-k\alpha}^! \IC_{\nu}) \to \coH_{A}^{\hdot}(t_{\nu-k\alpha}^! \IC_{\nu})
$
identifies with the morphism
\[
\coH^{A}_{\hdot-\nu(\alv)}(\{\bla\}) \to \coH^{A}_{\hdot-\nu(\alv)+2k}(\fT_{\nu-k\alpha} \cap \overline{\Gr_\Gv^{\nu}})
\]
given by proper push-forward in equivariant Borel--Moore homology. Hence we deduce the result from the description of $\fT_{\nu-k\alpha} \cap \overline{\Gr_\Gv^{\nu}}$ in Corollary \ref{cor:intersections-rk1} and the considerations on equivariant (co)homology in \S\ref{ss:equiv-cohomology}.
\end{proof}

\subsection{Proof of Theorem \ref{thm:loop-equivariant-version} for $\Gv$}
\label{ss:thm-cofibers-rk1}

By semisimplicity of the category $\Perv_{\Gv(\OO)}(\Gr)$,
it is enough to prove the theorem when $\cF=\IC_\nu$ for some $\nu \in \bX^+$. Then $V^\nu := \cS_\Gv(\IC_\nu)$ is the $G$-module with highest weight $\nu$. Comparing Corollary \ref{cor:image-algebra-sl2} and Corollary \ref{cor:kgeom-rk1} we observe that indeed the images of $\kgeom_{\IC_\nu,\la}$ and $\kalg_{V^\nu,\la}$ coincide, which implies the existence of the isomorphism $\zeta_{\IC_\nu, \la}$.

\subsection{Root vector and Mirkovi{\'c}--Vilonen basis}

We denote by $e \in \g_{\alpha}$ the vector constructed in \S\ref{ss:root-vectors}.

Let $\nu \in \bX^+$ and $k \in \{0, \cdots, \nu(\alv)\}$. Let $V^\nu:=\cS_\Gv(\IC_\nu)$, a simple $G$-module with highest weight $\nu$. As in the proof of Corollary \ref{cor:kgeom-rk1}, we have canonical isomorphisms
\[
V^\nu_{\nu-k\alpha} \overset{\eqref{eqn:weight-space-Satake}}{\cong} \coH^{\nu(\alv)-2k}(t_{\nu-k\alpha}^! \IC_\nu) \cong \coH^A_{0}(\fT_{\nu-k\alpha} \cap \overline{\Gr_\Gv^\nu}).
\]
The right-hand side is $1$-dimensional, and has a canonical generator, namely the fundamental class $[\fT_{\nu-k\alpha} \cap \overline{\Gr_\Gv^\nu}]$. We denote by $v^\nu_{\nu-k\alpha} \in V^\nu_{\nu-k\alpha}$ the vector corresponding to this generator.

\begin{rem}
Beware that the basis we consider here is not the same as the one used in \cite[\S 5.2]{bf} or in \cite[\S 2.2]{brf}, but rather the basis which is dual in the sense of Poincar{\'e} duality.
\end{rem}

The following lemma is a special case of \cite[Th{\'e}or{\`e}me 2]{baumann}.

\begin{lem}
For any $\nu \in \bX^+$ and $k \in \{0, \cdots, \nu(\alv)\}$ we have
\[
e \cdot v^{\nu}_{\nu - k\alpha} = k v_{\nu-(k-1)\alpha}^{\nu}.
\]
\end{lem}

It follows from this lemma that, for the choice of root vector $e \in \g_{\alpha}$, the basis just constructed satisfies conditions \eqref{eqn:modules-SL2}. Hence we can use the results and notation of \S\S\ref{ss:Verma-rk1}--\ref{ss:sigma-rk1} for this choice.

We can now give a more precise version of Theorem \ref{thm:loop-equivariant-version} for $\Gv$. As in the proof of Corollary \ref{cor:kgeom-rk1} there exists a canonical isomorphism
\[
\coH_A^\hdot(i_{\nu-k\alpha}^! \IC_\nu) \ \cong \ \coH^A_{\hdot-\nu(\alv)}(\{\bla\}).
\]
The right-hand side has a canonical generator, namely the unique element\index{cnulambda@$c^{\nu}_{\lambda}$}%
\[
c^\nu_{\nu-k\alpha} \in \coH_A^{\nu(\alv)}(i_{\nu-k\alpha}^! \IC_\nu)
\]
whose image in $\coH^{\nu(\alv)}(i_{\nu-k\alpha}^! \IC_\nu) \cong \coH_0(\{\bla\})$ is the fundamental class $[\{\bla\}]$. Then (see \S\ref{ss:equiv-cohomology}) we have
\beq{eqn:zeta-rk1}
\zeta_{\IC_\nu, \nu-k\alpha}(x^\nu_{\nu-k\alpha}) = c^\nu_{\nu-k\alpha}.
\eeq

\subsection{Proof of Theorem \ref{thm:W-symmetry} for $\Gv$}
\label{ss:W-symmetry-rk1}

We have already proved in Lemma \ref{lem:comparison-Phi-Theta} that the operators $\Phi$ and $\Theta$ match under the natural isomorphisms. Hence we only have to compare them with operators $\Xi$. Moreover, it is enough to prove the theorem in the case $\cF=\IC_\nu$ for some $\nu \in \bX^+$. By construction (see \S\ref{ss:W-symmetries}) we have $\Xi^{\IC_\nu,\nu-k\alpha}_s(c^\nu_{\nu-k\alpha}) = c^\nu_{\nu-(\nu(\alv)-k)\alpha}$. Hence the claim follows from \eqref{eqn:zeta-rk1} and Lemma \ref{lem:Phi-rk1} or Lemma \ref{lem:Theta-rk1}.

\section{Fourier transform for differential operators}
\label{sec:appendix-Fourier}

In this appendix we briefly explain how to adapt some classical constructions of Fourier transform for $\dd$-modules (see e.g.~\cite{bmv}) to the asymptotic setting.

\subsection{Partial Fourier transform}
\label{ss:partial-Fourier-definition}




Let $X$ be a smooth complex algebraic variety, and let $p : E \to X$ be
a  rank $r$ algebraic vector bundle. Let
${\check p} : E^* \to X$ be the dual vector bundle,
and let ${\mathfrak E}:=E\times_X E^*$ be the total space of the direct
sum $E\oplus E^*$, a vector bundle on $X$ of rank $2r$.
The canonical pairing of $E$ and $E^*$ gives a regular
function ${\mathbf f}: E\times_X E^*\to\C$. We define a 
connection $\nabla: \oo_{\mathfrak E}\to \Omega^1_{\mathfrak E}$ by
$\nabla=d-d{\mathbf f}$. This connection is flat and
makes ${\mathcal P}=(\oo_{\mathfrak E},\nabla)$ a holonomic left
$\dd_{\mathfrak E}$-module, with {\em irregular singularities}.
Explicitly, we have ${\mathcal P}=\dd_{\mathfrak E}/{\mathcal J}$ where
${\mathcal J}\sset \dd_{\mathfrak E}$ is the left ideal generated
by the elements $\xi-\xi({\mathbf f})$ for all
vector fields $\xi$ on ${\mathfrak E}$. (Note that this ideal is the annihilator of the function $\exp \circ \mathbf{f}$ on ${\mathfrak E}$, the classical kernel for the Fourier transform.)

Now, let $\epsilon: {\mathfrak E}=E\times_X E^*\into E\times E^*$ be the natural
closed embedding and let $\epsilon_*{\mathcal P}$ be a direct
image of the $\dd$-module ${\mathcal P}$. Thus, $\epsilon_*{\mathcal P}$
is a holonomic left $\dd_{E\times E^*}$-module supported
on the subvariety ${\mathfrak E}$. Write $\mathcal{K}_Y$ for 
 the canonical bundle
on a smooth variety $Y$ and
let 
\[
{\mathcal Q}:=(\oo_E\boxtimes{\check p}^*\mathcal{K}_X)\o_{\oo_{E \times E^*}}
\epsilon_*{\mathcal P}.
\]
The sheaf ${\mathcal Q}$
has the structure of a module of the algebra 
\[
\dd_E\boxtimes
\bigl( {\check p}^*\mathcal{K}_X\o_{\oo_{E^*}} \dd_{E^*}\o_{\oo_{E^*}} {\check p}^*\mathcal{K}_X\inv \bigr),
\]
and this module  has a canonical section $1_{\mathcal Q}\in{\mathcal Q}$ that corresponds
to the section $1\, \mathrm{mod}\, {\mathcal J}\in \dd_{\mathfrak E}/{\mathcal J}$.
Furthermore,
a local computation shows that 
${\mathcal Q}$ is a rank one free module (with generator $1_{\mathcal Q}$) over the ring
$\dd_E\boxtimes 1$,
as well as over the ring $1\boxtimes
\bigl( {\check p}^*\mathcal{K}_X\o\dd_{E^*}\o{\check p}^*\mathcal{K}_X\inv \bigr)$.

Let $\overline{p} : {\mathfrak E} \to X$ be the natural projection. Then $\overline{p}_* {\mathcal{Q}}$ is a rank one free module both over $p_* \dd_E$ and over $\mathcal{K}_X\o_{\oo_X} {\check p}_* \dd_{E^*}\o_{\oo_X} \mathcal{K}_X\inv$, with a canonical generator $1_{\overline{p}_* {\mathcal{Q}}}$.
Therefore, there is a uniquely determined morphism
$F: p_* \dd_E \to \mathcal{K}_X\o {\check p}_* \dd_{E^*}\o \mathcal{K}_X\inv$
such that one has $u \cdot 1_{\overline{p}_* \mathcal Q}=F(u) \cdot 1_{\overline{p}_* \mathcal Q}$.
It is immediate to check that  this morphism
is an anti-isomorphism of rings,
i.e.~it induces a ring isomomorphism
\beq{Fo}
p_* \dd_E \simto (\mathcal{K}_X\o_{\oo_X} {\check p}_* \dd_{E^*}\o_{\oo_X} \mathcal{K}_X\inv)^{\mathrm{op}}.
\eeq
On the other hand, we have
$\mathcal{K}_{E^*}={\check p}^*(\det(E)\o_{\oo_X} \mathcal{K}_X)$, where
 $\det(E)$ denotes the sheaf of sections
of the line bundle $\wedge^r E$ on $X$.
Hence, using the well-known isomorphism
$\dd_{E^*}^{\mathrm{op}}\cong \mathcal{K}_{E^*}\o_{\oo_{E^*}} \dd_{E^*}\o_{\oo_{E^*}} \mathcal{K}_{E^*}\inv$, we compute:
\begin{align*}
\bigl( \mathcal{K}_X\o_{\oo_X} {\check p}_* \dd_{E^*}\o_{\oo_X} \mathcal{K}_X\inv \bigr)^{\mathrm{op}} &
\cong\ \mathcal{K}_X\inv\o_{\oo_X} {\check p}_*(\dd_{E^*}^{\mathrm{op}})\o_{\oo_X}
 \mathcal{K}_X\\
&\cong \
\mathcal{K}_X\inv\o_{\oo_X} {\check p}_* \bigl( \mathcal{K}_{E^*}\o_{\oo_{E^*}} \dd_{E^*}\o_{\oo_{E^*}} \mathcal{K}_{E^*}\inv \bigr)
\o_{\oo_X} \mathcal{K}_X\\
& \cong\
\det(E)\o_{\oo_X} {\check p}_* \dd_{E^*}\o_{\oo_X} \det(E)\inv.
\end{align*}

Thus, from \eqref{Fo} we deduce a canonical isomorphism of sheaves of algebras on $X$, called {\em Fourier isomorphism}:
\[
{\mathbf F}:\ p_*\dd_E\ \simto \ \det(E)\o_{\oo_X} ({\check p}_* \dd_{E^*})
\o_{\oo_X} 
\det(E)^{-1}.
\]

Now, in a Rees algebra setting,
we define ${\mathcal J}_\hbar$ to be the left ideal of $\dd_{\hb,{\mathfrak E}}$
generated by the elements $\xi- \xi({\mathbf f})$ for all
vector fields $\xi$ on ${\mathfrak E}$, viewed as degree 2 homogeneous elements
of the graded algebra $\dd_{\hb,{\mathfrak E}}$. (Here we use the notational conventions of \S\ref{subsec2}. Note that the ideal ${\mathcal J}_\hb$ is the annihilator of the ``function'' $\exp(\frac{1}{\hb} \mathbf{f})$, considered as an element of some completion of $\oo_X[\hb,\hb^{-1}]$.)
We put
${\mathcal P}_\hb:=\dd_{\hb,{\mathfrak E}}/{\mathcal J}_\hb$,
 a $\dd_{\hb,{\mathfrak E}}$-module. 
Note that the ideal ${\mathcal J}_\hb$ is not homogeneous
so the module ${\mathcal P}_\hb$ has no natural grading.

Under the specialization $\hb=0$, we have
$\dd_{\hb,{\mathfrak E}}/(\hb)=(p_{{\mathfrak E}})_*\oo_{T^*{\mathfrak E}}$,
where $p_{{\mathfrak E}}: T^*{\mathfrak E}\to{\mathfrak E}$ is the cotangent bundle.
The differential of the function ${\mathbf f}$ gives
a section $d{\mathbf f}: {\mathfrak E}\to  T^*{\mathfrak E}$. The image of this
section is a smooth closed lagrangian subvariety
$\Lambda\sset T^*{\mathfrak E}$,
so the sheaf $(p_{{\mathfrak E}})_*\oo_\Lambda$ has a natural structure
of a $(p_{{\mathfrak E}})_*\oo_{T^*{\mathfrak E}}$-module.
Then, it  follows from definitions that the projection
$\dd_{\hb,{\mathfrak E}}\onto {\mathcal P}_\hb$ induces
an isomorphism of $(p_{{\mathfrak E}})_*\oo_{T^*{\mathfrak E}}$-modules:
\[
{\mathcal P}_\hb/(\hb)\ \cong\ (p_{{\mathfrak E}})_*\oo_\Lambda.
\]
We also consider the  cotangent bundle
$\mathbf{q}: T^*(E\times E^*)\to E\times E^*$ and let
$T^*(E\times E^*)|_{\mathfrak E}$ denote the total space of
the restriction of the  cotangent bundle
to ${\mathfrak E}\sset E\times E^*$, a closed subvariety. 
One has a natural diagram
$$
\xymatrix{
T^*E \times T^*(E^*)\ \ar@{=}[r] &\
T^*(E\times E^*)\ && \ T^*(E\times E^*)|_{\mathfrak E}\
\ar@{_{(}->}[ll]_<>(0.5){\varepsilon}\ar@{->>}[rr]^<>(0.5){\mathrm{pr}} &&
\ T^*{\mathfrak E}
}
$$
Here, the isomorphism
on the left involves a sign and the map $\mathrm{pr}$ on
the right is a smooth morphism.

The following
result is easily verified by a local computation.
\begin{lem}\label{ZL}The variety $Z:=\varepsilon(\mathrm{pr}\inv(\Lambda))$
is a smooth Lagrangian subvariety of $T^*(E \times E^*)$.
Furthermore, this subvariety is the graph of
an isomorphism $T^*E \iso T^*(E^*)$,
of algebraic varieties over $X$.
\end{lem}

To proceed further we observe that, for any  smooth variety $Y$,
the sheaf   $\mathcal{K}_Y[\hb]$ has a canonical
right
$\dd_{\hb,Y}$-action such that  a vector
field $\xi\in \mathscr{T}_Y$ acts on  $\mathcal{K}_Y$ by 
$\beta\mto -\hb\cdot L_\xi\beta$,
where  $L_\xi$ stands for the
Lie derivative. We write  $\mathcal{K}_Y^\hb$ for the resulting
right
$\dd_{\hb,Y}$-module. Then, one has a canonical
isomorphism 
\[
\dd_{\hb,Y}^{op}\ \cong\
\mathcal{K}_Y^\hb\o_{\oo_Y[\hb]}\dd_{\hb,Y}\o_{\oo_Y[\hb]}
(\mathcal{K}_Y^\hb)\inv.
\]
Note that  this isomorphism specializes at $\hb=0$ to the
identity map $(p_{{\mathfrak E}})_*\oo_{T^*Y}\to (p_{{\mathfrak E}})_*\oo_{T^*Y}$.

Next, mimicing the corresponding constructions for
$\dd$-modules, one can define a direct image
$\epsilon_*{\mathcal P}_\hb$, a left $\dd_{\hb,E\times E^*}$-module.
Further, we put
${\mathcal Q}_\hb:=(\oo_E\boxtimes{\check p}^*\mathcal{K}^\hb_X)\o_{\oo_{E \times E^*}[\hb]}
\epsilon_*{\mathcal P}_\hb$.
Then, one checks that there is a natural isomorphism 
\[{\mathcal Q}_\hb/(\hb)\ \cong \ \mathbf{q}_*\oo_Z,\]
of $\mathbf{q}_*\oo_{T^*(E\times
  E^*)}$-modules.
Furthermore, repeating earlier constructions,
one obtains
 a canonical isomorphism
\begin{equation}
\label{eqn:isom-Fourier}
{\mathbf F}_\hb:\
p_* \dd_{\hb,E} \ \simto \ \det(E)[\hb] \o_{\oo_X[\hb]} ({\check p}_* \dd_{\hb,E^*}) \o_{\oo_X[\hb]} \det(E)^{-1}[\hb].
\end{equation}
This is an  isomorphism of sheaves of $\C[\hb]$-algebras on $X$.
This   isomorphism 
does not respect the natural gradings on each side in
\eqref{eqn:isom-Fourier}
 unless $r=0$ and it specializes, at $\hb=0$, to the
isomorphism $p_* \oo_{T^*E}\iso {\check p}_*\oo_{T^*(E^*)}$
that results from Lemma \ref{ZL}.

The above isomorphism 
can be described locally as follows. Let $U \subset X$ be an open subvariety over which $E$ is trivializable, and let us choose an isomorphism of vector bundles $E_{|U} \cong \C^r \times U$. Then for $i=1, \cdots, r$ we have a function $x_i$ on $E_{|U}$ given by the projection on the $i$-th copy of $\C$, and the corresponding vector fields $\partial_{x_i}$, so that we have an isomorphism of sheaves of $\C[\hb]$-algebras
\begin{equation}
\label{eqn:DEU}
\bigl( p_* \dd_{\hb,E} \bigr)_{|U} \ \cong \ \dd_{\hb,U} \o_{\C[\hb]} \bigl( \C\langle x_i, \partial_{x_i} \rangle / [\partial_{x_i}, x_i]=\hb \bigr).
\end{equation}
(Here $i$ runs over $\{1, \cdots, r\}$.)
Our isomorphism $E_{|U} \cong \C^r \times U$ also defines a canonical section $\tau$ of $\det(E)$ over $U$, the dual section $\tau^\vee$ of $\det(E)^{-1}$, and an isomorphism $(E^*)_{|U} \cong \C^r \times U$, so that we obtain functions $\xi_1, \cdots, \xi_r$ and vector fields $\partial_{\xi_1}, \cdots, \partial_{\xi_r}$ on $(E^*)_{|U}$, and an isomorphism
\begin{equation}
\label{eqn:DE*U}
\bigl( p_* \dd_{\hb,E^*} \bigr)_{|U} \ \cong \ \dd_{\hb,U} \o_{\C[\hb]} \bigl( \C\langle \xi_i, \partial_{\xi_i} \rangle / [\partial_{\xi_i}, \xi_i]=\hb \bigr).
\end{equation}
Then using isomorphisms \eqref{eqn:DEU} and \eqref{eqn:DE*U}, the restriction of isomorphism \eqref{eqn:isom-Fourier} to $U$ can be described as follows: it sends any $P \in \dd_{\hb,U}$ to $\tau \o P \o \tau^\vee$, $x_i$ to $\tau \o -\partial_{\xi_i} \o \tau^\vee$, and $\partial_{x_i}$ to $\tau \o \xi_i \o \tau^\vee$.

\subsection{``Symplectic'' partial Fourier transform}
\label{ss:symplectic-Fourier}

Now we assume that $E$ is a \emph{symplectic} vector bundle with symplectic form $\omega$ over a smooth complex algebraic variety $X$. Then we have an isomorphism of vector bundles
\[
E \simto E^*, \qquad v \mapsto \omega(v,-)
\]
over $X$, hence an induced isomorphism $p_* \dd_{\hb,E} \cong {\check p}_* \dd_{\hb,E^*}$.
Moreover, $\omega$ defines a trivialization of $\det(E)$. Hence isomorphism \eqref{eqn:isom-Fourier} provides an automorphism of $p_* \dd_{\hb,E}$. We denote by
\[
\mathbf{F}_E : \dh(E) \simto \dh(E)
\]
the induced automorphism. One can easily check that $\mathbf{F}_E$ is equivariant under the natural action of the group of symplectic automorphisms of $E$, and that we have $\mathbf{F}_E \circ \mathbf{F}_E = \id_{\dh(E)}$.

\begin{ex}
\label{ex:Fourier-SL2}
If $X=\mathrm{pt}$, then $E$ is simply a symplectic vector space. For instance, assume that $E=\C^2=\C v_1 \oplus \C v_2$, equipped with the symplectic form such that $\omega(v_1,v_2)=1$. Let $(\eta_1,\eta_2)$ be the basis of $E^*$ dual to $(v_1,v_2)$.
Then $\dh(E)$ is generated by $\eta_1,\eta_2$ (considered as functions on $E$) and $v_1,v_2$ (considered as vector fields on $E$), and $\mathbf{F}_E$ is defined by
\[
\eta_1 \mapsto v_{2}, \quad \eta_{2} \mapsto -v_1, \quad v_1 \mapsto -\eta_{2}, \quad v_{2} \mapsto \eta_1.
\]
\end{ex}

The following result (which can easily be checked using local trivializations) is used in \S\ref{ss:proof-restriction-DX}.

\begin{lem}
\label{lem:twist-Fourier}
Let $f \in \C[X]$ be an invertible function, which we consider as a function on $E$ via the projection $E \to X$. Then the automorphism of $\dh(E)$ given by $D \mapsto f^{-1} \cdot D \cdot f$ commutes with $\mathbf{F}_E$.
\end{lem}

\printindex

\bigskip


\end{document}